\documentclass{article}%
\usepackage{amsmath}
\usepackage{amsfonts}
\usepackage{amssymb}
\usepackage{graphicx}%
\newtheorem{theorem}{Theorem}
\newtheorem{acknowledgement}[theorem]{Acknowledgement}

\newtheorem{claim}[theorem]{Claim}

\newtheorem{corollary}[theorem]{Corollary}

\newtheorem{definition}[theorem]{Definition}

\newtheorem{lemma}[theorem]{Lemma}

\newtheorem{problem}[theorem]{Problem}
\newtheorem{proposition}[theorem]{Proposition}

\newenvironment{proof}[1][Proof]{\noindent\textbf{#1.} }{\ \rule{0.5em}{0.5em}}
\begin{document}

\date{}
\title{The ultrafilter and almost disjointness numbers}
\author{Osvaldo Guzm\'{a}n
\and Damjan Kalajdzievski \thanks{The first author was partially supported by a CONACyT grant A1-S-16164 and PAPIIT grant IN104220. The second author was supported by York University. \newline%
\textit{keywords: }MAD families, ultrafilters, almost disjointness number,
P-points ultrafilter number. \newline\textit{AMS classification: }03E17,
03E35, 03E05}}
\maketitle

\begin{abstract}
We prove that every \textsf{MAD }family can be destroyed by a proper forcing
that preserves $P$-points. With this result, we prove that it is consistent
that $\omega_{1}=\mathfrak{u}<\mathfrak{a,}$ solving a nearly 20 year old
problem of Shelah, and a problem of Brendle. We will also present a simple
proof of a result of Blass and Shelah that the inequality $\mathfrak{u<s}$ is consistent.

\end{abstract}

\section{Introduction}

Ultrafilters and \textsf{MAD} families\footnote{The reader may consult the
next section for the definitions of the undefined concepts used in the
introduction.} play a fundamental role on infinite combinatorics, set
theoretic topology and other branches of mathematics. For this reason, it is
interesting to study the relationship between these two objects. In this paper,
we will focus on the cardinal invariants associated with each of them. The
\emph{ultrafilter number }$\mathfrak{u}$ is defined as the smallest size of a
base of an ultrafilter, and the \emph{almost disjointness number
}$\mathfrak{a}$ is the smallest size of a \textsf{MAD} family. The consistency
of the inequality $\mathfrak{a<u}$ is well known and easy to prove, in fact,
it holds in the Cohen, random, and Silver models, among many others. Proving
the consistency of the inequality $\mathfrak{u<a}$ is much harder and used to
be an open problem for a long time; In fact, it follows by the theorems of
Hru\v{s}\'{a}k, Moore and D\v{z}amonja that the inequality $\mathfrak{u<a}$
can not be obtained by using countable support iteration of proper Borel
partial orders (see theorem 6.6 and theorem 7.2 of \cite{ParametrizedDiamonds}%
). The consistency of $\mathfrak{u<a}$ was finally established by Shelah in
\cite{ShelahTemplates} (see also \cite{MADfamiliesandUltrafilters}) where he
proved the following theorem:

\begin{theorem}
[Shelah]Let $V$ be a model of \textsf{GCH, }$\kappa$ a measurable cardinal and
$\mu,\lambda$ two regular cardinals such that $\kappa<\mu<\lambda.$ There is a
\textsf{c.c.c. }forcing extension of $V$ that satisfies $\mu=\mathfrak{b=d=u}$
and $\lambda=\mathfrak{a=c}.$ In particular, \emph{CON}(\textsf{ZFC }$+$
\textquotedblleft there is a measurable cardinal\textquotedblright) implies
\emph{CON}(\textsf{ZFC }$+$ \textquotedblleft$\mathfrak{u<a}$%
\textquotedblright).
\end{theorem}

\qquad\ \ \ \ 

This theorem was one of the first results proved using \textquotedblleft
template iterations\textquotedblright, which is a very powerful method that
has been very useful and has continued applications to this day (see for
example, \cite{BrendleTemplates}, \cite{MADfamiliesandUltrafilters},
\cite{amaybeSingular},
\cite{VeraDiego},
\cite{VeraTornquist},
\cite{TemplateIterationsnonDefinable}). In spite of the beauty of this result,
it leaves open the following questions:

\begin{problem}
[Shelah \cite{Shelah666}]Does \emph{CON}(\textsf{ZFC}) imply \emph{CON}%
(\textsf{ZFC }$+$ \textquotedblleft$\mathfrak{u<a}$\textquotedblright)?
\end{problem}

\begin{problem}
[Brendle \cite{MADfamiliesandUltrafilters}]Is it consistent that $\omega
_{1}=\mathfrak{u<a}$?
\end{problem}

\qquad\ \ \ 

In this paper, we will provide a positive answer to both questions, by proving
(without appealing to large cardinals) that every \textsf{MAD }family can be
destroyed by a proper forcing that preserves $P$-points. We will also present
an alternative proof of the consistency of $\mathfrak{u<s},$ which was proved
first by Blass and Shelah in \cite{BlassShelah} (see also \cite{Barty}).

\qquad\ \ \ 

Our motivation comes from the theorems of Shelah that establishes that the
statements \textquotedblleft$\omega_{1}=\mathfrak{b=a}<$ $\mathfrak{s}%
=\omega_{2}$\textquotedblright\ and \textquotedblleft$\omega_{1}%
=\mathfrak{b<a}=$ $\mathfrak{s}=\omega_{2}$\textquotedblright\ are consistent
(see \cite{OnCardinalInvariantsoftheContinuum}, or \cite{ProperandImproper}). After these impressive results, different
models of $\omega_{1}=\mathfrak{b<a}=\omega_{2}$ had been constructed (see
for example \cite{TotallyMAD},
\cite{MadFamiliesSplittingFamiliesandLargeContinuum}, \cite{VeraSteprans} and
\cite{MobandMAD}). In every case, the forcings used add Cohen reals, so no
ultrafilter is preserved.

\qquad\ \ \ 

In order to construct the models of $\mathfrak{b<s}$ and $\mathfrak{b<a},$
Shelah used a creature forcing (see \cite{OnCardinalInvariantsoftheContinuum}, \cite{ProperandImproper} and
\cite{AbrahamHandbook}). In \cite{BrendleRaghavan} Brendle and Raghavan found
a simpler representation of Shelah forcing as a two step iteration, which we
will briefly describe (more details of the forcing will be studied in the next section).

\qquad\ \ \ 

The most natural way to increase the splitting number is to \emph{diagonalize
}an ultrafilter. In order to build a model of $\mathfrak{b<s},$ it is enough
to construct (or force) an ultrafilter that can be diagonalized without adding
dominating reals (even in the iteration). Denote by $\mathbb{F}_{\sigma}$ the
set of all $F_{\sigma}$-filters\footnote{We view filters as subspaces of
$2^{\omega},$ the notion of Borel or $F_{\sigma}$ is taken using the usual
topology on $2^{\omega}$.} on $\omega$. If $\mathcal{F},\mathcal{G}%
\in\mathbb{F}_{\sigma}$ we define $\mathcal{F}\leq\mathcal{G}$ if
$\mathcal{G\subseteq F}.$ It is not hard to see that $\mathbb{F}_{\sigma}$
naturally adds an ultrafilter $\mathcal{U}_{gen}.$ The reader wishing to learn
more about $\mathbb{F}_{\sigma}$ may consult
\cite{ForcingwithFiltersandCompleteCombinatorics} and
\cite{ForcingwithF_sigmaandwithSummableFilters}. It turns out that the forcing
of Shelah is equivalent to the two step iteration $\mathbb{F}_{\sigma}%
\ast\mathbb{M(}\mathcal{\dot{U}}_{gen}),$ where $\mathbb{M(}\mathcal{U}%
_{gen})$ is the Mathias forcing relative to $\mathcal{U}_{gen}$ (see
\cite{BrendleRaghavan}). It can be proved that $\mathbb{M(}\mathcal{U}_{gen})$
does not add dominating reals, even when iterated (see\ \cite{BrendleRaghavan}
or \cite{CanjarFiltersII}). \ \ 

\qquad\qquad\qquad

The method to build a model of $\mathfrak{b<a}$ is similar: Given a
\textsf{MAD} family $\mathcal{A},$ denote by $\mathbb{F}_{\sigma}\left(
\mathcal{A}\right)  $ the set of all $F_{\sigma}$-filters $\mathcal{F}$ such
that $\mathcal{F\cap I}\left(  \mathcal{A}\right)  =\emptyset$ (where
$\mathcal{I}\left(  \mathcal{A}\right)  $ is the ideal generated by
$\mathcal{A}$). Once again, we order $\mathbb{F}_{\sigma}\left(
\mathcal{A}\right)  $ with inclusion. It is easy to see that $\mathbb{F}%
_{\sigma}\left(  \mathcal{A}\right)  $ naturally adds an ultrafilter
$\mathcal{U}_{gen}\left(  \mathcal{A}\right)  ,$ furthermore, diagonalizing
$\mathcal{\dot{U}}_{gen}\left(  \mathcal{A}\right)  $ destroys the maximality
of $\mathcal{A}.$ In this case, it can be proved that $\mathbb{C}_{\omega_{1}%
}\ast\mathbb{F}_{\sigma}\ast\mathbb{M(}\mathcal{\dot{U}}_{gen}\left(
\mathcal{A}\right)  )$ does not add dominating reals, even when iterated
($\mathbb{C}_{\omega_{1}}$ denotes the forcing for adding $\omega_{1}$-Cohen
reals), see \cite{BrendleRaghavan} or \cite{CanjarFiltersII}. We want to point
out that both the forcing of Shelah (\cite{OnCardinalInvariantsoftheContinuum},\cite{ProperandImproper}) and the
forcing used by Brendle in \cite{MobandMAD} require adding Cohen reals in an
explicit way. Our work shows that adding the Cohen reals was in fact not needed.

\qquad\ \ \ \qquad\ \ \ 

We take a similar approach in order to build our model of $\mathfrak{u}$
$\mathfrak{<s}$. We will first force with $\mathbb{F}_{\sigma}$ and then we
will diagonalize $\mathcal{\dot{U}}_{gen}.$ The difference, is that instead of
using Mathias forcing, we will use a variant of Miller forcing. The same
technique will be used to build the model of $\omega_{1}=\mathfrak{u}$
$\mathfrak{<a}.$ Given a \textsf{MAD} family $\mathcal{A},$ we will force with
$\mathbb{F}_{\sigma}\left(  \mathcal{A}\right)  $ and then diagonalize in the
same way as before. In both cases, we will prove that the forcings preserve
all $P$-points of the ground model.

\qquad\ \ \ \ \qquad\ \ \ 

There is a huge body of work regarding the cardinal invariants $\mathfrak{b,s},\mathfrak{d,u}$, and $\mathfrak{a}$;
To provide some more historical context to this work, as was mentioned before, the story began when Shelah
(\cite{ProperandImproper}) constructed models of $\omega_{1}=\mathfrak{b<s}$
and $\omega_{1}=\mathfrak{b<a=s}.$ To achieve this, Shelah used a countable support iteration
of creature forcings. In \cite{MobandMAD}, Brendle used \textsf{c.c.c.}
forcings for constructing models of $\kappa=\mathfrak{b<a=\kappa}^{+}$, where
$\kappa$ is any uncountable regular cardinal, and later in \cite{VeraSteprans} Fischer
and Stepr\={a}ns constructed models of $\kappa=\mathfrak{b<s=\kappa}^{+}$,
where again $\kappa$ can be any uncountable regular cardinal. In
\cite{MadFamiliesSplittingFamiliesandLargeContinuum} Brendle and Fischer used
matrix iterations to prove that for any regular cardinals $\kappa<\lambda,$ it
is consistent that $\kappa=\mathfrak{b=a<s=\lambda}$, and if $\kappa\ $is
bigger than a measurable cardinal, then it is consistent that$\ \kappa
=\mathfrak{b<a=s=\lambda}$. The consistency of
$\omega_{1}<\mathfrak{d<a}$ and $\omega_{1}<\mathfrak{u<a}$ was obtained by
Shelah in \cite{ShelahTemplates} where he developed the technique of forcing
along a template (see also \cite{BrendleTemplates} and
\cite{MADfamiliesandUltrafilters}). 
Finally, in \cite{VeraDiego} Fischer and Mej\'{\i}a
proved that it is consistent that $\omega_{1}<\mathfrak{s<b<a}$ (see also
\cite{MatrixIterationsandCichonDiagram} and \cite{VeraDiegoDiana}).

\qquad\ \ \ \ 

There are still many interesting open questions remaining:

\begin{problem}
[Roitman]\textit{Does }$\mathfrak{d}=\omega_{1}$ \textit{imply that}
$\mathfrak{a}=\omega_{1}?$
\end{problem}

\begin{problem}
[Brendle and Raghavan]\textit{Does }$\mathfrak{b=s}=\omega_{1}$ \textit{imply
that} $\mathfrak{a}=\omega_{1}?$
\end{problem}

\qquad\ \ \ 

Note that a positive solution to the question of Brendle and Raghavan would
provide a positive solution to the problem of Roitman.

\section{Preliminaries and notation\qquad\ \ \ }

Let $f,g\in\omega^{\omega}$, define $f\leq g$ if and only if $f\left(
n\right)  \leq g\left(  n\right)  $ for every $n\in\omega$ and $f\leq^{\ast}g$
if and only if $f\left(  n\right)  \leq g\left(  n\right)  $ holds for all
$n\in\omega$ except finitely many. We say a family $\mathcal{B}\subseteq
\omega^{\omega}$ is \emph{unbounded }if $\mathcal{B}$ is unbounded with
respect to $\leq^{\ast}.$ A family $\mathcal{D}\subseteq\omega^{\omega}$ is a
\emph{dominating family }if for every $f\in\omega^{\omega},$ there is
$g\in\mathcal{D}$ such that $f\leq^{\ast}g.$ The \emph{bounding number
}$\mathfrak{b}$ is the size of the smallest unbounded family and the
\emph{dominating number }$\mathfrak{d}$ is the smallest size of a dominating
family. We say that $S$ \emph{splits }$X$ if $S\cap X$ and $X\setminus S$ are
both infinite. A family $\mathcal{S\subseteq}$ $\left[  \omega\right]
^{\omega}$ is a \emph{splitting family }if for every $X\in\left[
\omega\right]  ^{\omega}$ there is $S\in\mathcal{S}$ such that $S$ splits $X.$
The \emph{splitting number }$\mathfrak{s}$ is the smallest size of a splitting
family. A family $\mathcal{A}\subseteq\left[  \omega\right]  ^{\omega}$ is
\emph{almost disjoint (\textsf{AD}) }if the intersection of any two different
elements of $\mathcal{A}$ is finite, a \textsf{MAD }\emph{family}%
\textsf{\emph{ }}is a\textsf{ }maximal almost disjoint family. By
\textsf{cov}$\left(  \mathcal{M}\right)  $ we denote the smallest size of a
family of meager sets that covers the Baire space. The reader may consult the
surveys \cite{HandbookBlass} for the main properties of the cardinal
invariants used in this paper, \cite{AlmostDisjointFamiliesandTopology} to
learn more about almost disjoint families and
\cite{CombinatoricsofFiltersandIdeals} for a survey on filters and ideals.

\qquad\ \ \ \qquad\ \ \ \qquad\ \ \qquad\ \ 

In this paper, a \emph{tree }is a set of sequences closed under taking
restrictions (i.e. $p$ is a tree if whenever $s\in p$ and $n<\left\vert
s\right\vert $ then $s\upharpoonright n\in p).$ If $s,t\in\omega^{<\omega},$
by $s^{\frown}t$ we denote the concatenation of $s$ and $t$. If $n\in\omega,$
we will often write $s^{\frown}n$ instead of $s^{\frown}\left\langle
n\right\rangle .$ In this paper, we will say that $s\in p$ is a
\emph{splitting node }if $suc_{p}\left(  s\right)  =\left\{  n\mid s^{\frown
}n\in p\right\}  $ is infinite. We say that a splitting node $s\in p$ is the
\emph{stem of }$p$ (denoted by $stem\left(  p\right)  $ in case it exists) if
every predecessor of $s$ has exactly one immediate successor. If $p$ is a
tree, the \emph{set of branches of }$p$ is defined as $\left[  p\right]
=\left\{  x\mid\forall n\left(  x\upharpoonright n\in p\right)  \right\}  .$
For every $s\in p,$ we define the tree $p_{s}=\left\{  t\in p\mid s\subseteq
t\vee t\subseteq s\right\}  .$ Given $s\in\omega^{<\omega}$ define \emph{the
cone of }$s$ as $\left\langle s\right\rangle =\left\{  f\in\omega^{\omega}\mid
s\subseteq f\right\}  .$

\qquad\ \qquad\ \qquad\ \qquad\ \ \qquad\ \qquad\ \qquad\ \qquad\ \ 

Let $\mathcal{I}$ be an ideal on $\omega$, $\mathcal{F}$ a filter on $\omega$
and $\mathcal{A}$ a \textsf{MAD} family. Define\footnote{By $\wp\left(
a\right)  $ we denote the powerset of $a.$} $\mathcal{I}^{+}=\wp\left(
\omega\right)  \backslash\mathcal{I}$ i.e. the subsets of $\omega$ that are
not in $\mathcal{I}.$ We say that a forcing notion $\mathbb{P}$%
\emph{\ destroys }$\mathcal{I}$ if $\mathbb{P}$ adds an infinite subset of
$\omega$ that is almost disjoint with every element of $\mathcal{I}.$ We say
that $\mathbb{P}$ \emph{diagonalizes }$\mathcal{F}$ if $\mathbb{P}$ adds an
infinite set almost contained in every element of $\mathcal{F}.$ It is easy to
see that $\mathbb{P}$ destroys $\mathcal{I}$ if and only if $\mathbb{P}$
diagonalizes the filter $\mathcal{I}^{\ast}=\left\{  \omega\setminus A\mid
A\in\mathcal{I}\right\}  .$ By $\mathcal{I}\left(  \mathcal{A}\right)  $ we
denote the ideal generated by $\mathcal{A}$ (and the finite sets). We say that
$\mathbb{P}$ \emph{destroys }$\mathcal{A}$ if $\mathcal{A}$ is no longer
maximal after forcing with $\mathbb{P}.$ Note that $\mathbb{P}$ destroys
$\mathcal{A}$ if and only if $\mathbb{P}$ destroys the ideal $\mathcal{I}%
\left(  \mathcal{A}\right)  .$ The following result is well known (as well as
easy to prove) and will be frequently used:

\begin{lemma}
Let $\mathcal{A}$ be a \textsf{MAD} family. If $B\subseteq\omega,$ the
following are equivalent:

\begin{enumerate}
\item $B\in\mathcal{I}\left(  \mathcal{A}\right)  ^{+}.$

\item There is $\mathcal{A}_{1}\in\left[  \mathcal{A}\right]  ^{\omega}$ such
that $\left\vert B\cap A\right\vert =\omega$ for every $A\in\mathcal{A}_{1}.$
\end{enumerate}
\end{lemma}

\qquad\qquad\qquad\ \ \ 

We will need to recall the definition of the Kat\v{e}tov order:

\begin{definition}
Let $A\ $and $B$ be two countable sets,$\ \mathcal{I},\mathcal{J}$ be ideals
on $X$ and $Y$ respectively and $f:Y\longrightarrow X$.

\begin{enumerate}
\item We say that $f$ is a \emph{Kat\v{e}tov morphism from }$\left(
Y,\mathcal{J}\right)  $ \emph{to} $\left(  X,\mathcal{I}\right)  $ if
$f^{-1}\left(  A\right)  \in\mathcal{J}$ for every $A\in\mathcal{I}.$

\item We define $\mathcal{I}\leq_{\mathsf{K}}$ $\mathcal{J}$ ($\mathcal{I}$ is
Kat\v{e}tov smaller that $\mathcal{J}$ or $\mathcal{J}$ is Kat\v{e}tov above
$\mathcal{I}$) if there is a Kat\v{e}tov morphism from $\left(  Y,\mathcal{J}%
\right)  $ to $\left(  X,\mathcal{I}\right)  .$
\end{enumerate}
\end{definition}

\qquad\ \ \qquad\ \qquad\ 

The reader may consult \cite{KatetovOrderonBorelIdeals} for an interesting
survey of the Kat\v{e}tov order on Borel ideals. The Kat\v{e}tov order does
not play a crucial result in this article, but we include it in order to state
some results of Sabok and Zapletal below. The \emph{nowhere dense ideal
}\textsf{nwd }is the ideal on $2^{<\omega}$ where $A\in$ \textsf{nwd }if and
only if for every $s\in2^{<\omega}$ there is $t\in2^{<\omega}$ extending $s$
such that no further extension of $t$ is in $A.$ It is easy to see that
\textsf{nwd }is an ideal. For every $n\in\omega$ we define $C_{n}=\left\{
\left(  n,m\right)  \mid m\in\omega\right\}  $ and if $f:\omega\longrightarrow
\omega$ let $D\left(  f\right)  =\left\{  \left(  n,m\right)  \mid m\leq
f(n)\right\}  .$ The ideal \textsf{fin}$\times$\textsf{fin} is the ideal on
$\omega\times\omega$ generated by $\left\{  C_{n}\mid n\in\omega\right\}
\cup\left\{  D\left(  f\right)  \mid f\in\omega^{\omega}\right\}  .$

\qquad\ \ \qquad\ \ \ 

If $\mathcal{F}$ is a filter on $\omega$ (or on any countable set) we define
the \emph{Mathias forcing }$\mathbb{M}\left(  \mathcal{F}\right)  $ \emph{with
respect to }$\mathcal{F}$ as the set of all pairs $\left(  s,A\right)  $ where
$s\in\left[  \omega\right]  ^{<\omega}$ and $A\in\mathcal{F}$. If $\left(
s,A\right)  ,\left(  t,B\right)  \in\mathbb{M}\left(  \mathcal{F}\right)  $
then $\left(  s,A\right)  \leq\left(  t,B\right)  $ if the following
conditions hold:

\begin{enumerate}
\item $t$ is an initial segment of $s.$

\item $A\subseteq B.$

\item $\left(  s\setminus t\right)  \subseteq B.$\ \qquad\ \ \ 
\end{enumerate}

\qquad\ \ \ \ \ \qquad\qquad\ \ \qquad\ \ 

The \emph{Laver forcing }$\mathbb{L}\left(  \mathcal{F}\right)  $ \emph{with
respect to }$\mathcal{F}$ is the set of all trees $p$ such that $suc_{p}%
\left(  s\right)  \in\mathcal{F}$ for every $s\in p$ extending the stem of
$p.$ We say $p\leq q$ if $p\subseteq q.$

\qquad\ \ \ \ 

While $\mathbb{L}\left(  \mathcal{F}\right)  $ always adds a dominating real,
this may not be the case for $\mathbb{M}\left(  \mathcal{F}\right)  .$ A
trivial example is taking $\mathcal{F}$ to be the cofinite filter in $\omega,$
since in this case $\mathbb{M}\left(  \mathcal{F}\right)  $ is forcing
equivalent to Cohen forcing. A more interesting example was found by Canjar in
\cite{Canjar}, where an ultrafilter whose Mathias forcing does not add
dominating reals was constructed under $\mathfrak{d=c}$ (see also
\cite{CanjarFilters}). We say that a filter $\mathcal{F}$ is \emph{Canjar }if
$\mathbb{M}\left(  \mathcal{F}\right)  $ does not add dominating reals. Let
$\mathcal{F}$ be a filter on $\omega.$ We define the filter $\mathcal{F}%
^{<\omega}$ on $\left[  \omega\right]  ^{<\omega}\setminus\left\{
\emptyset\right\}  $ generated by $\left\{  \left[  A\right]  ^{<\omega
}\setminus\left\{  \emptyset\right\}  \mid A\in\mathcal{F}\right\}  .$ Note
that $X\in\left(  \mathcal{F}^{<\omega}\right)  ^{+}$ if and only if for every
$A\in\mathcal{F}$, there is $s\in X$ such that $s\subseteq A.$ It is important
to emphasize that if $X\in\left(  \mathcal{F}^{<\omega}\right)  ^{+}$ then by
convention $\emptyset\notin X$ (recall that $\mathcal{F}^{<\omega}$ is a
filter on $\left[  \omega\right]  ^{<\omega}\setminus\left\{  \emptyset
\right\}  $). In \cite{MathiasPrikryandLaverPrikryTypeForcing} Hru\v{s}\'{a}k
and Minami showed that the forcing properties of $\mathbb{M}\left(
\mathcal{F}\right)  $ are closely related to the combinatorics of
$\mathcal{F}^{<\omega}.$ They proved the following result:

\begin{proposition}
[\cite{MathiasPrikryandLaverPrikryTypeForcing}]Let $\mathcal{F}$ be a filter
on $\omega.$ The following are equivalent:

\begin{enumerate}
\item $\mathcal{F}$ is Canjar.

\item For every $\left\{  X_{n}\mid n\in\omega\right\}  \subseteq\left(
\mathcal{F}^{<\omega}\right)  ^{+}$ there are $Y_{n}\in\left[  X_{n}\right]
^{<\omega}$ such that $%
{\textstyle\bigcup\limits_{n\in\omega}}
Y_{n}\in\left(  \mathcal{F}^{<\omega}\right)  ^{+}$.
\end{enumerate}
\end{proposition}

\qquad\qquad\qquad\ \ \ \qquad\ \ \ \ \ \ 

In \cite{MathiasForcingandCombinatorialCoveringPropertiesofFilters} it was
proved that a filter is Canjar if and only if is has the Menger property (as a
subspace of $\wp\left(  \omega\right)  $). Canjar filters have been further
studied in \cite{HrusakVerner}, \cite{CanjarFilters}, \cite{CanjarFiltersII},
\cite{NonDominatingUltrafilters} and \cite{RodrigoPaul}.\qquad\ \ \ 

We will say that a family of functions $\mathcal{B}\subseteq\omega^{\omega}$
is a $\mathfrak{b}$\emph{-family }if the following holds:

\begin{enumerate}
\item Every element of $\mathcal{B}$ is an increasing function.

\item Given $\left\{  f_{n}\mid n\in\omega\right\}  \subseteq\mathcal{B}$
there is $g\in\mathcal{B}$ such that $f_{n}\leq^{\ast}g$ for every $n\in
\omega.$

\item $\mathcal{B}$ is unbounded.\qquad\ \ 
\end{enumerate}

\qquad\ \ 

An example of a $\mathfrak{b}$-family would be a well-ordered unbounded
family, another example is the set of all increasing functions. If
$\mathcal{B}$ is a $\mathfrak{b}$-family and $\mathbb{P}$ is a partial order,
we say that $\mathbb{P}$ \emph{preserves }$\mathcal{B}$ if $\mathcal{B}$ is
still unbounded after forcing with $\mathbb{P}.$ Note that if $\mathbb{P}$ is
a proper forcing that preserves $\mathcal{B},$ then $\mathcal{B}$ is still a
$\mathfrak{b}$-family in the extension. We will need the following easy lemma:

\begin{lemma}
Let $\mathcal{B}$ be a $\mathfrak{b}$-family. If $\mathcal{B}=%
{\textstyle\bigcup\limits_{n\in\omega}}
\mathcal{B}_{n},$ then there is $n\in\omega$ such that $\mathcal{B}_{n}$ is
cofinal in $\mathcal{B}$ (i.e. for every $f\in\mathcal{B}$ there is
$g\in\mathcal{B}_{n}$ such that $f\leq^{\ast}g$).
\end{lemma}

\begin{proof}
We will argue a proof by contradiction. Assume this is not the case, so for every
$n\in\omega$ there is $f_{n}\in\mathcal{B}$ such that $f_{n}$ is not bounded
by any element of $\mathcal{B}_{n}.$ Since $\mathcal{B}$ is a $\mathfrak{b}%
$-family, we can find $g\in\mathcal{B}$ such that $f_{n}\leq^{\ast}g$ for
every $n\in\omega.$ Since $\mathcal{B}=%
{\textstyle\bigcup\limits_{n\in\omega}}
\mathcal{B}_{n},$ there must be $m\in\omega$ such that $g\in\mathcal{B}_{m},$
hence $f_{m}\leq^{\ast}g$ which is a contradiction.
\end{proof}

\qquad\ \ \ \ 

Given a sequence $\overline{X}=\left\{  X_{n}\mid n\in\omega\right\}
\subseteq\left[  \omega\right]  ^{<\omega}\backslash\left\{  \emptyset
\right\}  $ and $f\in\omega^{\omega},$ we define the set $\overline{X}%
_{f}=\bigcup\limits_{n\in\omega}\left(  X_{n}\cap\wp\left(  f\left(  n\right)
\right)  \right)  $. Note that $\mathcal{F}$ is \emph{Canjar} if for every
sequence $\overline{X}=\left\{  X_{n}\mid n\in\omega\right\}  \subseteq\left(
\mathcal{F}^{<\omega}\right)  ^{+}$ there is $f\in\omega^{\omega}$ such that
$\overline{X}_{f}\in\left(  \mathcal{F}^{<\omega}\right)  ^{+}.$ Recall that
by a theorem of Hru\v{s}\'{a}k and Minami (see
\cite{MathiasPrikryandLaverPrikryTypeForcing}), a filter $\mathcal{F}$ is
Canjar if and only if $\mathbb{M}\left(  \mathcal{F}\right)  $ does not add
dominating reals. If $\mathcal{B}$ is a $\mathfrak{b}$-family, we say that
$\mathcal{F}$ is $\mathcal{B}$\emph{-Canjar} if for every sequence
$\overline{X}=\left\{  X_{n}\mid n\in\omega\right\}  \subseteq\left(
\mathcal{F}^{<\omega}\right)  ^{+}$ there is $f\in\mathcal{B}$ such that
$\overline{X}_{f}\in\left(  \mathcal{F}^{<\omega}\right)  ^{+}$ (such
$\overline{X}_{f}$ is called a \emph{pseudointersection according to}
$\mathcal{B}$). Note that if $\mathcal{F}$ is $\mathcal{B}$-Canjar (for some
$\mathfrak{b}$-family $\mathcal{B}$), then $\mathcal{F}$ is Canjar. As
expected, $\mathcal{B}$-Canjar filters have a similar characterization as the
one of Canjar. The following is a slight strengthening of proposition 1 of
\cite{CanjarFiltersII}, which is a generalization of the theorem of
Hru\v{s}\'{a}k and Minami:

\begin{proposition}
Let $\mathcal{B}$ be a $\mathfrak{b}$-family. A filter $\mathcal{F}$ is a
$\mathcal{B}$-Canjar filter if and only if $\mathbb{M}\left(  \mathcal{F}%
\right)  $ preserves $\mathcal{B}.$
\end{proposition}

\begin{proof}
Assume that $\mathcal{F}$ does not preserve $\mathcal{B}$, in other words,
there is a name $\dot{g}$ for an increasing function such that $1_{\mathbb{M}%
\left(  \mathcal{F}\right)  }\Vdash``\dot{g}$ is an upper bound for
$\mathcal{B}\textquotedblright.$ For every function $f\in\mathcal{B}$ let
$s_{f}\in$ $\left[  \omega\right]  ^{<\omega},n_{f}\in\omega$ and $F_{f}%
\in\mathcal{F}$ such that $\left(  s_{f},F_{f}\right)  \Vdash``\forall i\geq
n_{f}(f\left(  i\right)  <_{n_{f}}\dot{g}\left(  i\right)  )\textquotedblright%
.$ Since $\mathcal{B}$ is a $\mathfrak{b}$-family, there are $s\in$ $\left[
\omega\right]  ^{<\omega},n\in\omega$ and a cofinal family $\mathcal{B}%
^{\prime}\subseteq\mathcal{B}$ such that $s_{f}=s$ and $n_{f}=n$ for every
$f\in\mathcal{B}^{\prime}.$


For every $m\in\omega$ let $X_{m}$ be the set of all $t\in\lbrack\omega$
$\backslash\bigcup s]^{<\omega}$ such that there is $F\in\mathcal{F}$ with the
property that $\left(  s\cup t,F\right)  $ decides $\left\langle \dot
{g}\left(  0\right)  ,\ldots,\dot{g}\left(  m\right)  \right\rangle $ and
$\left(  s\cup t,F\right)  \Vdash``\dot{g}\left(  m\right)  <max\left(
t\right)  \textquotedblright.$ It is easy to see that $\overline{X}=\left\{
X_{m}\mid m\in\omega\right\}  $ is a sequence of sets in $\left(
\mathcal{F}^{<\omega}\right)  ^{+}$. We will see that it has no
pseudointersection according to $\mathcal{B}.$ Since $\mathcal{B}^{\prime}$ is
cofinal in $\mathcal{B},$ it is enough to show that $\overline{X}$ has no
pseudointersection according to $\mathcal{B}^{\prime}.$


Aiming for a contradiction, assume that there is $f\in\mathcal{B}^{\prime}$
such that $\overline{X}_{f}$ is positive. Since $\overline{X}_{f}\cap\left[
F_{f}\right]  ^{<\omega}$ is infinite, pick $t\in$ $\overline{X}_{f}%
\cap\left[  F_{f}\right]  ^{<\omega}$ such that $t\in X_{k}\cap\wp\left(
f\left(  k\right)  \right)  $ with $k>n.$ Since $t\in X_{k}$ there is
$F\in\mathcal{F}$ such that $\left(  s\cup t,F\right)  \Vdash``\dot{g}\left(
k\right)  \leq max\left(  t\right)  \textquotedblright$ and note that $\left(
s\cup t,F\right)  \Vdash``\dot{g}\left(  k\right)  \leq f\left(  k\right)
\textquotedblright.$ In this way, $\left(  s\cup t,F_{h}\cap F\right)  $
forces both $``f\left(  k\right)  <\dot{g}\left(  k\right)  \textquotedblright%
$ and $``\dot{g}\left(  k\right)  \leq f\left(  k\right)  \textquotedblright,$
which is a contradiction.

\qquad\ \ \ \ \ \ \ \ \ \ \qquad\ \ \ 

Now assume that $\mathbb{M}\left(  \mathcal{F}\right)  $ preserves
$\mathcal{B}.$ Let $\overline{X}=\left\langle X_{n}\mid n\in\omega
\right\rangle $ be a sequence of sets in $\left(  \mathcal{F}^{<\omega
}\right)  ^{+}$. Let $r_{gen}$ be a $\left(  V,\mathbb{M}\left(
\mathcal{F}\right)  \right)  $-generic real, observe that $\left[
r_{gen}\right]  ^{<\omega}$ intersect infinitely every member of $\left(
\mathcal{F}^{<\omega}\right)  ^{+}.$ In this way, in $V\left[  r_{gen}\right]
$ we may define an increasing function $g:\omega\longrightarrow\omega$ such
that $\left(  r_{gen}\setminus n\right)  \cap g\left(  n\right)  $ contains a
member of $X_{n}.$ Since $\mathcal{F}$ preserves $\mathcal{B}$, then there is
$f\in\mathcal{B}$ such that $f\nleq^{\ast}g,$ we will see that $\overline
{X}_{f}$ is positive. Let $F\in\mathcal{F}$ we must prove that $\overline
{X}_{f}\cap\left[  F\right]  ^{<\omega}$ is not empty. Since $F\in\mathcal{F}%
$, we know that $r_{gen}\subseteq^{\ast}F$ so there is $k\in\omega$ such that
$g\left(  k\right)  <f\left(  k\right)  $ and $r_{gen}\setminus k\subseteq F$
and hence $\overline{X}_{f}\cap\left[  F\right]  ^{<\omega}\neq\emptyset.$
\end{proof}

\qquad\ \ \qquad\ \ \ \ 

Moreover, Canjar filters satisfy the following stronger property:

\begin{lemma}
Let $\mathcal{B}$ be a $\mathfrak{b}$-family and $\mathcal{F}$ a $\mathcal{B}%
$-Canjar filter. For every family $\overline{X}=\left\{  X_{n}\mid n\in
\omega\right\}  \subseteq\left(  \mathcal{F}^{<\omega}\right)  ^{+}$ there is
$f\in\mathcal{B}$ such that for every $n\in\omega,$ if $Y_{n}=\left\{  s\in
X_{n}\mid s\subseteq\lbrack f\left(  n-1\right)  ,f\left(  n\right)
)\right\}  $ (where $f\left(  -1\right)  =0$) then $%
{\textstyle\bigcup\limits_{n\in\omega}}
Y_{n}\in\left(  \mathcal{F}^{<\omega}\right)  ^{+}.$ \label{CanjarParticion}
\end{lemma}

\begin{proof}
The idea of the proof is similar to the previous one. Let $r_{gen}$ be a
$\left(  \mathbb{M}\left(  \mathcal{F}\right)  ,V\right)  $-generic real. In
$V\left[  r_{gen}\right]  $ we find an increasing function $g\in\omega
^{\omega}$ such that for every $n\in\omega,$ the following holds: $r_{gen}%
\cap\lbrack g\left(  n-1\right)  ,g\left(  n\right)  )$ contains an element of
$X_{n}$ (where $g\left(  -1\right)  =0$). Furthermore, we may assume that $g$
is unbounded over $V$ (this is possible since $\mathbb{M}\left(
\mathcal{F}\right)  $ adds an unbounded real\footnote{It is well known that
every $\sigma$-centered forcing adds an unbounded real.}). Since $\mathcal{F}$
is $\mathcal{B}$-Canjar, we can find $\left(  s,A\right)  \in\mathbb{M}\left(
\mathcal{F}\right)  $ and $f\in\mathcal{B}$ such that $\left(  s,A\right)
\Vdash``f\nleq^{\ast}\dot{g}\textquotedblright.$ We claim that $f$ is the
function we are looking for. Define $Y_{n}=\left\{  s\in X_{n}\mid
s\subseteq\lbrack f\left(  n-1\right)  ,f\left(  n\right)  )\right\}  $ (for
every $n\in\omega$) and $Y=$ $%
{\textstyle\bigcup\limits_{n\in\omega}}
Y_{n},$ we must prove that $Y\in\left(  \mathcal{F}^{<\omega}\right)  ^{+}.$
Let $B\in\mathcal{F},$ we must show that $Y$ contains an element of $B.$
Without loss of generality, we may assume that $B\subseteq A.$ We can now find
$m\in\omega$, $t\in\left[  \omega\right]  ^{<\omega}$ and $C\in\mathcal{F}$
such that the following holds:

\begin{enumerate}
\item $max\left(  s\right)  <m,$ $C\subseteq B,$ $t\subseteq B.$

\item $max\left(  s\right)  <min\left(  t\right)  .$

\item $\left(  s\cup t,C\right)  \leq\left(  s,B\right)  .$

\item There is $w\in X_{m}$ such that $\left(  s\cup t,C\right)
\Vdash``w\subseteq\lbrack\dot{g}\left(  m-1\right)  ,\dot{g}\left(  m\right)
)\textquotedblright.$

\item $\left(  s\cup t,C\right)  \Vdash``\left(  f\left(  m-1\right)  \leq
\dot{g}\left(  m-1\right)  \right)  \textquotedblright$ and $\left(  s\cup
t,C\right)  \Vdash``\left(  \dot{g}\left(  m\right)  <f\left(  m\right)
\right)  \textquotedblright.$
\end{enumerate}

\qquad\ \ \ 

Such $m,t$ and $C$ can be obtained since $\left(  s,B\right)  \leq\left(
s,A\right)  ,$ so $\left(  s,B\right)  $ forces that $\dot{g}$ is an unbounded
real that does not dominate $f.$ Note that $w\subseteq t,$ so $w\subseteq B$
and $\left(  s\cup t,C\right)  \Vdash``[\dot{g}\left(  m-1\right)  ,\dot
{g}\left(  m\right)  )\subseteq\lbrack f\left(  m-1\right)  ,f\left(
m\right)  )\textquotedblright,$ which implies that $w\in Y_{m}.$ This finishes
the proof.
\end{proof}

\qquad\ \ \ \ \ \ 

\section{Miller forcing based on filters\qquad\ \ \ \ \ \ \ }

The theory of destructibility of ideals is very important in forcing theory,
since many important forcing properties may be stated in these terms. For
example, it is well known that a forcing $\mathbb{P}$ adds a dominating real
if and only if $\mathbb{P}$ destroys \textsf{fin}$\times$\textsf{fin. }The
reader may consult \cite{OrderingMADFamiliesalaKatetov},
\cite{ForcingIndestructibilityofMADFamilies}, \cite{KurilicMAD},
\cite{ForcingwithQuotients} or \cite{ZappingSmallFilters} for more on
destructibility of ideals.

\qquad\ \ \ 

In order to build models where $\mathfrak{a}$ is big and some other cardinal
invariant is small, we need to be able to destroy a \textsf{MAD} family by
dealing the least ``damage''
as possible to
the ground model. The most well known forcings to destroy an ideal (or to
diagonalize a filter) are the Mathias or Laver forcings relative to the ideal
(filter). The following result is well known:

\begin{proposition}
Let $\mathcal{F}$ be a filter on $\omega.$

\begin{enumerate}
\item $\mathbb{L}\left(  \mathcal{F}\right)  $ adds a dominating real.

\item $\mathbb{M}\left(  \mathcal{F}\right)  $ adds a Cohen real if and only
if $\mathcal{F}$ is not a Ramsey ultrafilter.

\item If $\mathcal{F}$ is a Ramsey ultrafilter, then $\mathbb{M}\left(
\mathcal{F}\right)  $ adds a dominating real.
\end{enumerate}
\end{proposition}

\qquad\ \ 

In particular, it follows that $\mathbb{M}\left(  \mathcal{F}\right)  $ adds
either a Cohen or a dominating real. In this section, we will introduce a
forcing relative to a filter $\mathcal{F}$ that in some cases, might destroy
$\mathcal{F}$ without adding Cohen or dominating reals.

\qquad\ \ \ \ \qquad\ \ \ \ 

We say that a tree $p\subseteq\omega^{<\omega}$ is a \emph{Miller tree }if the
following conditions hold:

\begin{enumerate}
\item $p$ consists of increasing sequences.

\item $p$ has a stem.

\item For every $s\in p,$ there is $t\in p$ such that $s\subseteq t$ and $t$
is a splitting node.\footnote{Recall that $s$ is a splitting node of $p$ if
$suc_{p}\left(  s\right)  $ is infinite.}
\end{enumerate}

\qquad\ \ \ 

Usually, Miller trees are required to satisfy the following extra condition:

\begin{enumerate}
\item[4.] Every node of $p$ either is a splitting node or it has exactly one
immediate successor.
\end{enumerate}

\qquad\ \ \ \ \qquad\qquad\qquad\ \ \ 

However, we will not assume this extra requirement. The \emph{Miller forcing}
$\mathbb{PT}$ consists of all Miller trees ordered by
inclusion.\footnote{Obviously, the trees satisfying property 4 are dense in
Miller forcing. However, this does not seem to be the case for our forcings.}
Miller forcing (also called \textquotedblleft super perfect
forcing\textquotedblright) was introduced by Miller in \cite{Rat}, and is one
of the most useful and studied forcings for adding new reals (see
\cite{Barty}, \cite{Rat} or \cite{Combinatorial}). By $split\left(  p\right)
$ we denote the collection of all splitting nodes, and by $split_{n}\left(
p\right)  $ we denote the collection of $n$-splitting nodes (i.e. $s\in
split_{n}\left(  p\right)  $ if $s\in split\left(  p\right)  $ and $s$ has
exactly $n$-restrictions that are splitting nodes). Note that $split_{0}%
\left(  p\right)  =\left\{  stem\left(  p\right)  \right\}  .$

\qquad\ \ \ 

In \cite{ZapletalSabok}, Sabok and Zapletal introduced the following
parametrized version of Miller forcing\footnote{In \cite{ZapletalSabok} the
authors use ideals instead of filters. Evidently, this choice is superflous.}:

\begin{definition}
Let $\mathcal{F}$ be a filter. By $\mathbb{Q}\left(  \mathcal{F}\right)  $ we
denote the set of all Miller trees $p\in\mathbb{PT}$ such that $suc_{p}\left(
s\right)  \in\mathcal{F}^{+}$ for every splitting node $s.$ We order
$\mathbb{Q}\left(  \mathcal{F}\right)  $ by inclusion.
\end{definition}

\qquad\ \ \ \qquad\ \ \ \ 

Sabok and Zapletal proved some very interesting results, like the following:
(the reader may consult \cite{ZapletalSabok} and \cite{TesisDavid} for the
definitions of \textsf{spl }and the \emph{Solecki ideal} $\mathcal{S}$).

\begin{proposition}
[\cite{ZapletalSabok}]Let $\mathcal{F}$ be a filter.

\begin{enumerate}
\item $\mathbb{Q}\left(  \mathcal{F}\right)  $ does not add Cohen reals if and
only if \textsf{nwd }$\nleq_{\mathsf{K}}\mathcal{F}^{\ast}\upharpoonright A $
for every $A\in\mathcal{F}^{+}.$

\item If $\mathcal{F}$ is a Borel filter, then$\ \mathbb{Q}\left(
\mathcal{F}\right)  $ preserves outer Lebesgue measure if and only if
$\mathcal{S}$\textsf{\ }$\nleq_{\mathsf{K}}\mathcal{F}^{\ast}\upharpoonright
A$ for every $A\in\mathcal{F}^{+}.$

\item If $\mathcal{F}$ is a Borel filter, then$\ \mathbb{Q}\left(
\mathcal{F}\right)  $ does not add splitting reals if and only if \textsf{spl
}$\nleq_{\mathsf{K}}\mathcal{F}^{\ast}\upharpoonright A$ for every
$A\in\mathcal{F}^{+}.$
\end{enumerate}
\end{proposition}

\qquad\ \ \ \ \qquad\ \ 

In some cases, $\mathbb{Q}\left(  \mathcal{F}\right)  $ may diagonalize
$\mathcal{F}$ while in some others cases not, as can be seen with the following:

\begin{lemma}
\qquad\ \ \ 

\begin{enumerate}
\item $\mathbb{Q(}$\textsf{fin}$\times$\textsf{fin}$^{\ast})$ does not destroy
\textsf{fin}$\times$\textsf{fin}$.$

\item $\mathbb{Q(}$\textsf{nwd}$^{\ast})$ destroys \textsf{nwd}$.$
\end{enumerate}
\end{lemma}

\begin{proof}
With an easy fusion argument, it is possible to prove that $\mathbb{Q}\left(
\mathcal{F}\right)  $ does not add dominating reals (for every $\mathcal{F}$).
Alternatively, this can be proved as follows: In \cite{ZapletalSabok} it was
proved that there is a $\sigma$-ideal $\mathcal{J}_{\mathcal{F}}$ generated by
closed sets such that $\mathbb{Q}\left(  \mathcal{F}\right)  $ is forcing
equivalent to \textsf{Borel}$\left(  \omega^{\omega}\right)  /\mathcal{J}%
_{\mathcal{F}},$ which does not add dominating reals by Theorem 4.1.2 of
\cite{ForcingIdealized}. It is well known that if a forcing destroys
\textsf{fin}$\times$\textsf{fin}$,$ then it must add a dominating real. From
these two facts, it follows that $\mathbb{Q(}$\textsf{fin}$\times$%
\textsf{fin}$^{\ast})$ does not destroy \textsf{fin}$\times$\textsf{fin}$.$ On
the other hand, note that $\mathbb{Q(}$\textsf{nwd}$^{\ast})$ adds a Cohen
real and it is easy to see that any forcing adding a Cohen real must destroy
\textsf{nwd.}
\end{proof}

\qquad\ \ \ \ 

We will now introduce a version of Miller forcing parametrized with a filter
$\mathcal{F}$ that will always diagonalizes $\mathcal{F}.$ Given
$p\in\mathbb{PT\ }$for every $s\in split_{n}\left(  p\right)  $ we define
$spsuc_{p}\left(  s\right)  =\left\{  t\setminus s\mid t\in split_{n+1}\left(
p\right)  \wedge s\subseteq t\right\}  .$

\begin{definition}
Let $\mathcal{F}$ be a filter. We say $p\in\mathbb{PT}\left(  \mathcal{F}%
\right)  $ if the following holds:

\begin{enumerate}
\item $p\in\mathbb{PT}.$

\item If $s\in split\left(  p\right)  $ then $spsuc_{p}\left(  s\right)
\in\left(  \mathcal{F}^{<\omega}\right)  ^{+}.$\qquad\ \ \ 
\end{enumerate}
\end{definition}

\qquad\ \ 

We order $\mathbb{PT}\left(  \mathcal{F}\right)  $ by inclusion. The following
lemma contains some basic results about $\left(  \mathcal{F}^{<\omega}\right)
^{+},$ which will be used implicitly.

\begin{lemma}
Let $\mathcal{F}$ be a filter on $\omega$ and $X\in\left(  \mathcal{F}%
^{<\omega}\right)  ^{+}.$

\begin{enumerate}
\item If $F\in\mathcal{F},$ then $X\cap\left[  F\right]  ^{<\omega}\in\left(
\mathcal{F}^{<\omega}\right)  ^{+}.$

\item If $X=A\cup B,$ then either $A\in\left(  \mathcal{F}^{<\omega}\right)
^{+}$ or $B\in\left(  \mathcal{F}^{<\omega}\right)  ^{+}.$

\item If $X=\left\{  s_{n}\mid n\in\omega\right\}  $ and $Y=\left\{  y_{n}\mid
n\in\omega\right\}  $ is such that every $y_{n}$ is a non-empty subset of
$s_{n}.$ then $Y\in\left(  \mathcal{F}^{<\omega}\right)  ^{+}.$

\item If $r_{gen}$ is an $\left(  \mathbb{M}\left(  \mathcal{F}\right)
,V\right)  $-generic real, then $r_{gen}$ contains an element of $X.$
\end{enumerate}
\end{lemma}

\qquad\ \ \ \ 

By point 1 above we get the following:

\begin{lemma}
$\mathbb{PT}\left(  \mathcal{F}\right)  $ diagonalizes $\mathcal{F}.$
\end{lemma}

\qquad\ \ \ 

Given $A\subseteq\left[  \omega\right]  ^{<\omega}$ we define $minimal\left(
A\right)  \subseteq A$ as the set of all minimal elements of $A$ with respect
to the initial segment relation. Note that every element of $A$ contains an
element of $minimal\left(  A\right)  .$ We conclude that if $A\in\left(
\mathcal{F}^{<\omega}\right)  ^{+}$ then $minimal\left(  A\right)  \in\left(
\mathcal{F}^{<\omega}\right)  ^{+}.$

\qquad\ \ \qquad\ \ \qquad\ \ 

Given $p\in\mathbb{PT}\left(  \mathcal{F}\right)  ,$ $s\in split\left(
p\right)  $ and $D\subseteq\mathbb{PT}\left(  \mathcal{F}\right)  $ an open
dense subset, we define $E\left(  D,p,s\right)  =minimal\left(  \left\{
t\setminus s\mid\exists q\leq p_{s}\left(  stem\left(  q\right)  =t\wedge q\in
D\right)  \right\}  \right)  .$ We now have the following:

\begin{lemma}
If $p\in\mathbb{PT}\left(  \mathcal{F}\right)  ,$ $s\in split\left(  p\right)
$ and $D\subseteq\mathbb{PT}\left(  \mathcal{F}\right)  $ is an open dense
subset, then $E\left(  D,p,s\right)  \in\left(  \mathcal{F}^{<\omega}\right)
^{+}.$
\end{lemma}

\begin{proof}
It would be enough to prove that $\left\{  t\setminus s\mid\exists q\leq
p_{s}\left(  stem\left(  q\right)  =t\wedge q\in D\right)  \right\}  $ is in
$\left(  \mathcal{F}^{<\omega}\right)  ^{+},$ which is immediate.
\end{proof}

\qquad\ \ \ \qquad\ \ 

We can then prove the following:

\begin{proposition}
Let $\mathcal{F}$ be a filter.

\begin{enumerate}
\item Let $M$ be a countable elementary submodel, $p\in\mathbb{PT}\left(
\mathcal{F}\right)  $, $D\subseteq\mathbb{PT}\left(  \mathcal{F}\right)  $ an
open dense subset with $p,D\in M$ and $s\in split\left(  p\right)  .$ There is
$q\leq p$ with $stem\left(  q\right)  =s$ such that $q\Vdash``\dot{G}\cap
D\cap M\neq\emptyset\textquotedblright.$ $\ \ \ \ \ \ $

\item $\mathbb{PT}\left(  \mathcal{F}\right)  $ is proper.
\end{enumerate}
\end{proposition}

\begin{proof}
Let $M$ be a countable elementary submodel, $p\in\mathbb{PT}\left(
\mathcal{F}\right)  $, $D\subseteq\mathbb{PT}\left(  \mathcal{F}\right)  $ an
open dense subset with $p,D\in M$ and $s\in split\left(  p\right)  .$ Since
$p,D,s\in M,$ it follows that $E\left(  D,p,s\right)  \in M.$ For every $t$
such that $t\setminus s\in E\left(  D,p,s\right)  ,$ we choose $q\left(
t\right)  \in M\cap D$ such that $q\left(  t\right)  \leq p_{s}$ and the stem
of $q\left(  t\right)  $ is $t$. Define $q=%
{\textstyle\bigcup}
\left\{  q\left(  t\right)  \mid t\setminus s\in E\left(  D,p,s\right)
\right\}  .$ We will show that $q$ has the desired properties. Define
$L=\left\{  t\mid t\setminus s\in E\left(  D,p,s\right)  \right\}  ,$ we will
first prove that $q\in\mathbb{PT}\left(  \mathcal{F}\right)  .$ Let $z\in q$
be a splitting node. If $z$ extends a $t\in L,$ then $suc_{q}\left(  z\right)
=suc_{q\left(  t\right)  }\left(  z\right)  ,$ so $suc_{q}\left(  z\right)
\in\left(  \mathcal{F}^{<\omega}\right)  ^{+}.$ Now, assume that $z$ does not
extend an element of $L,$ it is enough to prove that $suc_{q}\left(  z\right)
=suc_{p}\left(  z\right)  .$ Let $n$ such that $z\in split_{n}\left(
q\right)  $ and $w\in split_{n+1}\left(  q\right)  $ such that $z\subseteq w.$
Since $z$ does not extend an element of $L,$ we know that $w$ can be extended
to a $t\in L,$ so $w\in q\left(  t\right)  .$ Note that this argument shows
that $stem\left(  q\right)  =s.$


We will now prove that $q\Vdash``\dot{G}\cap D\cap M\neq\emptyset
\textquotedblright.$ Let $q_{1}\leq q$. We may assume that $q_{1}\in D.$ Let
$w=stem\left(  q_{1}\right)  .$ By construction, (recall that we took the
minimal elements) we know there is $t\in L$ such that $t\subseteq w$, in this
way $q_{1}\leq q\left(  t\right)  ,$ so $q_{1}\Vdash``q\left(  t\right)
\in\dot{G}\cap D\cap M\textquotedblright$ and we are done.


The second part of the lemma follows by the first part and a fusion argument.
\end{proof}

\qquad\ \ \ \ \qquad\qquad\qquad\qquad\qquad\ \ 

We will need the following lemma:

\begin{lemma}
Let $\mathcal{F}$ be a Canjar filter, $\left\{  D_{n}\mid n\in\omega\right\}
$ open dense subsets of $\mathbb{PT}\left(  \mathcal{F}\right)  $ and
$p\in\mathbb{PT}\left(  \mathcal{F}\right)  $ with $stem\left(  p\right)  =s.$
There are $q\in\mathbb{PT}\left(  \mathcal{F}\right)  $ and $\left\langle
F_{n}\right\rangle _{n\in\omega}$ such that the following holds:

\begin{enumerate}
\item $q\leq p$ and $stem\left(  q\right)  =stem\left(  p\right)  .$

\item $F_{n}$ is a finite subset of $\left[  \omega\right]  ^{<\omega
}\setminus\left\{  \emptyset\right\}  $ for every $n\in\omega.$

\item For every $n\in\omega,$ if $t_{0}\in F_{n}$ and $t_{1}\in F_{n+1},$ then
$max\left(  t_{0}\right)  <min\left(  t_{1}\right)  .$

\item $spsuc_{q}\left(  s\right)  =%
{\textstyle\bigcup\limits_{n\in\omega}}
F_{n}.$

\item If $t\in F_{n},$ then $q_{s^{\frown}t}\in D_{n}.$
\end{enumerate}
\end{lemma}

\begin{proof}
For every $n\in\omega,$ define $X_{n}=E\left(  D_{n},p,s\right)  $ (recall
that $E\left(  D,p,s\right)  =minimal\left(  \left\{  t\setminus s\mid\exists
q\leq p_{s}\left(  stem\left(  q\right)  =t\wedge q\in D\right)  \right\}
\right)  $). We know that $X_{n}\in\left(  \mathcal{F}^{<\omega}\right)
^{+}.$ By the previous result, we can find an increasing function $f\in
\omega^{\omega}$ such that if $Y_{n}=\left\{  s\in X_{n}\mid s\subseteq\lbrack
f\left(  n-1\right)  ,f\left(  n\right)  )\right\}  $ (for every $n\in\omega$)
then $%
{\textstyle\bigcup\limits_{n\in\omega}}
Y_{n}\in\left(  \mathcal{F}^{<\omega}\right)  ^{+}.$


For every $t\in Y_{n},$ choose $q\left(  t\right)  \leq p_{s^{\frown}t}$ such
that $q\left(  t\right)  \in D_{n}$ and the stem of $q\left(  t\right)  $ is
$s\cup t.$ Define $q=%
{\textstyle\bigcup\limits_{t\in Y}}
q\left(  t\right)  .$
\end{proof}

\qquad\ \ \ 

By taking all the open dense sets to be the same, we obtain the following:

\begin{corollary}
Let $\mathcal{F}$ be a Canjar filter, $D$ an open dense subset of
$\mathbb{PT}\left(  \mathcal{F}\right)  ,$ $p\in\mathbb{PT}\left(
\mathcal{F}\right)  $ with $stem\left(  p\right)  =s.$ There is $q\in
\mathbb{PT}\left(  \mathcal{F}\right)  $ such that the following holds:

\begin{enumerate}
\item $q\leq p$ and $stem\left(  q\right)  =stem\left(  p\right)  .$

\item If $t\in spsuc_{q}\left(  s\right)  $ then $q_{s^{\frown}t}\in
D.$\bigskip
\end{enumerate}
\end{corollary}

We can now prove the following result:

\begin{proposition}
Let $\mathcal{B}$ be a $\mathfrak{b}$-family and $\mathcal{F}$ a $\mathcal{B}%
$--Canjar filter.

\begin{enumerate}
\item Let $p\in\mathbb{PT}\left(  \mathcal{F}\right)  ,$ $s\in split_{m}%
\left(  p\right)  $ and $\dot{g}$ such that $p\Vdash``\dot{g}\in\omega
^{\omega}\textquotedblright.$ There are $q\in\mathbb{PT}\left(  \mathcal{F}%
\right)  ,$ $f\in\mathcal{B}$ and $\left\langle F_{n}\right\rangle
_{n\in\omega}$ such that the following holds:

\begin{enumerate}
\item $q\leq p,$ $s\in q$ and if $t\in p$ is incomparable with $s,$ then $t\in
q$ and $q_{t}=p_{t}.$

\item For every $n\in\omega,$ if $t_{0}\in F_{n}$ and $t_{1}\in F_{n+1},$ then
$max\left(  t_{0}\right)  <min\left(  t_{1}\right)  .$

\item $spsuc_{q}\left(  s\right)  =%
{\textstyle\bigcup\limits_{n\in\omega}}
F_{n}.$

\item If $t\in F_{n}$ then $q_{t}\Vdash``\dot{g}\left(  n\right)  <f\left(
n\right)  \textquotedblright.$
\end{enumerate}

\item $\mathbb{PT}\left(  \mathcal{F}\right)  $ preserves $\mathcal{B}$.
\end{enumerate}
\end{proposition}

\begin{proof}
We first prove point 1. Let $p\in\mathbb{PT}\left(  \mathcal{F}\right)  ,$
$s\in split_{m}\left(  p\right)  $ and $\dot{g}$ such that $p\Vdash``\dot
{g}\in\omega^{\omega}\textquotedblright.$ Define $D_{n}\subseteq
\mathbb{PT}\left(  \mathcal{F}\right)  $ as the set of all $q$ such that
$q\Vdash``\dot{g}\left(  n\right)  <max\left(  stem\left(  q\right)  \right)
\textquotedblright.$ It is easy to see that each $D_{n}$ is an open dense
subset of $\mathbb{PT}\left(  \mathcal{F}\right)  .$\ We now apply the
previous lemma.


We will now prove that $\mathbb{PT}\left(  \mathcal{F}\right)  $ preserves
$\mathcal{B}$ as an unbounded family. Let $p\in\mathbb{PT}\left(
\mathcal{F}\right)  $ and $\dot{g}$ such that $p\Vdash``\dot{g}\in
\omega^{\omega}\textquotedblright.$\ By the previous point and a fusion
argument, we may assume that for every $s\in split_{m}\left(  p\right)  $
there are $f_{s}\in\mathcal{B}$ and $\left\langle F_{n}^{s}\right\rangle
_{n\in\omega}$ such that the following holds:

\begin{enumerate}
\item Each $F_{n}^{s}$ is a finite subset of $p.$

\item $spsuc_{p}\left(  s\right)  =%
{\textstyle\bigcup\limits_{n\in\omega}}
F_{n}^{s}.$

\item if $t\in F_{n}^{s}$ then $p_{t}\Vdash``\dot{g}\left(  n\right)
<f_{s}\left(  n\right)  \textquotedblright.$
\end{enumerate}

\qquad\ \ 

Since $\mathcal{B}$ is a $\mathfrak{b}$-family, we can find $f\in\mathcal{B}$
such that $f_{s}\leq^{\ast}f$ for every $s\in split\left(  p\right)  .$ For
every $s\in split\left(  p\right)  ,$ let $m_{s}$ such that if $i>m_{s}$ then
$f_{s}\left(  i\right)  \leq f\left(  i\right)  .$ We can recursively build
$q\leq p$ such that $split\left(  q\right)  =split\left(  p\right)  \cap q$
and $suc_{q}\left(  s\right)  =%
{\textstyle\bigcup\limits_{n>m_{s}}}
F_{n}^{s}.$ It is easy to see that $q$ forces that $\dot{g}$ does not dominate
$f.$
\end{proof}

\qquad\ \ 

It is worth mentioning that $\mathbb{PT}\left(  \mathcal{F}\right)  $ may add
dominating reals for certain filters $\mathcal{F}.$ The simplest example is
taking $\mathcal{F}$ to be \textsf{(fin}$\times$\textsf{fin)}$^{\ast}.$

\qquad\ \ \ \ 

The following is a very useful fact about our forcings:

\begin{lemma}
[Pure decision property]Let $\mathcal{F}$ be a Canjar filter, $p\in
\mathbb{PT}\left(  \mathcal{F}\right)  $ and $A$ a finite set. If $\dot{x}$ is
a $\mathbb{PT}\left(  \mathcal{F}\right)  $-name such that $p\Vdash``\dot
{x}\in A\textquotedblright$ then there are $q\leq p$ and $a\in A$ such that
$stem\left(  q\right)  =stem\left(  p\right)  $ and $q\Vdash``\dot
{x}=a\textquotedblright.$
\end{lemma}

\begin{proof}
Let $D$ be the set of all $q\in\mathbb{PT}\left(  \mathcal{F}\right)  $ for
which there is $a_{q}\in A$ such that $q\Vdash``\dot{x}=a_{q}%
\textquotedblright.$ Since $D$ is an open dense set, by the previous results,
we can find $\overline{p}\leq p$ with the following properties:

\begin{enumerate}
\item $stem\left(  \overline{p}\right)  =stem\left(  p\right)  .$

\item If $t\in split_{1}\left(  \overline{p}\right)  $ then $\overline{p}%
_{t}\in D.$
\end{enumerate}

\qquad\ \ \ 

Finally, since $A$ is finite, we can find $q\leq p$ with $stem\left(
q\right)  =stem\left(  p\right)  $ and $a\in A$ such that $q_{t}\Vdash
``\dot{x}=a\textquotedblright$ for every $t\in split_{1}\left(  q\right)  .$
It follows that $q\Vdash``\dot{x}=a\textquotedblright.$
\end{proof}

\qquad\ \ 

By $\mathcal{K}_{\sigma}$ we denote the ideal generated by all $\sigma
$-compact sets on the Baire space. For every function $L:\omega^{<\omega
}\longrightarrow\omega,$ let $K\left(  L\right)  $ be the set defined as
$\left\{  x\in\omega^{\omega}\mid\forall^{\infty}n\left(  x\left(  n\right)
\leq L\left(  x\upharpoonright n\right)  \right)  \right\}  .$ It is easy to
see that $K\left(  L\right)  \in\mathcal{K}_{\sigma}.$ It is well known that
for every $K\in\mathcal{K}_{\sigma}$ there is $f\in\omega^{\omega}$ such that
$\mathcal{K}\subseteq\left\{  g\in\omega^{\omega}\mid g\leq^{\ast}f\right\}  $
(see \cite{Barty} page 6). The following well known result follows from this remark:

\begin{lemma}
Let $\mathcal{B}$ $\subseteq\omega^{\omega}$ be an unbounded family. If
$\mathbb{P}$ is a forcing that preserves $\mathcal{B},$ then $\mathbb{P\Vdash
}``\mathcal{B}\notin\mathcal{K}_{\sigma}\textquotedblright.$
\end{lemma}

\qquad\ \ \ 

If $T$ is a finite tree, we will denote the set of its maximal nodes as
$\left[  T\right]  .$ We will need the following notion:\qquad\ \ \ \ \ 

\begin{definition}
Let $p\in\mathbb{PT}\left(  \mathcal{F}\right)  $ and $T\subseteq p$ a finite
tree such that $\left[  T\right]  \subseteq split\left(  p\right)  $. We say
that $q\leq_{T}p$ if the following conditions hold:

\begin{enumerate}
\item $q\leq p.$

\item $T\subseteq q.$

\item $T\cap split\left(  q\right)  =T\cap split\left(  p\right)  .$
\end{enumerate}
\end{definition}

\qquad\ \ \ 

We fix a canonical bijection $d:\omega\longrightarrow\omega^{<\omega}$ such
that if $d\left(  m\right)  \subseteq d\left(  n\right)  $ then $m\leq n.$
Given $p\in\mathbb{PT}\left(  \mathcal{F}\right)  $ and $n\in\omega,$ we
define the set $\widetilde{T}\left(  p,n\right)  =\left\{  s\in split\left(
p\right)  \mid d^{-1}\left(  s\right)  \leq n\right\}  .$ Let $T\left(
p,n\right)  \subseteq\omega^{<\omega}$ be the smallest tree such that
$\widetilde{T}\left(  p,n\right)  \subseteq T\left(  p,n\right)  .$ It is
clear that $T\left(  p,n\right)  $ is a finite subtree of $p$ such that
$\left[  T\left(  p,n\right)  \right]  \subseteq$ $split\left(  p\right)  $.
It is easy to see that if $q\leq_{T\left(  p,n\right)  }p$ and $n\leq m,$ then
$T\left(  p,n\right)  \subseteq T\left(  q,m\right)  .$\qquad\ \ \ 

\qquad\ \ \ \ \qquad\ \ \ \ \ \ 

Let $p\in\mathbb{PT}\left(  \mathcal{F}\right)  $ and $\mathcal{B}$ a
$\mathfrak{b}$-family. We define the game $\mathcal{G}\left(  \mathcal{F}%
,p,\mathcal{B}\right)  $ as follows:

\qquad\ \ \ \qquad\ \ \qquad\qquad\ \ \ \ \ 

\begin{center}%
\begin{tabular}
[c]{|l|l|l|l|l|l|l|l|}\hline
$\mathsf{I}$ & $p_{0}$ &  & $p_{1}$ &  & $p_{2}$ &  & $...$\\\hline
$\mathsf{II}$ &  & $n_{0}$ &  & $n_{1}$ &  & $n_{2}$ & \\\hline
\end{tabular}

\end{center}

\qquad\ \ \ \qquad\ \ 


\begin{enumerate}
\item $p_{i}\in\mathbb{PT}\left(  \mathcal{F}\right)  $ and $n_{i}\in\omega$
for every $i\in\omega.$

\item $p_{0}=p.$

\item $\left\langle n_{i}\right\rangle _{i\in\omega}$ is increasing.

\item $p_{m+1}\leq_{T_{m}}p_{m}$ where $T_{m}=T\left(  p_{m},n_{m}\right)  .$
\end{enumerate}

\qquad\ \ \ 

The player $\mathsf{II}$ will \emph{win the game }$\mathcal{G}\left(
\mathcal{F},p,\mathcal{B}\right)  $ if $%
{\textstyle\bigcup}
T_{m}\in\mathbb{PT}\left(  \mathcal{F}\right)  $ and $f\in\mathcal{B}$ where
$f$ is the function given by $f\left(  i\right)  =n_{i}.$ 

\begin{proposition}
Let $\mathcal{F}$ be a filter, $p\in\mathbb{PT}\left(  \mathcal{F}\right)  $
and $\mathcal{B}$ a $\mathfrak{b}$-family. If $\mathcal{F}$ is $\mathcal{B}%
$-Canjar, then $\mathsf{I}$ does not have a winning strategy in $\mathcal{G}%
\left(  \mathcal{F},p,\mathcal{B}\right)  .$
\end{proposition}

\begin{proof}
Let $\sigma$ be a strategy for player $\mathsf{I},$ we must prove that player
$\mathsf{II}$ can defeat $\sigma.$ Define $\left\{  p\left(  s\right)  \mid
s\in\omega^{<\omega}\right\}  \subseteq\mathbb{PT}\left(  \mathcal{F}\right)
$ as follows:

\begin{enumerate}
\item $p\left(  \emptyset\right)  =p.$

\item If $s=\left\langle n_{0},...,n_{m}\right\rangle $ then $p\left(
s\right)  $ is the tree played by player $\mathsf{I}$ at the $m$-step if he is
playing according to $\sigma$ and $\mathsf{II}$ plays $n_{i}$ at the step $i$
for $i\leq m.$
\end{enumerate}

\qquad\ \ \ \ \ \ \ 

Let $r_{gen}$ be an $\left(  \mathbb{M}\left(  \mathcal{F}\right)  ,V\right)
$-generic real (note that $r_{gen}$ is generic for $\mathbb{M}\left(
\mathcal{F}\right)  ,$ not for $\mathbb{PT}\left(  \mathcal{F}\right)  $). In
$V\left[  r_{gen}\right]  $ we define a function $L:\omega^{<\omega
}\longrightarrow\omega$ as follows: Let $s=\left\langle n_{i}\right\rangle
_{i<m}\in\omega^{<\omega},$ we look at $p\left(  s\right)  ,$ assume that
player $\mathsf{II}$ plays $n_{i}$ at the step $i$ for $i<m$ and player
$\mathsf{I}$ is following $\sigma.$ Let $T_{m}$ be the tree defined so far (as
in the definition in the game). Let $t\in T_{m}\cap split\left(  p\left(
s\right)  \right)  $ and since $spsuc_{p\left(  s\right)  }\left(  t\right)
\in\left(  \mathcal{F}^{<\omega}\right)  ^{+},$ it follows by genericity that
there is $m_{s}\left(  t\right)  \in\omega$ such that there is $u_{s}\left(
t\right)  \in spsuc_{p\left(  s\right)  }\left(  t\right)  $ for which
$u_{s}\left(  t\right)  \subseteq\left(  r_{gen}\setminus m\right)  \cap
m_{s}\left(  t\right)  $. Let $L\left(  s\right)  $ such that $d^{-1}\left(
t\cup u_{s}\left(  t\right)  \right)  <L\left(  s\right)  $ for all $t\in
T_{m}\cap split\left(  p\left(  s\right)  \right)  .$


Since $\mathcal{F}$ is $\mathcal{B}$-Canjar, we can find $f\in\mathcal{B}$ and
$\left(  z,F\right)  \in\mathbb{M}\left(  \mathcal{F}\right)  $ such that
$\left(  z,F\right)  \Vdash``f\notin K(\dot{L})\textquotedblright.$ We claim
that if $\mathsf{II}$ plays $f\left(  n\right)  $ at the $n$-step of the game,
then she will win the match. Let $q=%
{\textstyle\bigcup}
T_{m},$ we must show that $q\in\mathbb{PT}\left(  \mathcal{F}\right)  .$


Let $t\in q$ be a splitting node, find $n\in\omega$ such that $t\in T_{n}$
(where $T_{n}$ is defined as in the game). We must prove that $spsuc_{q}%
\left(  t\right)  \in\left(  \mathcal{F}^{<\omega}\right)  ^{+}.$ Let
$H\in\mathcal{F}$ and note that $\left(  z,F\cap H\right)  \Vdash``f\notin
K(\dot{L})\textquotedblright.$ We know we can find $\left(  z\cup
z_{0},G\right)  \in\mathbb{M}\left(  \mathcal{F}\right)  $ and $m\in\omega$
such that the following conditions hold:

\begin{enumerate}
\item $m>n,max\left(  z\right)  .$

\item $max\left(  z\right)  <min\left(  z_{0}\right)  .$

\item $\left(  z\cup z_{0},G\right)  \leq\left(  z,F\cap H\right)  $ (in
particular, $z_{0}\subseteq H$).

\item $\left(  z\cup z_{0},G\right)  \Vdash``\dot{L}\left(  f\upharpoonright
m\right)  <f\left(  m\right)  \textquotedblright.$
\end{enumerate}

\qquad\ \ 

It follows by all the definitions that $T_{m+1}$ will contain an element of
the set $spsuc_{p\left(  f\upharpoonright m\right)  }\left(  t\right)  .$ Note
that such element must be a subset of $z_{0},$ so in particular is a subset of
$H.$ This finishes the proof.
\end{proof}

\qquad\ \qquad\ \ \qquad\ \qquad\ \qquad\ \ 

The previous argument was motivated by the fact that Canjar is equivalent to
its \textquotedblleft game version\textquotedblright. This interesting result
is a corollary of the theorems proved by Chodounsk\'{y}, Repov\v{s} and
Zdomskyy in \cite{MathiasForcingandCombinatorialCoveringPropertiesofFilters},
we will comment more about this in the next section.

\begin{definition}
We say that $D\subseteq\mathbb{PT}\left(  \mathcal{F}\right)  $ is
\emph{purely dense} if the following conditions hold:

\begin{enumerate}
\item If $p\in D$ and $q\leq p$ then $q\in D$ ($D$ is open).

\item For every $p\in\mathbb{PT}\left(  \mathcal{F}\right)  $ and for every
finite tree $T\subseteq p$ such that $\left[  T\right]  \subseteq split\left(
p\right)  ,$ there is $q\leq_{T}p$ such that $q\in D.$
\end{enumerate}
\end{definition}

\qquad\qquad\ \ \ \ \ \qquad\qquad\qquad\ \ \ \ \ \ \qquad\ \ \qquad\ \ \ 

Intuitively, the purely dense sets are the open sets we can get in by only
using the pure decision property.

\qquad\ \ \ \ 

We will now prove that if $\mathcal{F}$ is a Canjar filter, then
$\mathbb{PT}\left(  \mathcal{F}\right)  $ does not add Cohen or random reals.
Recall the following notion:

\begin{definition}
We say that $c\in\omega^{\omega}$ is a \emph{half-Cohen real over }$V$ if for
every $f\in\omega^{\omega}\cap V$ the set $\left\{  n\mid c\left(  n\right)
=f\left(  n\right)  \right\}  $ is infinite.
\end{definition}

\qquad\qquad\ \ 

Obviously every Cohen real over $V$ is half-Cohen over $V.$ It can be proved
that if one adds an unbounded real and then a half-Cohen, a Cohen real is
added (see \cite{HandbookBlass}). It was a long standing question of Fremlin
if it was possible to add a half-Cohen real without adding a Cohen real. This
problem was finally solved positively by Zapletal in \cite{DimensionalForcing}%
. We will prove that if $\mathcal{F}$ is Canjar, then $\mathbb{PT}\left(
\mathcal{F}\right)  $ does not add a half-Cohen real. We start with the
following lemma:

\begin{lemma}
Let $\mathcal{F}\ $be a Canjar filter and$\ \dot{a}$ a $\mathbb{PT}\left(
\mathcal{F}\right)  $-name for a natural number. The set $D=\left\{
p\in\mathbb{PT}\left(  \mathcal{F}\right)  \mid\exists k\in\omega\left(
p\Vdash``\dot{a}\neq k\textquotedblright\right)  \right\}  $ is purely dense.
\end{lemma}

\begin{proof}
Let $p\in\mathbb{PT}\left(  \mathcal{F}\right)  $ and $T\subseteq p$ a finite
tree such that every maximal node of $T$ is an element of $split\left(
p\right)  .$ We need to find $q\leq_{T}p$ such that $q\in D.$ Let
$n=\left\vert T\cap split\left(  p\right)  \right\vert $ and for every $s\in
T\cap split\left(  p\right)  $ let $p\left(  s\right)  \leq p_{s}$ be the
biggest subtree of $p_{s}$ such that every $t\in T\cap p\left(  s\right)  $ is
a restriction of $s.$ Using the pure decision property, for every $s\in T\cap
split\left(  p\right)  $ we can find $q\left(  s\right)  \in\mathbb{PT}\left(
\mathcal{F}\right)  $ and $k_{s}\in\omega$ with the following properties:

\begin{enumerate}
\item $q\left(  s\right)  \leq p\left(  s\right)  $ and the stem of $q\left(
s\right)  $ is $s.$

\item $q\left(  s\right)  \Vdash``\dot{a}\leq n+1\textquotedblright$ or
$q\left(  s\right)  \Vdash``\dot{a}>n+1\textquotedblright$ .

\item If $q\left(  s\right)  \Vdash``\dot{a}\leq n+1\textquotedblright$ then
$q\left(  s\right)  \Vdash``\dot{a}=k_{s}\textquotedblright.$
\end{enumerate}

\qquad\ \ \ \qquad\ \ 

Since $n=\left\vert T\cap split\left(  p\right)  \right\vert ,$ we can find
$k\leq n+1$ such that $k\neq k_{s}$ for every $s\in T\cap split\left(
p\right)  .$ Let $q=%
{\textstyle\bigcup}
\left\{  q\left(  s\right)  \mid s\in T\cap split\left(  p\right)  \right\}
.$ Note that $q\leq_{T}p$ and $q\Vdash``\dot{a}\neq k\textquotedblright,$ this
finishes the proof.
\end{proof}

\qquad\ \ \ 

We can now prove the following:

\begin{proposition}
If $\mathcal{F}$ is a Canjar filter, then $\mathbb{PT}\left(  \mathcal{F}%
\right)  $ does not add half-Cohen reals.
\end{proposition}

\begin{proof}
Let $p\in\mathbb{PT}\left(  \mathcal{F}\right)  $ and $\dot{f}$ a
$\mathbb{PT}\left(  \mathcal{F}\right)  $-name for an element of
$\omega^{\omega}.$ We must prove that there is $q\leq p$ and $g\in
\omega^{\omega}$ such that $q\Vdash``\dot{f}\cap g=\emptyset\textquotedblright%
.$ For every $n\in\omega,$ define $D_{n}=\{q\in\mathbb{PT}\left(
\mathcal{F}\right)  \mid\exists k\in\omega(q\Vdash``\dot{f}\left(  n\right)
\neq k\textquotedblright)\},$ note that each $D_{n}$ is purely dense by the
previous lemma.


Let $\preceq$ be any well order of $\mathbb{PT}\left(  \mathcal{F}\right)  .$
We will recursively define a strategy $\sigma$ for player $\mathsf{I}$ on the
game $\mathcal{G}\left(  \mathcal{F},p,\omega^{\omega}\right)  $ as follows:

\begin{enumerate}
\item $\mathsf{I}$ starts by playing $p.$

\item Assume we are at round $m+1$ and $\mathsf{I}$ has played $p_{0},...,$
$p_{m}$ while player $\mathsf{II}$ has played $n_{0},...,n_{m-1}.$ If player
$\mathsf{II}$ plays $n_{m},$ then player $\mathsf{I}$ plays $p_{m+1}$ where
$p_{m+1}$ is the $\preceq$-least element of $\mathbb{PT}\left(  \mathcal{F}%
\right)  $ such that $p_{m+1}\in D_{m}$ and $p_{m+1}\leq_{T_{m}}p_{m}$ (where
$T_{m}=T\left(  p_{m},n_{m}\right)  $).
\end{enumerate}

\qquad\ \ \ 

Since $\mathcal{F}$ is Canjar, we know that $\sigma$ is not a winning strategy
for player $\mathsf{I,}$ which means that there is a function $d\in
\omega^{\omega}$ such that if player $\mathsf{II}$ plays $d\left(  n\right)  $
at the $n$-step then she will win. Let $q$ be the condition constructed at the
end of the game. Note that $q\in%
{\textstyle\bigcap\limits_{\in\omega}}
D_{n}$ in this way, we can define a function $g:\omega\longrightarrow\omega$
such that $q\Vdash``\dot{f}\left(  n\right)  \neq g\left(  n\right)
\textquotedblright$ for every $n\in\omega.$ This finishes the proof.
\end{proof}

\qquad\ \ \ 

We will now prove that if $\mathcal{F}$ is Canjar, then $\mathbb{PT}\left(
\mathcal{F}\right)  $ does not add bounded eventually different reals. By
$Fn\left(  \omega\right)  $ we will denote the set of all functions
$z:a\longrightarrow$ $\omega$ such that $a\in\left[  \omega\right]  ^{<\omega
}.$ We will need the following lemma:

\begin{lemma}
Let $m\in\omega,g\in\omega^{\omega}$ and $\dot{f}$ be a $\mathbb{PT}\left(
\mathcal{F}\right)  $-name for a function bounded by $g$. The set
$D=\{p\in\mathbb{PT}\left(  \mathcal{F}\right)  \mid\exists z\in Fn\left(
\omega\right)  (m\cap dom\left(  z\right)  =\emptyset\wedge p\Vdash``\dot
{f}\cap z\neq\emptyset\textquotedblright)\}$ is purely dense.
\end{lemma}

\begin{proof}
Let $p\in\mathbb{PT}\left(  \mathcal{F}\right)  $ and $T\subseteq p$ a finite
tree such that every maximal node of $T$ is an element of $split\left(
p\right)  .$ We need to find $q\leq_{T}p$ such that $q\in D.$ Let
$n=\left\vert T\cap split\left(  p\right)  \right\vert $ and for every $s\in
T\cap split\left(  p\right)  $ let $p\left(  s\right)  \leq p_{s}$ be the
biggest subtree of $p_{s}$ such that every $t\in T\cap p\left(  s\right)  $ is
a restriction of $s.$ Let $T\cap split\left(  p\right)  =\left\{  s_{i}\mid
i<n\right\}  .$ Using the pure decision property, for every $s_{i}\in T\cap
split\left(  p\right)  $ we can find $q\left(  s_{i}\right)  \in
\mathbb{PT}\left(  \mathcal{F}\right)  $ and $k_{i}\in\omega$ with the
following properties:

\begin{enumerate}
\item $q\left(  s_{i}\right)  \leq p\left(  s_{i}\right)  $ and the stem of
$q\left(  s_{i}\right)  $ is $s_{i}.$

\item $q\left(  s_{i}\right)  \Vdash``\dot{f}\left(  m+i\right)
=k_{i}\textquotedblright$.
\end{enumerate}

\qquad\ \ \ \qquad\ \ \qquad\ 

Note that $q\left(  s_{i}\right)  $ can be found (with the pure decision
property) since there are only finitely many possibilities for $\dot{f}\left(
m+i\right)  .$ We now define a function $z:[m,m+n)\longrightarrow\omega$ given
by $z\left(  m+i\right)  =k_{i}.$ Let $q=%
{\textstyle\bigcup}
\left\{  q\left(  s\right)  \mid s\in T\cap split\left(  p\right)  \right\}
.$ Note that $q\leq_{T}p$ and $q\Vdash``\dot{f}\cap z\neq\emptyset
\textquotedblright,$ this finishes the proof.
\end{proof}

\qquad\ \ \ 

We now can prove the following:

\begin{proposition}
If $\mathcal{F}$ is a Canjar filter, then $\mathbb{PT}\left(  \mathcal{F}%
\right)  $ does not add bounded eventually different reals.
\end{proposition}

\begin{proof}
Let $p\in\mathbb{PT}\left(  \mathcal{F}\right)  ,$ $g\in\omega^{\omega}$ and
$\dot{f}$ an $\mathbb{PT}\left(  \mathcal{F}\right)  $-name for an element of
$\omega^{\omega}$ bounded by $g.$ We must prove that there is $q\leq p$ and
$h\in\omega^{\omega}$ such that $q\Vdash``\left\vert \dot{f}\cap h\right\vert
=\omega\textquotedblright.$ For every $m\in\omega,$ define $D_{m}%
=\{p\in\mathbb{PT}\left(  \mathcal{F}\right)  \mid\exists z\in Fn\left(
\omega\right)  (m\cap dom\left(  z\right)  =\emptyset\wedge p\Vdash``\dot
{f}\cap z\neq\emptyset\textquotedblright)\},$ which we already know is a
purely dense set.


Let $\preceq$ be any well order of $\mathbb{PT}\left(  \mathcal{F}\right)  .$
We will recursively define a strategy $\sigma$ for player $\mathsf{I}$ on the
game $\mathcal{G}\left(  \mathcal{F},p,\omega^{\omega}\right)  $ as follows:

\begin{enumerate}
\item $\mathsf{I}$ starts by playing $p.$

\item Assume we are at round $m+1$ and $\mathsf{I}$ has played $p_{0},...,$
$p_{m}$ while player $\mathsf{II}$ has played $n_{0},...,n_{m-1}.$ If player
$\mathsf{II}$ plays $n_{m},$ then player $\mathsf{I}$ plays $p_{m+1}$ where
$p_{m+1}$ is the $\preceq$-least element of $\mathbb{PT}\left(  \mathcal{F}%
\right)  $ such that $p_{m+1}\in D_{m}$ and $p_{m+1}\leq_{T_{m}}p_{m}$ (where
$T_{m}=T\left(  p_{m},n_{m}\right)  $).
\end{enumerate}

\qquad\ \ \ 

Since $\mathcal{F}$ is Canjar, we know that $\sigma$ is not a winning strategy
for player $\mathsf{I,}$ hence there is a function $d\in\omega^{\omega}$ such
that if player $\mathsf{II}$ plays $d\left(  n\right)  $ at the $n$-step then,
she will win. Let $q$ be the condition constructed at the end of the game.
Note that $q\in%
{\textstyle\bigcap\limits_{\in\omega}}
D_{n}$, this means that for every $n\in\omega,$ we can find a finite function
$z$ such that $dom\left(  z\right)  =\emptyset$ and $q\Vdash``z\cap\dot{f}%
\neq\emptyset\textquotedblright.$ We can now easily find a function
$h:\omega\longrightarrow\omega$ such that $q\Vdash``\left\vert \dot{f}\cap
h\right\vert =\omega\textquotedblright.$
\end{proof}

\qquad\ \ \qquad\ \ \ \ 

It is well known that adding a random real adds a bounded eventually different
real. In this way we conclude the following:

\begin{corollary}
If $\mathcal{F}$ is a Canjar filter, then $\mathbb{PT}\left(  \mathcal{F}%
\right)  $ does not add random reals. \ 
\end{corollary}

\qquad\ \ \ \ \ 

We also have the following lemma:

\begin{lemma}
Let $T\subseteq\omega^{<\omega}$ be a finite tree, $p\in\mathbb{PT}\left(
\mathcal{F}\right)  $ such that $\left[  T\right]  \subseteq split\left(
p\right)  ,$ let $n=\left\vert split\left(  p\right)  \cap T\right\vert $ and
let $A$ be a finite set. If $\dot{x}$ is a $\mathbb{PT}\left(  \mathcal{F}%
\right)  $-name such that $p\Vdash``\dot{x}\in A\textquotedblright,$ there is
$q\leq_{T}p$ and $B\in\left[  A\right]  ^{n}$ such that $q\Vdash``\dot{x}\in
B\textquotedblright.$
\end{lemma}

\begin{proof}
The lemma easily follows by applying the pure decision property $n$-many times.
\end{proof}

\section{Preservation of $P$-points}

Let $\mathcal{U}$ be an ultrafilter and $\mathbb{P}$ a partial order. We say
that $\mathbb{P}$ \emph{preserves }$\mathcal{U}$ if $\mathcal{U}$ is the base
of an ultrafilter after forcing with $\mathbb{P}.$ It is well known that no
ultrafilter is preserved by Cohen, random, Silver or forcings adding a
dominating real. Also there is an ultrafilter that is destroyed by any forcing
that adds a new real (see \cite{Barty}). On the other hand, certain forcings
may preserve some ultrafilters, this is the case for Sacks and Miller
forcings. The preservation of $P$-points is particularly interesting in light
of the following theorem of Shelah (see \cite{ProperandImproper} or
\cite{Barty}):

\begin{proposition}
[Shelah]Let $\delta$ be a limit ordinal, $\langle\mathbb{P}_{\alpha
},\mathbb{\dot{Q}}_{\alpha}\mid\alpha<\delta\rangle$ a countable support
iteration of proper forcings and let $\mathcal{U}$ be a $P$-point. If
$\mathbb{P}_{\alpha}\Vdash``\mathbb{\dot{Q}}_{\alpha}$ preserves
$\mathcal{U}\textquotedblright$ for every $\alpha<\delta,$ then $\mathbb{P}%
_{\delta}$ preserves $\mathcal{U}.$
\end{proposition}

\qquad\ \ \ \ 

In particular, forcings that preserve $P$-points do not add Cohen or random
reals, even in the iteration. It is well known that Sacks forcing preserves
$P$-points and Miller forcing preserves an ultrafilter $\mathcal{U}$ if and
only if $\mathcal{U}$ is a $P$-point (see \cite{Rat}). Note that if
$\mathbb{P}$ diagonalizes $\mathcal{U},$ then $\mathbb{P}$ does not preserve
$\mathcal{U},$ however, it is possible to not preserve an ultrafilter without
diagonalizing it. For more on preservation of ultrafilters, the reader may
consult \cite{ProductsonInfinitelymanyTrees}, \cite{Rat}, \cite{Barty},
\cite{DiagonalizingHeike}, \cite{ForcingIdealized} and
\cite{PreservingPpointsinDefinableForcings}.

\qquad\ \ \ 

In light of the results of the previous section, it might be tempting to
conjecture that if $\mathcal{F}$ is a Canjar filter, then $\mathbb{PT}\left(
\mathcal{F}\right)  $ preserves $P$-points. However, this is not the case: the
simplest example is to take $\mathcal{U}$ a Canjar $P$-point, since
$\mathbb{PT}\left(  \mathcal{U}\right)  $ diagonalizes $\mathcal{U},$ it
follows that $\mathbb{PT}\left(  \mathcal{U}\right)  $ does not preserve
$\mathcal{U}.$ In this section, we will find a condition on a filter
$\mathcal{F}$ that guarantees preserving a certain $P$-point.

\qquad\ \ \ 

We will need the following result, which is a particular case of lemma
\ref{CanjarParticion}:

\begin{lemma}
Let $\mathcal{F}$ be a Canjar filter and $X\in\left(  \mathcal{F}^{<\omega
}\right)  ^{+}.$ There is $Y\subseteq X$ such that $Y\in\left(  \mathcal{F}%
^{<\omega}\right)  ^{+}$ and for every $n\in\omega,$ the set $\left\{  s\in
Y\mid s\cap n\neq\emptyset\right\}  $ is finite.
\end{lemma}

\qquad\ \ \ \ 

We can now prove the following:

\begin{lemma}
Let $\mathcal{F}$ be a Canjar filter, $p\in\mathbb{PT}\left(  \mathcal{F}%
\right)  $ and $c:split\left(  p\right)  \longrightarrow2.$ There is $q\leq p$
such that $split\left(  q\right)  $ is $c$-monochromatic. \label{RamseyMiller}
\end{lemma}

\begin{proof}
Assume there is no $q\leq p$ such that $split\left(  q\right)  $ is
$0$-monochromatic. We will prove that there is $q\leq p$ such that
$split\left(  q\right)  $ is $1$-monochromatic. Given $s\in split\left(
p\right)  ,$ we define $X\left(  s\right)  $ as the set of all $t\setminus s$
for which $t\in split\left(  p\right)  ,$ $s\subseteq t$ and $c\left(
t\right)  =1.$ We claim that $X\left(  s\right)  \in\left(  \mathcal{F}%
^{<\omega}\right)  ^{+}.$ Assume this is not the case, so there is
$A\in\mathcal{F}$ such that $A$ does not contain any element of $X\left(
s\right)  .$ Let $q\leq p_{s}$ such that if $t\in split\left(  q\right)  ,$
then $t\setminus s\subseteq A.$ It follows that if $t\in split\left(
q\right)  ,$ then $c\left(  t\right)  =0,$ so $split\left(  q\right)  $ is
$0$-monochromatic, which is a contradiction. We conclude that $X\left(
s\right)  \in\left(  \mathcal{F}^{<\omega}\right)  ^{+}$ for every $s\in
split\left(  p\right)  .$ By the previous result, we can find $Y\left(
s\right)  \subseteq X\left(  s\right)  $ such that $Y\left(  s\right)
\in\left(  \mathcal{F}^{<\omega}\right)  ^{+}$ and for every $n\in\omega,$ the
set $\left\{  z\in Y\mid z\cap n\neq\emptyset\right\}  $ is finite. The proof
follows by a simple fusion argument.
\end{proof}

\qquad\ \ \ \ \ \ \ \ \ \qquad\ \ \ \qquad\ \ \qquad\qquad\qquad\ \ \ 

Let $\mathcal{F}$ be a filter. The \emph{Canjar game} $\mathcal{G}%
_{Canjar}\left(  \mathcal{F}\right)  $ is defined as follows:

\qquad\ \ \ \qquad\ \ \qquad\qquad\ \ \ \ \ 

\begin{center}%
\begin{tabular}
[c]{|l|l|l|l|l|l|l|l|}\hline
$\mathsf{I}$ & $X_{0}$ &  & $X_{1}$ &  & $X_{2}$ &  & $...$\\\hline
$\mathsf{II}$ &  & $Y_{0}$ &  & $Y_{1}$ &  & $Y_{2}$ & \\\hline
\end{tabular}

\qquad\ \ \ \ \qquad\ \ \ \ \ \ \ 
\end{center}

Where $X_{i}\in\left(  \mathcal{F}^{<\omega}\right)  ^{+}$ and $Y_{i}%
\in\left[  X_{i}\right]  ^{<\omega}$ for every $i\in\omega.$ The player
$\mathsf{II}$ \emph{wins the game }$\mathcal{G}_{Canjar}\left(  \mathcal{F}%
\right)  $ if $%
{\textstyle\bigcup\limits_{n\in\omega}}
Y_{n}\in\left(  \mathcal{F}^{<\omega}\right)  ^{+}.$ In
\cite{MathiasForcingandCombinatorialCoveringPropertiesofFilters},
Chodounsk\'{y}, Repov\v{s} and Zdomskyy showed that sets in $\left(
\mathcal{F}^{<\omega}\right)  ^{+}$ naturally correspond to open covers of
$\mathcal{F}$ (viewed as a subspace of $\wp\left(  \omega\right)  $).
Moreover, they proved that a filter $\mathcal{F}$ is Canjar if and only if
$\mathcal{F}$ has the Menger property (see \cite{sel} for the definition of
Menger property). In this way, the Canjar game is just a particular case of
the Menger game that has been extensively studied in topology.

\begin{proposition}
[\cite{MathiasForcingandCombinatorialCoveringPropertiesofFilters}]Let
$\mathcal{F}$ be a filter. The following are equivalent:

\begin{enumerate}
\item $\mathcal{F}$ is Canjar.

\item $\mathcal{F}$ is Menger.

\item Player $\mathsf{I}$ does not have a winning strategy in $\mathcal{G}%
_{Canjar}\left(  \mathcal{F}\right)  .$
\end{enumerate}
\end{proposition}

\qquad\ \ \ \qquad\ \ \ 

We will now prove the following lemma:

\begin{lemma}
Let $\mathcal{F}$ be a Canjar filter, $p\in\mathbb{PT}\left(  \mathcal{F}%
\right)  ,$ $\dot{B}$ a $\mathbb{PT}\left(  \mathcal{F}\right)  $-name such
that $p\Vdash``\dot{B}\in\left[  \omega\right]  ^{\omega}\textquotedblright$
and $s\in split\left(  p\right)  .$ There are $q\in\mathbb{PT}\left(
\mathcal{F}\right)  ,$ $B_{s}\subseteq\omega$ and $\left\langle F_{n}%
\right\rangle _{n\in\omega}$ such that the following holds:
\label{nombrecompacto}

\begin{enumerate}
\item $q\leq p$ and $stem\left(  q\right)  =s.$

\item $F_{n}$ is a finite subset of $\left[  \omega\right]  ^{<\omega
}\setminus\left\{  \emptyset\right\}  $ for every $n\in\omega.$

\item For every $n\in\omega,$ if $t_{0}\in F_{n}$ and $t_{1}\in F_{n+1},$ then
$max\left(  t_{0}\right)  <min\left(  t_{1}\right)  .$

\item $spsuc_{q}\left(  s\right)  =%
{\textstyle\bigcup\limits_{n\in\omega}}
F_{n}.$

\item If $t\in F_{n},$ then $q_{s^{\frown}t}\Vdash``\dot{B}\cap\left(
n+1\right)  =B_{s}\cap\left(  n+1\right)  \textquotedblright.$
\end{enumerate}
\end{lemma}

\begin{proof}
We define $\sigma$ a strategy for player $\mathsf{I}$ in $\mathcal{G}%
_{Canjar}\left(  \mathcal{F}\right)  $ as follows:

\begin{enumerate}
\item Using the pure decision property, player $\mathsf{I}$ finds $q_{0}\leq
p$ with $stem\left(  q\right)  =s$ and $w_{0}$ such that $q^{0}\Vdash``\dot
{B}\cap1=w_{0}\textquotedblright$ and plays $X_{0}=spsuc_{q^{0}}\left(
s\right)  .$

\item Assume that player $\mathsf{II}$ plays $Y_{0}\in\left[  X_{0}\right]
^{<\omega}.$ Let $l_{0}\in\omega$ be the least such that $%
{\textstyle\bigcup}
Y_{0}\subseteq l_{0}$ and $Z_{0}=\left\{  t\in X_{0}\mid t\cap l_{0}%
=\emptyset\right\}  \in\left(  \mathcal{F}^{<\omega}\right)  ^{+}.$ Let
$\overline{q}^{1}\leq q^{0}$ be the condition defined as $\overline{q}^{1}=%
{\textstyle\bigcup\limits_{t\in Z_{0}}}
\left(  q^{0}\right)  _{s^{\frown}t}.$ Using the pure decision property,
player $\mathsf{I}$ finds $q^{1}\leq\overline{q}^{1}$ with $stem\left(
q\right)  =s$ and $w_{1}$ such that $q^{1}\Vdash``\dot{B}\cap2=w_{1}%
\textquotedblright$ and plays $X_{1}=spsuc_{q^{1}}\left(  s\right)  .$

\item In general, at step $n$, the player $\mathsf{I}$ has constructed a
decreasing sequence $\left\langle q^{i}\right\rangle _{i\leq n}$ where
$stem\left(  q^{n}\right)  =s$, an increasing sequence $\left\langle
w_{i}\right\rangle _{i\leq n}$ such that $q^{n}\Vdash``\dot{B}\cap\left(
n+1\right)  =w_{n}\textquotedblright$ has played the sequence $\left\langle
X_{i}\right\rangle _{i\leq n}$ where $X_{i}=spsuc_{q^{i}}\left(  s\right)  ,$
he has also constructed an increasing sequence $\left\langle l_{i}%
\right\rangle _{i\leq n}$ such that $Y_{i}\subseteq\lbrack l_{i-1},l_{i})$
(where $l_{-1}=0$ and $Y_{i}$ is the response of player $\mathsf{II}$ at round
$i$). Assume that player $\mathsf{II}$ plays $Y_{n}\in\left[  X_{n}\right]
^{<\omega}.$ Let $l_{n}\in\omega$ be the least such that $%
{\textstyle\bigcup}
Y_{n}\subseteq l_{n}$ and $Z_{n}=\left\{  t\in X_{n}\mid t\cap l_{n}%
=\emptyset\right\}  \in\left(  \mathcal{F}^{<\omega}\right)  ^{+}.$ Let
$\overline{q}^{n+1}\leq q^{n}$ be the condition defined as $\overline{q}%
^{n+1}=%
{\textstyle\bigcup\limits_{t\in Z_{n}}}
\left(  q^{n}\right)  _{s^{\frown}t}.$ Using the pure decision property,
player $\mathsf{I}$ finds $q^{n+1}\leq\overline{q}^{n}$ with $stem\left(
q\right)  =s$ and $w_{n+1}$ such that $q^{n+1}\Vdash``\dot{B}\cap\left(
n+2\right)  =w_{n+1}\textquotedblright$ and plays $X_{n+1}=spsuc_{q^{n+1}%
}\left(  s\right)  .$
\end{enumerate}


Since $\mathcal{F}$ is Canjar, we know that $\sigma$ is not a winning strategy
for the Canjar game. Consider a run in which player $\mathsf{II}$ defeated the
strategy. Let $q=%
{\textstyle\bigcup\limits_{t\setminus s\in Y_{i}}}
\left(  q^{i}\right)  _{s^{\frown}t}$ and note that $q\in\mathbb{PT}\left(
\mathcal{F}\right)  $ since player $\mathsf{II}$ won the game. Define
$F_{i}=Y_{i}$ and $B_{s}=%
{\textstyle\bigcup\limits_{i\in\omega}}
w_{i},$ it is clear that this are the items we were looking for.
\end{proof}

\qquad\ \ \ 

Let $\mathcal{U}$ be an ultrafilter. Recall that the $P$\emph{-point game
}$\mathcal{G}_{P\text{-}point}\left(  \mathcal{U}\right)  $ is defined as follows:

\qquad\qquad\ \qquad\ \ \ \ \ \ \qquad

\begin{center}%
\begin{tabular}
[c]{|l|l|l|l|l|l|}\hline
$\mathsf{I}$ & $W_{0}$ &  & $W_{1}$ &  & $...$\\\hline
$\mathsf{II}$ &  & $z_{0}$ &  & $z_{1}$ & \\\hline
\end{tabular}

\qquad\ \ \ 
\end{center}

Where $W_{i}\in\mathcal{U}$ and $z_{i}\in\left[  W_{i}\right]  ^{<\omega}.$
The player $\mathsf{II}$ will \emph{win the game }$\mathcal{G}_{P\text{-}%
point}\left(  \mathcal{U}\right)  $ if $%
{\textstyle\bigcup\limits_{m\in\omega}}
z_{m}\in\mathcal{U}.$ It is well known that player $\mathsf{II}$ can not have
a winning strategy for this game. The following is a well known result of
Galvin and Shelah (see \cite{Barty} for a proof):

\begin{proposition}
[Galvin, Shelah]Let $\mathcal{U}$ be an ultrafilter. The following are equivalent:

\begin{enumerate}
\item $\mathcal{U}$ is a $P$-point.

\item Player $\mathsf{I}$ does not have a winning strategy in $\mathcal{G}%
_{P\text{-}point}\left(  \mathcal{U}\right)  .$
\end{enumerate}
\end{proposition}

\ \ \ \ \qquad\qquad\ \ \ \ \qquad

Let $\mathcal{U}$ be an ultrafilter and $\mathcal{F}$ a filter. We will now
define the game $\mathcal{H}\left(  \mathcal{U},\mathcal{F}\right)  ,$ which
is a fusion between the $P$-point game and the game for $\mathbb{PT}\left(
\mathcal{F}\right)  .$ The game is defined as follows:

\qquad\ \ \ \qquad\ \ \qquad\qquad\ \ \ \ \ 

\begin{center}%
\begin{tabular}
[c]{|l|l|l|l|l|l|l|l|l|l|}\hline
$\mathsf{I}$ & $W_{0}$ &  & $p_{0}$ &  & $W_{1}$ &  & $p_{1}$ &  &
$...$\\\hline
$\mathsf{II}$ &  & $z_{0}$ &  & $n_{0}$ &  & $z_{1}$ &  & $n_{1}$ & \\\hline
\end{tabular}

\qquad\ \ 
\end{center}

Where the following conditions hold for every $i\in\omega:$

\begin{enumerate}
\item $W_{i}\in\mathcal{U}.$

\item $z_{i}\in\left[  W_{i}\right]  ^{<\omega}.$

\item $p_{i}\in\mathbb{PT}\left(  \mathcal{F}\right)  .$

\item $\left\langle n_{i}\right\rangle _{i\in\omega}$ is an increasing
sequence of natural numbers.

\item $p_{m+1}\leq_{T_{m}}p_{m}$ where $T_{m}=T\left(  p_{m},n_{m}\right)  .$
\end{enumerate}

\qquad\ \ \ 

The player $\mathsf{II}$ will \emph{win the game }$\mathcal{H}\left(
\mathcal{U},\mathcal{F}\right)  $ if $%
{\textstyle\bigcup\limits_{m\in\omega}}
T_{m}\in\mathbb{PT}\left(  \mathcal{F}\right)  $ and $%
{\textstyle\bigcup\limits_{m\in\omega}}
z_{m}\in\mathcal{U}.$

\begin{definition}
Let $\mathcal{F}$ be a filter and $\mathcal{U}$ an ultrafilter. We will say
that $\mathcal{F}$ is an $\mathcal{U}$\emph{-Canjar filter }if player
$\mathsf{I}$ has no winning strategy in $\mathcal{H}\left(  \mathcal{U}%
,\mathcal{F}\right)  .$
\end{definition}

Note that the previous notion is only of interest when
$\mathcal{F}$ is Canjar and $\mathcal{U}$ is a $P$-point. It is easy to see
that if $\mathcal{U}$ is an ultrafilter, then $\mathcal{U}$ is not
$\mathcal{U}$-Canjar (and will also follow by the next result). Our interest
in $\mathcal{U}$-Canjar filters is that (as we are about to prove), its Miller
forcing preserves $\mathcal{U}.$ Our proof is based on the argument of Miller
that the superperfect forcing preserves $P$-points (see \cite{Rat}).\qquad\ \ 

\begin{proposition}
If $\mathcal{U}$ is a $P$-point and $\mathcal{F}$ is an $\mathcal{U}$-Canjar
filter, then $\mathbb{PT}\left(  \mathcal{F}\right)  $ preserves $\mathcal{U}$.
\end{proposition}

\begin{proof}
Let $p\in\mathbb{PT}\left(  \mathcal{F}\right)  $ and $\dot{B}$ a
$\mathbb{PT}\left(  \mathcal{F}\right)  $-name such that $p\Vdash``\dot{B}%
\in\left[  \omega\right]  ^{\omega}\textquotedblright.$ By lemma
\ref{nombrecompacto}, we may assume that for every $s\in split\left(
p\right)  ,$ there is $B_{s}\subseteq\omega$ and $\left\langle F_{n}%
^{s}\right\rangle _{n\in\omega}$ with the following properties:

\begin{enumerate}
\item $F_{n}^{s}$ is a finite subset of $\left[  \omega\right]  ^{<\omega
}\setminus\left\{  \emptyset\right\}  $ for every $n\in\omega.$

\item For every $n\in\omega,$ if $t_{0}\in F_{n}^{s}$ and $t_{1}\in
F_{n+1}^{s},$ then $max\left(  t_{0}\right)  <min\left(  t_{1}\right)  .$

\item $spsuc_{p}\left(  s\right)  =%
{\textstyle\bigcup\limits_{n\in\omega}}
F_{n}^{s}.$

\item If $t\in F_{n}^{s},$ then $p_{s^{\frown}t}\Vdash``\dot{B}\cap\left(
n+1\right)  =B_{s}\cap\left(  n+1\right)  \textquotedblright.$
\end{enumerate}

\qquad\ \ \ \ 

Furthermore, by the lemma \ref{RamseyMiller} we may assume that either
$B_{s}\in\mathcal{U}$ for all $s\in split\left(  p\right)  $ or $B_{s}%
\in\mathcal{U}^{\ast}$ for all $s\in split\left(  p\right)  .$ We will assume
that $B_{s}\in\mathcal{U}$ for all $s\in split\left(  p\right)  $ (in the
other case we work with $\omega\setminus\dot{B})$). Let $s_{0}$ be the stem of
$p.$ We will define a strategy $\sigma$ for player $\mathsf{I}$ in
$\mathcal{H}\left(  \mathcal{U},\mathcal{F}\right)  $ as follows:

\begin{enumerate}
\item $\mathsf{I}$ starts by playing $W_{0}=B_{s_{0}}.$

\item Assume that player $\mathsf{II}$ plays $z_{0}\in\left[  W_{0}\right]
^{<\omega}.$ Let $l_{0}=max\left(  z_{0}\right)  ,$ player $\mathsf{I}$ will
play $p_{0}=%
{\textstyle\bigcup}
\left\{  p_{s_{0}^{\frown}t}\mid t\in F_{i}^{s_{0}}\wedge i>l_{0}\right\}  .$
Note that $p_{0}\Vdash``z_{0}\subseteq\dot{B}\textquotedblright.$

\item Assume that player $\mathsf{II}$ plays $n_{0}\in\omega.$ Now, the player
$\mathsf{I}$ will play the set $W_{1}=%
{\textstyle\bigcap}
\left\{  B_{s}\setminus l_{0}\mid s\in split\left(  p_{0}\right)  \cap
T_{0}\right\}  $ (where $T_{0}=T\left(  p_{0},n_{0}\right)  $).

\item In general, lets assume in the game it has already been played the
sequence $\left\langle W_{0},z_{0},p_{0},n_{0},W_{1}...W_{m},z_{m}%
,p_{m}\right\rangle .$ At the same time, player $\mathsf{I}$ has been making
sure that the sequence $\left\langle p_{i}\right\rangle _{i\leq m}$ has the
following properties:

\begin{enumerate}
\item $p_{i+1}\leq_{T_{i}}p_{i}$ (where $T_{i}=T\left(  p_{i},n_{i}\right)  $)
for all $i<m.$

\item $p_{i}\Vdash``z_{i}\subseteq\dot{B}\textquotedblright$ for all $i\leq
m.$

\end{enumerate}
\end{enumerate}

Now, assume that player $\mathsf{II}$ plays $n_{m}\in\omega.$ Player
$\mathsf{I}$ proceeds to play $W_{m+1}=%
{\textstyle\bigcap}
\left\{  B_{s}\setminus l_{m}\mid s\in split\left(  p_{m}\right)  \cap
T_{m}\right\}  $ where $l_{m}=max\left(  z_{m}\right)  .$ Assume player
$\mathsf{II}$ responds with $z_{m+1}\in\left[  W_{m+1}\right]  ^{<\omega}.$
For every $t\in T_{m}\cap split\left(  p_{m}\right)  ,$ let $p_{m}^{t}\leq
p_{m}$ be the biggest subtree such that $stem\left(  p_{m}^{t}\right)  =t$ and
every node in $p_{m}^{t}\cap T_{m}$ is contained in $t.$ Let $q^{t}=%
{\textstyle\bigcup}
\left\{  (p_{m}^{t})_{x}\mid x\in F_{i}^{t}\wedge i>l_{m+1}\right\}  $ (where
$l_{m+1}=max\left(  z_{m+1}\right)  $) and now player $\mathsf{I}$
counterattacks with $p_{m+1}=%
{\textstyle\bigcup}
\left\{  q^{t}\mid t\in T_{m}\cap split\left(  p_{m}\right)  \right\}  .$ It
is easy to see that $p_{m+1}\leq_{T_{m}}p_{m}$ and $p_{m+1}\Vdash
``z_{m+1}\subseteq\dot{B}\textquotedblright.$


Since $\mathcal{F}$ is $\mathcal{U}$-Canjar, we know that $\sigma$ is not a
winning strategy. Consider a run of the game where player $\mathsf{I}$
followed the strategy $\sigma$, but player $\mathsf{II}$ was the winner. In
this way, we know that $U=%
{\textstyle\bigcup\limits_{i\in\omega}}
z_{i}\in\mathcal{U}$ and $q=%
{\textstyle\bigcup\limits_{i\in\omega}}
T_{i}$ is a condition of $\mathbb{PT}\left(  \mathcal{F}\right)  .$ By
construction, it follows that $q\Vdash``U\subseteq\dot{B}\textquotedblright$
and we are done.
\end{proof}

\section{A model of $\omega_{1}=\mathfrak{u<s}$}

\qquad\qquad\ \ \ 

In order to increase the splitting number, it is enough to diagonalize an
ultrafilter, and to preserve the ultrafilter number, it is enough preserve a
$P$-point. In this way, in order to construct a model of $\mathfrak{u<s}$ it
is enough to find a $P$-point $\mathcal{W}$ and an ultrafilter $\mathcal{U}$
that is $\mathcal{W}$-Canjar. In this situation, we will have that
$\mathbb{PT}\left(  \mathcal{W}\right)  $ adds an unsplit real while
preserving $\mathcal{W}.$ In this section, we will use our results to build a
model of $\mathfrak{u<s}.$ This result is not new, as it already holds in the
Blass-Shelah model (see \cite{BlassShelah} or \cite{Barty}). At least in the
opinion of the authors, the combinatorics involved in our forcing are simpler
than the ones from the Blass-Shelah forcing.

\qquad\ \ \ 

We will first focus on constructing a $\mathcal{B}$-Canjar ultrafilter for
some $\mathfrak{b}$-family $\mathcal{B}.$ Such ultrafilters can either be
constructed under the Continuum Hypothesis or forced with a $\sigma$-closed
forcing (see \cite{MobandMAD}, \cite{BrendleRaghavan}, \cite{BrendleTaylor} or
\cite{CanjarFiltersII}). This follows by the result
of Shelah and the decomposition representation of Brendle and Raghavan. In
\cite{CanjarFiltersII} the first author, Michael Hru\v{s}\'{a}k and Arturo
Antonio Martinez Celis-Rodriguez published a proof of the consistency of
$\mathfrak{b<s}$ and $\mathfrak{b<a}$ using directly the representation of
Brendle and Raghavan. This section and the following borrows some of the
arguments from \cite{CanjarFiltersII}.

\qquad\ \qquad\ \ \ \ \ \qquad\ \ \ \ \ \qquad\ 

Given $X$ a collection of finite non-empty subsets of $\omega$, we define
$\mathcal{C}\left(  X\right)  =\left\{  A\subseteq\omega\mid\forall s\in
X\left(  s\cap A\neq\emptyset\right)  \right\}  .$ The following lemma
contains some of the combinatorial properties of compact sets that we will
need:\qquad\ \qquad\ \qquad\qquad\ 

\begin{lemma}
Let $\mathcal{F}$ be a filter, $\mathcal{D}\subseteq\mathcal{F}$ be a compact
set and $X\in\left(  \mathcal{F}^{<\omega}\right)  ^{+}$.
\label{Lema compactos}

\begin{enumerate}
\item $\mathcal{C}\left(  X\right)  $ is a compact set.

\item There is $Y\in\left[  X\right]  ^{<\omega}$ such that for every
$A\in\mathcal{D}$ there is $s\in Y$ such that $s\subseteq A.$

\item If $\mathcal{C}\left(  X\right)  \subseteq\mathcal{F}$ then for every
$n\in\omega$ there is $S\in\left[  X\right]  ^{<\omega}$ such that if
$A_{0},...,A_{n}\in\mathcal{C}\left(  S\right)  $ and $F\in\mathcal{D}$ then
$A_{0}\cap...\cap A_{n}\cap F\neq\emptyset.$

\item If $\mathcal{U}$ is an ultrafilter and $Y\subseteq\left[  \omega\right]
^{<\omega}$\textsf{\ }then $Y\in\left(  \mathcal{U}^{<\omega}\right)  ^{+}$ if
and only if $\mathcal{C}\left(  Y\right)  \subseteq\mathcal{U}.$
\end{enumerate}
\end{lemma}

\begin{proof}
For item 1, it is easy to see that $\mathcal{C}\left(  X\right)  $ is a closed
subset of $\wp\left(  \omega\right)  .$ We now prove item 2, let $X\in\left(
\mathcal{F}^{<\omega}\right)  ^{+}$ and $\mathcal{D}\subseteq\mathcal{F}$ a
compact set. For every $s\in X$, we define $U\left(  s\right)  =\left\{  A\mid
s\subseteq A\right\}  .$ It is easy to see that $U\left(  s\right)  $ is an
open set and $\left\{  U\left(  s\right)  \mid s\in X\right\}  $ is an open
cover for $\mathcal{D}$ (because $X\in\left(  \mathcal{F}^{<\omega}\right)
^{+}$). Since $\mathcal{D}$ is compact, there is $Y\in\left[  X\right]
^{<\omega}$ such that $\left\{  U\left(  s\right)  \mid s\in Y\right\}  $ is
an open cover for $\mathcal{D}.$ Clearly $Y$ is the set we were looking for.


We now prove 3, let $\mathcal{C}\left(  X\right)  \subseteq\mathcal{F}$ and
$n\in\omega.$ Given $s\in X$ define $K\left(  s\right)  $ as the set of all
$\left(  A_{0},\ldots,A_{n}\right)  \in\mathcal{C}\left(  s\right)  ^{n+1}$
with the property that there is $F\in\mathcal{D}$ such that $A_{0}\cap
\ldots\cap A_{n}\cap F=\emptyset.$ It is easy to see that $K\left(  s\right)
$ is a compact. Note that if $\left(  A_{0},\ldots,A_{n}\right)  \in
\bigcap\limits_{s\in X}K\left(  s\right)  $ then $A_{0},\ldots,A_{n}%
\in\mathcal{C}\left(  X\right)  \subseteq\mathcal{F}$ and there would be
$F\in\mathcal{D\subseteq F}$ such that $A_{0}\cap\ldots\cap A_{n}\cap
F=\emptyset$ which is clearly a contradiction. Since the $K\left(  s\right)  $
are compact, then there must be $S\in\left[  F\right]  ^{<\omega}$ such that
$\bigcap\limits_{s\in S}K\left(  s\right)  =\emptyset.$ It is easy to see that
this is the $S$ we are looking for.


We now prove 4. Let $\mathcal{U}$ be an ultrafilter and $Y\subseteq\left[
\omega\right]  ^{<\omega}.$ We will prove that $Y\notin\left(  \mathcal{U}%
^{<\omega}\right)  ^{+}$ if and only if $\mathcal{C}\left(  Y\right)
\nsubseteq\mathcal{U}.$ First assume that $Y\notin\left(  \mathcal{U}%
^{<\omega}\right)  ^{+},$ this means that there is $A\in\mathcal{U}$ that does
not contain any element of $Y,$ so $B=\omega\setminus A$ intersects every
element of $Y$, hence $B\in\mathcal{C}\left(  Y\right)  $ which implies that
$\mathcal{C}\left(  Y\right)  \nsubseteq\mathcal{U}.$ Now assume that
$\mathcal{C}\left(  Y\right)  \nsubseteq\mathcal{U},$ so there is
$B\in\mathcal{C}\left(  Y\right)  $ such that $B\notin\mathcal{U},$ hence
$A=\omega\setminus B\in\mathcal{U}.$ Since $B\in\mathcal{C}\left(  Y\right)
,$ this implies that $A$ does not contain any element of $Y,$ so
$Y\notin\left(  \mathcal{U}^{<\omega}\right)  ^{+}.$ \ \ 
\end{proof}

\qquad\ \ 

We will need the following notion:

\begin{definition}
Let $\mathcal{I}$ be an ideal on $\omega.$ We define $\mathbb{F}_{\sigma
}\left(  \mathcal{I}\right)  $ as the collection of all $F_{\sigma}$-filters
$\mathcal{F}$ such that $\mathcal{F}\cap\mathcal{I}=\emptyset.$ We order
$\mathbb{F}_{\sigma}\left(  \mathcal{I}\right)  $ by inclusion.
\end{definition}

\qquad\ \ \ 

Note that an $F_{\sigma}$-filter $\mathcal{F}$ is in $\mathbb{F}_{\sigma
}\left(  \mathcal{I}\right)  $ if and only if $\mathcal{F}\cup\mathcal{I}%
^{\ast}$ generates a filter. The following are some properties of this types
of forcings:

\begin{lemma}
Let $\mathcal{I}$ be an ideal on $\omega.$

\begin{enumerate}
\item $\mathbb{F}_{\sigma}\left(  \mathcal{I}\right)  $ is a $\sigma$-closed forcing.

\item $\mathbb{F}_{\sigma}\left(  \mathcal{I}\right)  $ adds an ultrafilter
(which we will denote by $\mathcal{U}_{gen}\left(  \mathcal{I}\right)  $)
disjoint from $\mathcal{I}.$

\item $\mathbb{F}_{\sigma}\left(  \mathcal{I}\right)  \ast\mathbb{PT}%
\mathcal{(\dot{U}}_{gen}(\mathcal{I}))$and $\mathbb{F}_{\sigma}\left(
\mathcal{I}\right)  \ast\mathbb{M}\mathcal{(\dot{U}}_{gen}(\mathcal{I}))$ are
proper forcings that destroy $\mathcal{I}.$
\end{enumerate}
\end{lemma}

\qquad\ \ \ 

If $\mathcal{A}$ is a \textsf{MAD} family, we will denote $\mathbb{F}_{\sigma
}\left(  \mathcal{A}\right)  $ instead of $\mathbb{F}_{\sigma}\left(
\mathcal{I}\left(  \mathcal{A}\right)  \right)  $ and $\mathcal{U}%
_{gen}\left(  \mathcal{A}\right)  $ instead of $\mathcal{U}_{gen}\left(
\mathcal{I}\left(  \mathcal{A}\right)  \right)  .$ Note that $\mathbb{F}%
_{\sigma}\left(  \left[  \omega\right]  ^{<\omega}\right)  $ is the collection
of all $F_{\sigma}$-filters. In this case, we will only denote it by
$\mathbb{F}_{\sigma}$ and by $\mathcal{U}_{gen}$ we will denote the generic
ultrafilter added by $\mathbb{F}_{\sigma}.$ The following lemma is easy and
left to the reader:

\begin{lemma}
If $\mathcal{I}$ is an ideal, $X\subseteq\left[  \omega\right]  ^{<\omega}$
and $\mathcal{F}\in\mathbb{F}_{\sigma}\left(  \mathcal{I}\right)  ,$ then
$\mathcal{F\Vdash}``X\in(\mathcal{\dot{U}}_{gen}\left(  \mathcal{I}\right)
^{<\omega}\mathcal{)}^{+}\textquotedblright$ if and only if $\mathcal{C}%
\left(  X\right)  \subseteq\left\langle \mathcal{F\cup I}^{\ast}\right\rangle
$ (where $\left\langle \mathcal{F\cup I}^{\ast}\right\rangle $ is the filter
generated by $\mathcal{F\cup I}^{\ast}$). \qquad\ \ 
\end{lemma}

\qquad\ \ \ 

In particular, if $\mathcal{F}\in\mathbb{F}_{\sigma}$ and $X\subseteq\left[
\omega\right]  ^{<\omega},$ then $\mathcal{F\Vdash}``X\in(\mathcal{\dot{U}%
}_{gen}{}^{<\omega}\mathcal{)}^{+}\textquotedblright$ if and only if
$\mathcal{C}\left(  X\right)  \subseteq\mathcal{F}.$ We will now prove the following:

\begin{proposition}
If $\mathcal{B}\in V$ is a $\mathfrak{b}$-family, then $\mathbb{F}_{\sigma}$
forces that $\mathcal{\dot{U}}_{gen}$ is $\mathcal{B}$-Canjar.
\label{UgenesCanajr}
\end{proposition}

\begin{proof}
By the previous observation and since $\mathbb{F}_{\sigma}$ is $\sigma
$-closed, it is enough to show that if $\mathcal{F}\Vdash``\overline
{X}=\left\langle X_{n}\right\rangle _{n\in\omega}\subseteq(\mathcal{\dot{U}%
}_{gen}^{<\omega})^{+}\textquotedblright$ then there is $\mathcal{G\leq F}$
and $f\in\mathcal{B}$ such that $\mathcal{C}\left(  \overline{X}_{f}\right)
\subseteq\mathcal{G}.$


Let $\mathcal{F=}\bigcup\mathcal{C}_{n}$ where each $\mathcal{C}_{n}$ is
compact and they form an increasing chain. By a previous lemma, there is
$g:\omega\longrightarrow\omega$ such that if $n\in\omega,$ $F\in
\mathcal{C}_{n}$ and $A_{0},\ldots,A_{n}\in\mathcal{C}\left(  X_{n}\cap
\wp\left(  g\left(  n\right)  \right)  \right)  $ then $A_{0}\cap\ldots.\cap
A_{n}\cap F\neq\emptyset.$ Since $\mathcal{B}$ is unbounded, then there is
$f\in\mathcal{B}$ such that $f\nleq^{\ast}g.$ We claim that $\mathcal{F\cup
C}\left(  \overline{X}_{f}\right)  $ generates a filter. Let $F\in
\mathcal{C}_{n}$ and $A_{0},\ldots,A_{m}\in\mathcal{C}\left(  \overline{X}%
_{f}\right)  .$ We must show that $A_{0}\cap\ldots.\cap A_{m}\cap
F\neq\emptyset.$ Since $f$ is not bounded by $g,$ we may find $r>n,m$ such
that $f\left(  r\right)  >g\left(  r\right)  .$ In this way, $A_{0}%
,\ldots,A_{n}\in\mathcal{C}\left(  X_{n}\cap\wp\left(  g\left(  n\right)
\right)  \right)  $ and then $A_{0}\cap\ldots.\cap A_{m}\cap F\neq\emptyset.$
Finally, we can define $\mathcal{G}$ as the filter generated by
$\mathcal{F\cup C}\left(  \overline{X}_{f}\right)  .$
\end{proof}

\qquad\ \ \ 

In this way, we conclude the following:

\begin{corollary}
The forcing $\mathbb{F}_{\sigma}\ast\mathbb{PT(}\mathcal{\dot{U}}_{gen})$ is
proper, adds an unsplit real, preserves all $\mathfrak{b}$-scales from the
ground model and does not add Cohen reals.
\end{corollary}

\qquad\ \ \ \qquad\ \ \ \qquad\ \ 

A forcing notion is called \emph{weakly }$\omega^{\omega}$\emph{-bounding }if
it does not add dominating reals. Unlike the $\omega^{\omega}$-bounding
property, the weakly $\omega^{\omega}$-bounding property is not preserved
under two step iteration (see \cite{AbrahamHandbookProper}). However, Shelah proved
the following preservation result:

\begin{proposition}
[Shelah, see \cite{AbrahamHandbookProper}]If $\gamma\leq\omega_{2}\ $is limit
and$\ \langle\mathbb{P}_{\alpha},\mathbb{\dot{Q}}_{\alpha}\mid\alpha\leq
\gamma\rangle$ is a countable support iteration of proper forcings and each
$\mathbb{P}_{\alpha}$ is weakly $\omega^{\omega}$-bounding (over $V$) then
$\mathbb{P}_{\gamma}$ is weakly $\omega^{\omega}$-bounding.
\end{proposition}

\qquad\ \ \ \ \ \ \ \ \ \ \qquad\ \ \ 

Note that $\mathbb{P}$ is weakly $\omega^{\omega}$-bounding if and only if it
preserves the unboundedness of all (one) dominating families. By applying the
result of Shelah we can easily conclude the following result.

\begin{corollary}
If $V$ satisfies \textsf{CH} and $\langle\mathbb{P}_{\alpha},\mathbb{\dot{Q}%
}_{\alpha}\mid\alpha\leq\omega_{2}\rangle$ is a countable support iteration of
proper forcings such that $\mathbb{P}_{\alpha}$ forces that $\mathbb{\dot{Q}%
}_{\alpha}$ preserves the unboundedness of all well-ordered unbounded
families, then $\mathbb{P}_{\omega_{2}}$ is weakly $\omega^{\omega}$-bounding.
\end{corollary}

\qquad\ \ \ \ \ \ \ \ \ \ \qquad\ \ \ \qquad\ \ \ 

With these results we can conclude the following result of Shelah:

\begin{proposition}
[Shelah]There is a model of $\omega_{1}=\mathfrak{b<s=}$ \textsf{cov}$\left(
\mathcal{M}\right)  =\mathfrak{c}=\omega_{2}.$
\end{proposition}

\begin{proof}
We perform a countable support iteration $\langle\mathbb{P}_{\alpha
},\mathbb{\dot{Q}}_{\alpha}\mid\alpha\leq\omega_{2}\rangle$ where
$\mathbb{P}_{\alpha}\Vdash``\mathbb{\dot{Q}}_{\alpha}=\mathbb{F}_{\sigma}%
\ast\mathbb{M}\mathcal{(\dot{U}}_{gen}\mathcal{)}\textquotedblright.$ The
result follows by the previous results.
\end{proof}

\qquad\ \ \ 
\ \ \qquad\ \ 

Recall that a forcing $\mathbb{P}$ is \emph{almost }$\omega^{\omega}%
$\emph{-bounding }if for every $\mathbb{P}$-name $\dot{f}$ for an element of
$\omega^{\omega}$ and $q\in\mathbb{P},$ there is $g\in\omega^{\omega}$ such
that for every $A\in\left[  \omega\right]  ^{\omega},$ there is $q_{A}\leq q$
such that $q_{A}\Vdash$\textquotedblleft$g\upharpoonright A\nleq^{\ast}\dot
{f}\upharpoonright A$\textquotedblright\ (the reader may consult
\mbox{
\cite{AbrahamHandbook} }\hspace{0pt}
to learn more about this property). It follows by the
results of Shelah and Brendle and Raghavan that the iterands in the previous
model are almost $\omega^{\omega}$-bounding, but it is not clear that this
follows from our approach. In
\mbox{
\cite{MathiasForcingandCombinatorialCoveringPropertiesofFilters} }\hspace{0pt}
it was proved
that a forcing $\mathbb{P}$ is almost $\omega^{\omega}$-bounding if and only
if it preserves every unbounded family of the ground model. The referee asked
us if preserving the unboundedness of all well-ordered unbounded families
imply the almost $\omega^{\omega}$-bounding property, we do not know the
answer to this question.

\qquad\ \ \ \ \qquad\ \ \

As was mentioned before, by the result of Brendle and Raghavan, $\mathbb{F}%
_{\sigma}\ast\mathbb{M}\mathcal{(\dot{U}}_{gen}\mathcal{)}$ is forcing
equivalent to the original creature forcing of Shelah for $\mathfrak{b<s}.$ We
will now prove that if we iterate $\mathbb{F}_{\sigma}\ast\mathbb{PT}%
\mathcal{(\dot{U}}_{gen}\mathcal{)},$ we will get a model of $\mathfrak{u<s}.$
Although we will not need the following result, it is illustrative to prove
first the following:

\begin{proposition}
If $\mathcal{U}$ is a $P$-point and $\mathcal{F}$ is an $F_{\sigma}$-filter,
then $\mathcal{F}$ is $\mathcal{U}$-Canjar. \label{FsigmaUCanjar}
\end{proposition}

\begin{proof}
Let $\mathcal{U}$ be a $P$-point and $\mathcal{F=}%
{\textstyle\bigcup\limits_{n\in\omega}}
\mathcal{C}_{n}$ be an $F_{\sigma}$-filter, where $\left\langle \mathcal{C}%
_{n}\right\rangle _{n\in\omega}$ is an increasing sequence of compact sets. We
will argue by contradiction, so assume that $\mathcal{F}$ is not
$\mathcal{U\,}$-Canjar, i.e. player $\mathsf{I}$ has a winning strategy for
the game $\mathcal{H}\left(  \mathcal{U},\mathcal{F}\right)  ,$ call $\sigma$
such strategy. We will use $\sigma$ to construct a winning strategy for
$\mathsf{I}$ in the $P$-point game, which will obviously entail a contradiction.


Given $X\in\left(  \mathcal{F}^{<\omega}\right)  ^{+}$ and $n\in\omega,$
choose $Y\left(  X,n\right)  \in\left[  X\right]  ^{<\omega}$ such that every
element of $\mathcal{C}_{n}$ contains an element of $Y\left(  X,n\right)  $
(which is possible by lemma \ref{Lema compactos}). We now define $\pi$ a
strategy for player $\mathsf{I}$ in $\mathcal{G}_{P\text{-}point}\left(
\mathcal{U}\right)  $ as follows:

\begin{enumerate}
\item Player $\mathsf{I}$ starts by playing $W_{0}=\sigma\left(
\emptyset\right)  $ (i.e. $W_{0}$ is the first play in the game $\mathcal{H}%
\left(  \mathcal{U},\mathcal{F}\right)  $).

\item Assume player $\mathsf{II}$ plays $z_{0}\in\left[  W_{0}\right]
^{<\omega}$ as her response in $\mathcal{H}\left(  \mathcal{U},\mathcal{F}%
\right)  .$ Let $p_{0}=\sigma\left(  W_{0},z_{0}\right)  $ \ and $s_{0}$ be
the stem of $p_{0}.$ Define $n_{0}>d^{-1}\left(  s_{0}\right)  $ to be the
least integer such that$\ d^{-1}(s_{0}{}^{\frown}t)<n_{0}$ for all $t\in
Y\left(  spsuc_{p_{0}}\left(  s_{0}\right)  ,0\right)  .$ Player $\mathsf{I}$
will play (in $\mathcal{G}_{P\text{-}point}\left(  \mathcal{U}\right)  $)
$W_{1}=\sigma\left(  W_{0},z_{0},p_{0},n_{0}\right)  $ (i.e. his response in
$\mathcal{H}\left(  \mathcal{U},\mathcal{F}\right)  $ if player $\mathsf{II}$
had played $n_{0}$).

\item In general assume that it has been played the sequence $\left\langle
W_{0},z_{0},...,W_{m}\right\rangle .$ At the same time, in secret the player
$\mathsf{I}$ has been constructed a partial play $\left\langle W_{0}%
,z_{0},p_{0},n_{0},W_{1},z_{1},p_{1},n_{1}...,W_{m}\right\rangle $ in the game
$\mathcal{H}\left(  \mathcal{U},\mathcal{F}\right)  $ following $\sigma$ such
that for every $i<m,$ the integer $n_{i}$ has the following property: for
every $u\in T\left(  p_{i},n_{i-1}\right)  $ (where $n_{-1}=d^{-1}\left(
s_{0}\right)  $) and for every $t\in Y\left(  spsuc_{p_{i}}\left(  u\right)
,i\right)  ,$ we have that $d^{-1}(u^{\frown}t)<n_{i}.$ Assume that player
$\mathsf{II}$ plays $z_{m}$ as her next response in $\mathcal{G}%
_{P\text{-}point}\left(  \mathcal{U}\right)  .$ Let $p_{m}$ be the tree
defined as $\sigma\left(  W_{0},z_{0},n_{0},W_{1},...,W_{m},z_{m}\right)  $
and let $n_{m}>n_{m-1}$ the least integer with the following property: for
every $u\in T\left(  p_{m},n_{m-1}\right)  $ and for every $t\in Y\left(
spsuc_{p_{m}}\left(  u\right)  ,m\right)  ,$ we have that $d^{-1}(u^{\frown
}t)<n_{m}.$ Player $\mathsf{I}$ will play the set $W_{m+1}$ that is defined as
$\sigma\left(  W_{0},z_{0},n_{0},W_{1},...,W_{m},z_{m},p_{m},n_{m}\right)  .$
\end{enumerate}

\qquad\ \ \ \qquad\ \ \ \ \ \ \ \ \ \ \ \ \ \ \ \qquad
\ \ \ \ \ \ \ \ \ \ \ \ \ \ \ \ \ \ \qquad\ \ 

The game $\mathcal{G}_{P\text{-}point}\left(  \mathcal{U}\right)  :$

\qquad\ \ %

\begin{tabular}
[c]{|l|l|l|l|l|l|}\hline
$\mathsf{I}$ & $W_{0}$ &  & $W_{1}$ &  & $...$\\\hline
$\mathsf{II}$ &  & $z_{0}$ &  & $z_{1}$ & \\\hline
\end{tabular}

\qquad\ \ \ \ \qquad\ \ \ \qquad\ \ \ \newline

\ The game $\mathcal{H}\left(  \mathcal{U},\mathcal{F}\right)  :$

\qquad\ \ \ \qquad\ \ \ %

\begin{tabular}
[c]{|l|l|l|l|l|l|l|l|l|l|}\hline
$\mathsf{I}$ & $W_{0}$ &  & $p_{0}$ &  & $W_{1}$ &  & $p_{1}$ &  &
$...$\\\hline
$\mathsf{II}$ &  & $z_{0}$ &  & $n_{0}$ &  & $z_{1}$ &  & $n_{1}$ & \\\hline
\end{tabular}

\qquad\ \ \ \ \ \qquad\ \ \ \ \qquad\ \ \ \bigskip

We claim that $\pi$ is a winning strategy for player $\mathsf{I}$ in
$\mathcal{G}_{P\text{-}point}\left(  \mathcal{U}\right)  .$ Consider a run of
the game in which player $\mathsf{I}$ played according to $\pi.$ Let $Z=%
{\textstyle\bigcup\limits_{n\in\omega}}
z_{n},$ we will prove that $Z\notin\mathcal{U}.$ Let $q=%
{\textstyle\bigcup\limits_{i\in\omega}}
T\left(  p_{i},n_{i}\right)  $ be the tree that was constructed during the
play. It is easy to see that $q\in\mathbb{PT}\left(  \mathcal{F}\right)  ,$
but since player $\mathsf{I}$ was following his winning strategy $\sigma$ in
the side game, we know that he won, so it must be the case that $Z\notin
\mathcal{U}.$ This shows that $\pi$ is a winning strategy for player
$\mathsf{I}$ in $\mathcal{G}_{P\text{-}point}\left(  \mathcal{U}\right)  .$
Since player $\mathsf{I}$ can not have a winning strategy in the $P$-point
game, we get a contradiction.
\end{proof}

\qquad\ \ \ 

We will now prove that the forcing $\mathbb{F}_{\sigma}\ast\mathbb{PT(}%
\mathcal{\dot{U}}_{gen}\mathcal{)}$ preserves all ground model $P$-points.
First we will need the following lemma, which is a slight generalization of
part of lemma \ref{Lema compactos}:\qquad\ 

\begin{lemma}
Let $\mathcal{F}$ be a filter, $\mathcal{D\subseteq F}$ a compact set and
$X_{1},...,X_{n}\subseteq\left[  \omega\right]  ^{<\omega}$ such that
$\mathcal{C}\left(  X_{1}\right)  ,...,\mathcal{C}\left(  X_{n}\right)
\subseteq\mathcal{F}.$ There are $Y_{1}\in\left[  X_{1}\right]  ^{<\omega
},...,Y_{n}\in\left[  X_{n}\right]  ^{<\omega}$ such that for every
$F\in\mathcal{D}$ and for every $A_{i}^{1},...,A_{i}^{n}\in\mathcal{C}\left(
Y_{i}\right)  $ (with $i\leq n$), we have that $F\cap%
{\textstyle\bigcap\limits_{i,j\leq n}}
A_{i}^{j}\neq\emptyset.$\label{generalizacioncompacto}
\end{lemma}

\begin{proof}
Consider the space $Z=(%
{\textstyle\prod\limits_{i=1}^{n}}
\wp\left(  \omega\right)  ^{n})\times\mathcal{D},$ which we know is compact.
Given $l\in\omega,$ let $K\left(  l\right)  $ be the set of all $(\left\langle
A_{i}^{1},...,A_{i}^{n}\right\rangle _{i\leq n},F)\in Z$ such that $A_{i}%
^{1},...,A_{i}^{n}\in\mathcal{C(}X_{i}\cap\wp\left(  l\right)  )$ (for every
$i\leq n$) and $F\cap%
{\textstyle\bigcap\limits_{i,j\leq n}}
A_{i}^{j}=\emptyset.$ Clearly $K\left(  l\right)  $ is a closed subspace.
Since $\mathcal{C}\left(  X_{1}\right)  ,...,\mathcal{C}\left(  X_{n}\right)
,\mathcal{D}\subseteq\mathcal{F}$, we conclude that $%
{\textstyle\bigcap\limits_{l\in\omega}}
K\left(  l\right)  =\emptyset,$ hence by the compactness of $Z,$ we conclude
that there is $l\in\omega$ such that $K\left(  l\right)  =\emptyset.$ Let
$Y_{i}=X_{i}\cap\wp\left(  l\right)  .$ It is clear that these are the sets we
were looking for.
\end{proof}

\qquad\ \ \ \qquad\ \ \ 

We can now prove the following result, which is a fusion of the proofs of
proposition \ref{FsigmaUCanjar} and proposition \ref{UgenesCanajr}%
:\qquad\ \ \ \qquad\ \ \ \ 

\begin{proposition}
If $\mathcal{W}$ is a $P$-point, then $\mathbb{F}_{\sigma}$ forces that
$\mathcal{\dot{U}}_{gen}$ is $\mathcal{W}$-Canjar. \label{UgenesWCanjar}
\end{proposition}

\begin{proof}
We will prove the proposition by contradiction. Assume there is $\mathcal{F}$
an $F_{\sigma}$-filter and $\sigma$ such that $\mathcal{F}$ forces that
$\sigma$ is a winning strategy for player $\mathsf{I}$ in
$\mathcal{H\mathbb{(}W},\mathcal{\dot{U}}_{gen})$. Note that strategies for
$\mathsf{I}$ are countable objects and since $\mathbb{F}_{\sigma}$ is $\sigma
$-closed, it is enough to consider ground model strategies. Let $\mathcal{F}=%
{\textstyle\bigcup\limits_{n\in\omega}}
\mathcal{C}_{n}$ where $\left\langle \mathcal{C}_{n}\right\rangle _{n\in
\omega}$ is an increasing sequence of compact sets. We will use $\sigma$ to
construct a winning strategy for $\mathsf{I}$ in the game $\mathcal{G}%
_{P\text{-}point}\left(  \mathcal{W}\right)  $, which will be a contradiction.


Note that if $p$ is a Miller tree such that $p$ is a possible response of
player $\mathsf{I}$ according to $\sigma$ and $s\in split\left(  p\right)  ,$
then $\mathcal{F}\Vdash``spsuc_{p}\left(  s\right)  \in(\mathcal{\dot{U}%
}_{gen}^{<\omega}\mathcal{)}^{+}\textquotedblright$ (this is because
$\mathcal{F}$ is forcing that $\sigma$ is a strategy for player $\mathsf{I,}$
which implies that $p$ must be a legal move), in particular $\mathcal{C}%
\left(  spsuc_{p}\left(  s\right)  \right)  \subseteq\mathcal{F}.$


For every $\mathcal{X}=\left\{  X_{1},...,X_{n}\right\}  $ such that
$X_{i}\subseteq$ $\left[  \omega\right]  ^{<\omega}\setminus\left\{
\emptyset\right\}  $ and $\mathcal{C}\left(  X_{i}\right)  \subseteq
\mathcal{F}$ for every $i\leq n$ and for every $k\in\omega,$ fix a function
$F_{\left(  \mathcal{X},k\right)  }:\mathcal{X\longrightarrow}$ $\left[
\left[  \omega^{<\omega}\right]  \right]  ^{<\omega}$ with the following properties:

\begin{enumerate}
\item $Y_{i}=F_{\left(  \mathcal{X},k\right)  }\left(  X_{i}\right)
\in\left[  X_{i}\right]  ^{<\omega}$ for every $i\leq n.$

\item For every $B\in\mathcal{C}_{k}$ and for every $A_{i}^{1},...,A_{i}%
^{n}\in\mathcal{C}\left(  Y_{i}\right)  $ (with $i\leq n$), we have that
$B\cap%
{\textstyle\bigcap\limits_{i,j\leq n}}
A_{i}^{j}\neq\emptyset.$
\end{enumerate}

\qquad\ \ 

We know such $F_{\left(  \mathcal{X},k\right)  }$ exists by lemma
\ref{generalizacioncompacto}. The proof now proceeds in a very similar way as
the proof of proposition \ref{FsigmaUCanjar}. We define $\pi$ a strategy for
player $\mathsf{I}$ in $\mathcal{G}_{P\text{-}point}\left(  \mathcal{W}%
\right)  $ as follows:

\begin{enumerate}
\item Player $\mathsf{I}$ starts by playing $W_{0}=\sigma\left(
\emptyset\right)  $.

\item Assume player $\mathsf{II}$ plays $z_{0}\in\left[  W_{0}\right]
^{<\omega}$. Let $p_{0}=\sigma\left(  W_{0},z_{0}\right)  $, $s_{0}$ be the
stem of $p_{0}$ and $\mathcal{X}_{0}=\left\{  spsuc_{p_{0}}\left(
s_{0}\right)  \right\}  .$ Define $n_{0}>d^{-1}\left(  s_{0}\right)  $ to be
the least integer such that$\ d^{-1}(s_{0}{}^{\frown}t)<n_{0}$ for all $t\in
F_{\left(  \mathcal{X}_{0},0\right)  }\left(  spsuc_{p_{0}}\left(
s_{0}\right)  \right)  $. Player $\mathsf{I}$ will play (in $\mathcal{G}%
_{P\text{-}point}\left(  \mathcal{W}\right)  $) $W_{1}=\sigma\left(
W_{0},z_{0},p_{0},n_{0}\right)  $.

\item In general assume that it has been played the sequence $\left\langle
W_{0},z_{0},...,W_{m}\right\rangle .$ At the same time, secretly the player
$\mathsf{I}$ has been constructing a sequence $\left\langle W_{0},z_{0}%
,p_{0},n_{0},W_{1},z_{1},p_{1},n_{1}...,W_{m}\right\rangle $ that is being
forced to be a partial play of the game $\mathcal{H(W},\mathcal{\dot{U}}%
_{gen}\mathcal{)}$ following $\sigma,$ such that for every $i<m,$ the integer
$n_{i}$ has the following property: letting $\mathcal{X}_{i}$ to be the set
defined as $\left\{  spsuc_{p_{i}}\left(  u\right)  \mid u\in T\left(
p_{i},n_{i-1}\right)  \right\}  $ (where $n_{i-1}=d^{-1}\left(  s_{0}\right)
$), for every $t\in F_{\left(  \mathcal{X}_{i},i\right)  }\left(
spsuc_{p_{i}}\left(  u\right)  \right)  ,$ we have that $d^{-1}(u^{\frown
}t)<n_{i}.$ Assume that player $\mathsf{II}$ plays $z_{m}$ as her next
response in $\mathcal{H}\left(  \mathcal{W},\mathcal{\dot{U}}_{gen}\right)  .$
Let $p_{m}$ be the tree given by $\sigma\left(  W_{0},z_{0},n_{0}%
,W_{1},...,W_{m},z_{m}\right)  $ and let $n_{m}>n_{m-1}$ be the least integer
with the following property: letting $\mathcal{X}_{m}=\left\{  spsuc_{p_{m}%
}\left(  u\right)  \mid u\in T\left(  p_{m},n_{m-1}\right)  \right\}  $, for
every $t\in F_{\left(  \mathcal{X}_{m},m\right)  }\left(  spsuc_{p_{m}}\left(
u\right)  \right)  ,$ we have that $d^{-1}(u^{\frown}t)<n_{m}.$ Player
$\mathsf{I}$ will play $W_{m+1}=\sigma\left(  W_{0},z_{0},n_{0},W_{1}%
,...,W_{m},z_{m},p_{m},n_{m}\right)  .$
\end{enumerate}

\qquad\qquad\ \ \ \qquad\ \ \ \ \ \qquad\ \ \ \ \ \ 

The game $\mathcal{G}_{P\text{-}point}\left(  \mathcal{W}\right)  :$

\qquad\ \ %

\begin{tabular}
[c]{|l|l|l|l|l|l|}\hline
$\mathsf{I}$ & $W_{0}$ &  & $W_{1}$ &  & $...$\\\hline
$\mathsf{II}$ &  & $z_{0}$ &  & $z_{1}$ & \\\hline
\end{tabular}

\qquad\ \ \ \qquad\ \ \ \ \qquad\ \ \ \ \qquad
\ \ \ \ \ \ \ \ \ \ \ \ \ \ \newline

The game $\mathcal{H(W},\mathcal{\dot{U}}_{gen}\mathcal{)}:$

\qquad\ \ \ \qquad\ \ \ %

\begin{tabular}
[c]{|l|l|l|l|l|l|l|l|l|l|}\hline
$\mathsf{I}$ & $W_{0}$ &  & $p_{0}$ &  & $W_{1}$ &  & $p_{1}$ &  &
$...$\\\hline
$\mathsf{II}$ &  & $z_{0}$ &  & $n_{0}$ &  & $z_{1}$ &  & $n_{1}$ & \\\hline
\end{tabular}

\qquad\ \ \ \ \ \qquad\ \ \ \qquad\ \ \ \qquad\ \qquad\ \ \ \qquad
\ \qquad\ \ \bigskip

We claim that $\pi$ is a winning strategy for player $\mathsf{I}$ in
$\mathcal{G}_{P\text{-}point}\left(  \mathcal{W}\right)  .$ Consider a run of
the game in which player $\mathsf{I}$ played according to $\pi.$ Let $Z=%
{\textstyle\bigcup\limits_{n\in\omega}}
z_{n},$ we will prove that $Z\notin\mathcal{U}.$ Let $q=%
{\textstyle\bigcup\limits_{i\in\omega}}
T\left(  p_{i},n_{i}\right)  $ be te tree that was constructed by player
$\mathsf{I}$ during the play. It is easy to see that $\mathcal{F\cup}\left\{
\mathcal{C}\left(  spsuc_{q}\left(  s\right)  \right)  \mid s\in split\left(
q\right)  \right\}  $ generates an $F_{\sigma}$-filter, call it $\mathcal{K}.$
Note that $\mathcal{K\leq F}$ hence $\mathcal{K}$ forces that $\sigma$ is a
winning strategy for player $\mathsf{I}$ in $\mathcal{H\mathbb{(}%
W},\mathcal{\dot{U}}_{gen})$. Moreover, $\mathcal{K}$ forces that
$q\in\mathbb{PT(}\mathcal{\dot{U}}_{gen}).$ Since player $\mathsf{I}$ is
forced to win in $\mathcal{H\mathbb{(}W},\mathcal{\dot{U}}_{gen})$, it must be
the case that $Z\notin\mathcal{W}.$ This shows that $\pi$ is a winning
strategy for player $\mathsf{I}$ in $\mathcal{G}_{P\text{-}point}\left(
\mathcal{W}\right)  .$ Since player $\mathsf{I}$ can not have a winning
strategy in the $P$-point game, we get a contradiction.
\end{proof}

\qquad\ \ \ \ 

In this way, we conclude that $\mathbb{F}_{\sigma}\ast\mathbb{PT(}%
\mathcal{\dot{U}}_{gen}\mathcal{)}$ preserves all ground model $P$-points.
Note that after forcing with $\mathbb{F}_{\sigma}\ast\mathbb{PT(}%
\mathcal{\dot{U}}_{gen}\mathcal{)},$ there are intermediate extensions with
$P$-points that are not preserved ($\mathcal{U}_{gen}$ for example), however,
all ground model $P$-points are preserved. By iterating $\mathbb{F}_{\sigma
}\ast\mathbb{PT(}\mathcal{\dot{U}}_{gen}\mathcal{)}$ with countable support,
we get the following result \cite{BlassShelah}:

\begin{corollary}
[Blass-Shelah]The inequality $\mathfrak{u<s}$ is consistent with \textsf{ZFC.}
\end{corollary}

\ \ \qquad\ \ \ \ \ \ \qquad\ \ \ \ \ 

In \cite{DiagonalizingHeike} Mildenberger proved the following interesting result:

\begin{proposition}
[Mildenberger]It is consistent that there is a proper forcing that
diagonalizes an ultrafilter and preserves a $P$-point.
\end{proposition}

\qquad\ \ \ \ 

Note that our work provides an alternative proof of the theorem of Mildenberger.

\qquad\ \ \ 

We would like to mention that in original model of Shelah of $\mathfrak{b<s},$
in the Blass-Shelah model, and in our model, the almost disjointness number is
equal to $\omega_{1}.$ In \cite{ProperandImproper} (using also the results
from \cite{BrendleRaghavan}) it is proved that $\mathfrak{a}=\omega_{1}$ after
iterating (with countable support) the forcing $\mathbb{F}_{\sigma}%
\ast\mathbb{M(}\mathcal{\dot{U}}_{gen}).$ A similar approach works when
iterating $\mathbb{F}_{\sigma}\ast\mathbb{PT(}\mathcal{\dot{U}}_{gen}).$ It is
also possible to use the technique of theorem 6.6 in
\cite{ParametrizedDiamonds} to show that $\Diamond\left(  \mathfrak{b}\right)
$ holds in that model, hence $\mathfrak{a}=\omega_{1}$ (see
\cite{ParametrizedDiamonds} for the definition of $\Diamond\left(
\mathfrak{b}\right)  $ and the proof that $\Diamond\left(  \mathfrak{b}%
\right)  $ implies $\mathfrak{a}=\omega_{1}$). Since this result will not be
used in the paper, we omit the details.

\qquad\ \ \ 

Regarding the \emph{groupwise-density number }$\mathfrak{g}$, it can be proved
that $\mathfrak{g}=\omega_{2}$ holds in our model. In particular, this is a
model of $\mathfrak{u<g,}$ so the \emph{Near Coherence of Filters }principle
holds in our model. This was proved by David Chodounsk\'{y}, Jonathan Verner
and the first author and will be published in a consequent paper. The reader
may consult \mbox{
\cite{HandbookBlass} }\hspace{0pt}
to lear nore about $\mathfrak{g}$ and the
Near Coherence of filters principle.

\qquad\ \ \

\section{A model of $\omega_{1}=\mathfrak{u<a}$}

In this section, we will prove that every \textsf{MAD} family can be destroyed
with a proper forcing that preserves $P$-points, answering the questions of
Brendle and Shelah. First, we will prove that if $\mathcal{A}$ is a
\textsf{MAD} family, then $\mathcal{\dot{U}}_{gen}\left(  \mathcal{A}\right)
$ is forced to be $\mathcal{B}$-Canjar for every $\mathfrak{b}$-family
$\mathcal{B}$ in the ground model. In \cite{BrendleRaghavan} and
\cite{CanjarFiltersII} it is proved that after adding $\omega_{1}$-Cohen
reals, $\mathbb{F}_{\sigma}\left(  \mathcal{A}\right)  $ forces that
$\mathcal{U}_{gen}\left(  \mathcal{A}\right)  $ has these properties.
Obviously, we can not use these results since we do not want to add Cohen
reals. In this section, we will prove that the Cohen reals were really not
needed in the first place. This proof takes inspiration in the proof of
$\mathfrak{b<a}$ by Brendle in \cite{MobandMAD}. When the authors were
preparing the paper, they learned from Zdomskyy that he has found a different
proof that the preliminary Cohen reals are not needed.

\begin{definition}
We say a \textsf{MAD} family $\mathcal{A}$ is a \emph{Laflamme family }if
$\mathcal{I}\left(  \mathcal{A}\right)  $ can not be extended to an
$F_{\sigma}$ ideal (or equivalently, $\mathcal{I}\left(  \mathcal{A}\right)
^{\ast}$ can not be extended to an $F_{\sigma}$-filter).
\end{definition}

\qquad\ \ \ 

Laflamme proved that the Continuum Hypothesis implies that there is a Laflamme
\textsf{MAD }family. Minami and Sakai constructed a Laflamme \textsf{MAD
}family assuming $\mathfrak{p=c.}$ It is a major open problem if \textsf{ZFC
}implies the existence of Laflamme \textsf{MAD }families.  The reader may
consult \cite{ZappingSmallFilters} and
\cite{KatetovandKatetovBlassOrdersFsigmaIdeals} for more information on
Laflamme \textsf{MAD} families. 

\qquad\ \qquad\ \qquad\ \qquad\ \ \ 

Note that if $\mathcal{A}$ is not Laflamme
(i.e. $\mathcal{A}$ can be extended to an $F_{\sigma}$-ideal), then
$\mathbb{F}_{\sigma}\ast\mathbb{PT(}\mathcal{\dot{U}}_{gen})$ destroys
$\mathcal{A}$ below some condition, in this way, we only need to take care of
Laflamme families. The following is a simple lemma that will be needed later:

\begin{lemma}
Let $\mathcal{A}$ be a \textsf{MAD} family and $\mathcal{F}\in\mathbb{F}%
_{\sigma}\left(  \mathcal{A}\right)  .$ If there is a proper forcing
$\mathbb{P}$ such that $\mathbb{P}$ forces the following statement:
\textquotedblleft There is $\mathcal{D\in}\left[  \mathcal{A}\right]
^{\omega}$ such that $\mathcal{I}\left(  \mathcal{A}\right)  ^{\ast}%
\subseteq\left\langle \mathcal{F\cup}\left\{  \omega\setminus A\mid
A\in\mathcal{D}\right\}  \right\rangle $\textquotedblright\ then $\mathcal{A}$
is not Laflamme.
\end{lemma}

\begin{proof}
Since $\mathbb{P}$ is a proper forcing, we can find a condition $p\in
\mathbb{P}$ and $\mathcal{D}_{1}\in\left[  \mathcal{A}\right]  ^{\omega}$ in
$V$ such that $p\Vdash``\mathcal{\dot{D}\subseteq D}_{1}\textquotedblright.$
It is then easy to see that $\mathcal{I}\left(  \mathcal{A}\right)  ^{\ast
}\subseteq\left\langle \mathcal{F\cup}\left\{  \omega\setminus A\mid
A\in\mathcal{D}_{1}\right\}  \right\rangle $ so $\mathcal{A}$ is not Laflamme.
\end{proof}

\qquad\ \qquad\ \ 

Given $X\subseteq\left[  \omega\right]  ^{<\omega}$ and $A\in\left[
\omega\right]  ^{\omega},$ we define $Catch\left(  X,A\right)  =\left\{  s\in
X\mid s\subseteq A\right\}  .$ We will need the following definition:

\begin{definition}
Let $\mathcal{F}$ be an $F_{\sigma}$-filter, $X\subseteq\left[  \omega\right]
^{<\omega}$ and $A\in\left[  \omega\right]  ^{\omega}.$ We will say that
$\bigstar\left(  \mathcal{F},X,A\right)  $ holds, if the following conditions
are satisfied:

\begin{enumerate}
\item $A\in\mathcal{F}^{+}.$

\item If $B\in\left[  A\right]  ^{\omega}\cap\mathcal{F}^{+}$ then
$Catch\left(  X,B\right)  \in\left(  \mathcal{F}^{<\omega}\right)  ^{+}$ (i.e.
for every $F\in\mathcal{F}$ there is $s\in X$ such that $s\subseteq F\cap B$).
\end{enumerate}
\end{definition}

\qquad\ \ \qquad\ \ \qquad\ \qquad\ 

Let $\mathcal{A}$ be a \textsf{MAD} family, $\mathcal{F}\in\mathbb{F}_{\sigma
}\left(  \mathcal{A}\right)  $ and $X\subseteq\left[  \omega\right]
^{<\omega}$ such that $\mathcal{C}\left(  X\right)  \subseteq\left\langle
\mathcal{F\cup I}\left(  \mathcal{A}\right)  ^{\ast}\right\rangle .$ Fix
$\left\langle \mathcal{C}_{n}\right\rangle _{n\in\omega}$ an increasing family
of compact sets such that $\mathcal{F}=%
{\textstyle\bigcup}
\mathcal{C}_{n}.$ The\emph{\ Brendle game}\footnote{This game was based on the
rank arguments used by Brendle in \cite{MobandMAD}. A similar (but different)
approach using games was used by Brendle and Taylor in \cite{BrendleTaylor}.}
$\mathcal{BR}\left(  \mathcal{A},\mathcal{F},X\right)  $ is defined as follows,

\qquad\ \ \ \ \ \ \ \ \ \ \qquad\ \ \ \qquad\qquad\qquad\ \ \ \ \ \ \qquad
\ \qquad\ \qquad\ \ \ \ \ \ \ \ \ 

\begin{center}%
\begin{tabular}
[c]{|l|l|l|l|l|l|l|l}\hline
$\mathsf{I}$ & $Y_{0}$ &  & $Y_{1}$ &  & $Y_{2}$ &  & $\cdots$\\\hline
$\mathsf{II}$ &  & $s_{0}$ &  & $s_{1}$ &  & $s_{2}$ & $\cdots$\\\hline
\end{tabular}

\qquad\ \ \ \ \ \ \ \ \ \qquad
\end{center}

\qquad\ \ \ \ \ \ \ \ \ \ \qquad\ \ \ 

Where $Y_{m}\in\mathcal{I}\left(  \mathcal{A}\right)  ^{\ast},$ $s_{m}%
\in\left[  Y_{m}\right]  ^{<\omega}$ intersects all the elements of
$\mathcal{C}_{m}$ and $\max\left(  s_{m}\right)  $ $<\min\left(
s_{m+1}\right)  .$\footnote{Note that the game $\mathcal{BR}\left(
\mathcal{A},\mathcal{F},X\right)  $ does not only depend on $\mathcal{F},$ but
on its representation as an increasing union of compact sets. A more formal
notation would be $\mathcal{BR(A},\left\langle \mathcal{C}_{n}\right\rangle
_{n\in\omega},X).$} Player $\mathsf{I}$ wins the game if $\bigcup
\limits_{n\in\omega}s_{n}$ contains an element of $X.$ Note that this is an
open game for $\mathsf{I,}$ i.e., if she wins, then she wins already in a
finite number of steps. By the Gale-Stewart theorem (see \cite{Kechris}), the
Brendle game is determined. We will now prove the following:

\begin{proposition}
Let $\mathcal{A}$ be a Laflamme \textsf{MAD} family and $\mathcal{F}%
\in\mathbb{F}_{\sigma}\left(  \mathcal{A}\right)  .$ For every family
$\left\{  X_{n}\mid n\in\omega\right\}  $ such that $\mathcal{C}\left(
X_{n}\right)  \subseteq\left\langle \mathcal{F\cup I}\left(  \mathcal{A}%
\right)  ^{\ast}\right\rangle ,$ there is a countable family $\mathcal{D}%
\in\left[  \mathcal{A}\right]  ^{\omega}$ such that $\bigstar\left(
\mathcal{F},A,X_{n}\right)  $ holds for every $n\in\omega$ and $A\in
\mathcal{D}.$ \label{ObtenerEstrellita}
\end{proposition}

\begin{proof}
By $V\left[  C_{\alpha}\right]  $ we denote an extension of $V$ by adding
$\alpha$-Cohen reals (the reader should not be worried by the use of Cohen
reals in the proof, see the paragraph after this result for more information).
We first claim the following:

\begin{claim}
If $X\subseteq\left[  \omega\right]  ^{<\omega}$ is such that $\mathcal{C}%
\left(  X\right)  \subseteq\left\langle \mathcal{F\cup I}\left(
\mathcal{A}\right)  ^{\ast}\right\rangle ,$ then in $V\left[  C_{\omega_{1}%
}\right]  $ the player $\mathsf{I}$ has a winning strategy for $\mathcal{BR}%
\left(  \mathcal{A},\mathcal{F},X\right)  .$
\end{claim}

\qquad\ \ 

We will prove the claim by contradiction, since $\mathcal{BR}\left(
\mathcal{A},\mathcal{F},X\right)  $ is determined, we assume that
$\mathsf{II}$ has a winning strategy, call it $\pi.$ We will choose a tree
$T\subseteq\left(  \left[  \omega\right]  ^{<\omega}\right)  ^{<\omega}$ and a
family $\left\{  B_{t}\mid t\in T\right\}  \subseteq\mathcal{I}\left(
\mathcal{A}\right)  ^{\ast}$ with the following properties:

\begin{enumerate}
\item $\emptyset\in T$ and $B_{\emptyset}=\omega$ (this is just a technical step).

\item If $t=\left\langle s_{0},s_{1},...,s_{n}\right\rangle \in T$ then
$\left\langle B_{\left\langle s_{0}\right\rangle },s_{0},B_{\left\langle
s_{0},s_{1}\right\rangle },s_{1},...,B_{\left\langle s_{0},s_{1}%
,...,s_{n}\right\rangle },s_{n}\right\rangle $ is a legal partial play of
$\mathcal{BR}\left(  \mathcal{A},\mathcal{F},X\right)  $ in which Player
$\mathsf{II}$ is using her strategy $\pi.$
\end{enumerate}

\qquad\ \ 

An important remark is in order here: Note that for example, for every
$s\in\left[  \omega\right]  ^{<\omega}$ there may be infinitely many
$B\in\mathcal{I}\left(  \mathcal{A}\right)  ^{\ast}$ such that $\left\langle
B,s\right\rangle $ is a legal partial play of $\mathcal{BR}\left(
\mathcal{A},\mathcal{F},X\right)  $ in which Player $\mathsf{II}$ is using her
strategy $\pi.$ For $B_{\left\langle s\right\rangle }$ we just choose and fix
one of them. The tree $T$ and $\left\{  B_{t}\mid t\in T\right\}  $ are
recursively constructed as follows:

\begin{enumerate}
\item $\emptyset\in T$ and $B_{\emptyset}=\omega.$

\item $T_{1}$ is the set of all $\left\langle s\right\rangle $ such that
$s\in\left[  \omega\right]  ^{<\omega}$ and there is $B\in\mathcal{I}\left(
\mathcal{A}\right)  ^{\ast}$ for which $\left\langle B,s\right\rangle $ is a
legal partial play of $\mathcal{BR}\left(  \mathcal{A},\mathcal{F},X\right)  $
in which Player $\mathsf{II}$ is using her strategy $\pi.$

\item For every $s$ such that $\left\langle s\right\rangle \in T_{1},$ we
choose $B_{s}\in\mathcal{I}\left(  \mathcal{A}\right)  ^{\ast}$ for which
$\left\langle B_{s},s\right\rangle $ is a legal partial play.

\item Given a node $t=\left\langle s_{0},s_{1},...,s_{n}\right\rangle \in T$
(and we know that the sequence $\left\langle B_{\left\langle s_{0}%
\right\rangle },s_{0},B_{\left\langle s_{0},s_{1}\right\rangle }%
,s_{1},...,B_{\left\langle s_{0},s_{1},...,s_{n}\right\rangle },s_{n}%
\right\rangle $ is a legal partial play) let $suc_{T}\left(  t\right)  $ be
the set of all $z\in\left[  \omega\right]  ^{<\omega}$ for which there is
$B\in\mathcal{I}\left(  \mathcal{A}\right)  ^{\ast}$ such that $\left\langle
B_{\left\langle s_{0}\right\rangle },s_{0},B_{\left\langle s_{0}%
,s_{1}\right\rangle },s_{1},...,B_{\left\langle s_{0},s_{1},...,s_{n}%
\right\rangle },s_{n},B,z\right\rangle $ is a legal partial play (in which
Player $\mathsf{II}$ is using her strategy $\pi$). We fix $B_{t^{\frown
}\left\langle z\right\rangle }\in\mathcal{I}\left(  \mathcal{A}\right)
^{\ast}$ with this property.
\end{enumerate}

\qquad\ \ \ \qquad\ \qquad\ \qquad\ \qquad\ \qquad\ 

Note that if $t=\left\langle s_{0},s_{1},...,s_{n}\right\rangle \in T,$ then $%
{\textstyle\bigcup\limits_{i\leq n}}
s_{i}$ does not contain an element of $X,$ this is because $\pi$ is a winning
strategy for player $\mathsf{II.}$ Clearly $T$ is a countable tree with no
isolated branches, so it is equivalent to Cohen forcing when viewed as a
forcing notion. Since $T$ is countable, it appears in an intermediate
extension of $V\left[  C_{\omega_{1}}\right]  .$ Let $\beta<\omega_{1}$ such
that $T\in V\left[  C_{\beta}\right]  .$


Given $Y\in\mathcal{I}\left(  \mathcal{A}\right)  ^{\ast}$ define the set
$D_{Y}$ of all $t=\left\langle s_{0},s_{1},...,s_{n}\right\rangle \in T$ such
that there is $i\leq n$ for which $s_{i}\subseteq Y.$ It is easy to see that
each $D_{Y}$ is an open dense subset of $T$. Let $G\in V\left[  C_{\omega_{1}%
}\right]  $ be a $\left(  T,V\left[  C_{\beta}\right]  \right)  $-generic
branch through $T$. It is easy to see that $G$ induces a legal play of the
game in which $\mathsf{II}$ followed her strategy. Let $D=\bigcup G$, and
since $\pi$ is a winning strategy for $\mathsf{II,}$ we conclude that $D$ does
not contain an element of $X.$ By genericity $D\in\left\langle \mathcal{I}%
\left(  \mathcal{A}\right)  ^{\ast}\cup\mathcal{F}\right\rangle ^{+}$ however,
$\omega\setminus D\in\mathcal{C}\left(  X\right)  \subseteq\left\langle
\mathcal{I}\left(  \mathcal{A}\right)  ^{\ast}\cup\mathcal{F}\right\rangle $
which is obviously a contradiction. This finishes the proof of the claim.


We work in $V\left[  C_{\omega_{1}}\right]  ,$ where player $\mathsf{I}$ has
winning strategies for all of the games $\mathcal{BR}\left(  \mathcal{A}%
,\mathcal{F},X_{n}\right)  $ with $n\in\omega.$ Let $\pi_{n}$ be the winning
strategy for the game $\mathcal{BR}\left(  \mathcal{A},\mathcal{F}%
,X_{n}\right)  .$ Let $\mathcal{W}$ be set of elements of $\mathcal{I}\left(
\mathcal{A}\right)  ^{\ast}$ that may be played by $\mathsf{I}$ following her
winning strategy$\,\ $in any of these games$.$ It is not hard to see that
$\mathcal{W}$ is countable. Note that if $W\in\mathcal{W}$ then $W$ almost
contains every element of $\mathcal{A}$ except for finitely many (this is
because $W\in\mathcal{I}\left(  \mathcal{A}\right)  ^{\ast}$). Let
$\mathcal{A}^{\prime}\subseteq\mathcal{A}$ be the set of all $A\in\mathcal{A}$
for which there is $W\in\mathcal{W}$ such that $A\nsubseteq^{\ast}W.$ Note
that $\mathcal{A}^{\prime}$ is countable. Since $\mathcal{A}$ is Laflamme in
$V,$ by a previous lemma, $\mathcal{I}\left(  \mathcal{A}\right)  ^{\ast}$ it
is not contained in $\left\langle \mathcal{F\cup}\left\{  \omega\setminus
B\mid B\in\mathcal{A}^{\prime}\right\}  \right\rangle ,$ so there is $A_{0}%
\in\mathcal{A}$ such that \ $\omega\setminus A_{0}\notin\left\langle
\mathcal{F\cup}\left\{  \omega\setminus B\mid B\in\mathcal{A}^{\prime
}\right\}  \right\rangle .$ This implies that $A_{0}\in\mathcal{F}^{+}$ and
$A_{0}$ is almost contained in every member of $\mathcal{W}.$ We claim that
$\bigstar\left(  \mathcal{F},A_{0},X_{n}\right)  $ holds for each $n\in
\omega.$ Let $B\in\wp\left(  A_{0}\right)  \cap\mathcal{F}^{+}$ we will now
show that $Catch\left(  X_{n},B\right)  $ is positive for each $n\in\omega.$
Let $F\in\mathcal{F}$ and consider the following play in $\mathcal{BR}\left(
\mathcal{A},\mathcal{F},X_{n}\right)  $,

\qquad\ \ \ \ \ \ \ \ \ \ \ \qquad\ \qquad\ \qquad\ \qquad\ \qquad
\ \qquad\ \ \ \ \ \bigskip%

\begin{tabular}
[c]{|l|l|l|l|l|l|l|l}\hline
$\mathsf{I}$ & $W_{0}$ &  & $W_{1}$ &  & $W_{2}$ &  & $\cdots$\\\hline
$\mathsf{II}$ &  & $s_{0}$ &  & $s_{1}$ &  & $s_{2}$ & $\cdots$\\\hline
\end{tabular}

\bigskip\qquad\ \qquad\ \qquad\ \qquad\ \qquad\ \qquad\ \qquad\ 

\noindent where the $W_{i}$ are played by $\mathsf{I}$ according to $\pi_{n},$ $s_{i}%
\in\left[  B\cap F\right]  ^{<\omega}$ and intersects every element of
$\mathcal{C}_{i}.$ This is possible since $B\cap F$ is positive and is almost
contained in every $W_{n}.$ Since $\pi_{n}$ is a winning strategy, this means
that $\mathsf{I}$ wins the game,which entails that $\bigcup s_{n}\subseteq
B\cap F$ contains an element of $X_{n}.$ We can then obtain each $A_{n+1}$ by
repeating the same argument and using that $\mathcal{I}\left(  \mathcal{A}%
\right)  ^{\ast}$ it is not contained in $\left\langle \mathcal{F\cup}\left\{
\omega\setminus B\mid B\in\mathcal{A}^{\prime}\right\}  \cup\left\{
\omega\setminus A_{0},..,\omega\setminus A_{n}\right\}  \right\rangle .$ Let
$\mathcal{D}_{1}=\left\{  A_{n}\mid n\in\omega\right\}  .$


We know that $V\left[  C_{\omega_{1}}\right]  \models\bigstar\left(
\mathcal{F},A_{n},X_{m}\right)  $ for every $n,m\in\omega.$ However, it is
easy to see that the statement $\bigstar\left(  \mathcal{F},A_{n}%
,X_{m}\right)  $ is absolute between models of \textsf{ZFC} (in fact, we only
need that it is downwards absolute, which is easy). So $V\models
\bigstar\left(  \mathcal{F},A_{n},X_{m}\right)  $ for every $n,m\in\omega.$
Since $\mathbb{C}_{\omega_{1}}$ has the countable chain condition, there is
$\mathcal{D}\in\left[  \mathcal{A}\right]  ^{\omega}$ such that $\mathbb{C}%
_{\omega_{1}}\Vdash``\mathcal{D}_{1}\subseteq\mathcal{D}\textquotedblright.$
By the previous remark, we may assume that that $\bigstar\left(
\mathcal{F},A,X_{n}\right)  $ holds for every $n\in\omega$ and $A\in
\mathcal{D}.$
\end{proof}

\qquad\ \ \ 

The reader might feel that we cheated in the previous proof by adding the
Cohen reals, and sincerely, we have, but it was a \textquotedblleft legal
cheating\textquotedblright. We only used the Cohen reals to find ground model
objects, and after finding them, we came back to the ground model as if
nothing happened.

\qquad\ \ 

Given $A\in\left[  \omega\right]  ^{\omega}$ and $l\in\omega$ define
$Part_{l}\left(  A\right)  $ as the set of all sequences $\left\langle
B_{1},...,B_{l}\right\rangle $ such that $A=\bigcup\limits_{i\leq l}B_{i}$ and
$B_{i}\cap B_{j}=\emptyset$ whenever $i\neq j.$ Note that $Part_{l}\left(
A\right)  $ is a compact space with the natural topology. Also it is clear
that if $A\in\mathcal{F}^{+}$ and $\left\langle B_{1},...,B_{l}\right\rangle
\in Part_{l}\left(  A\right)  $ then there is $j\leq l$ such that $B_{j}%
\in\mathcal{F}^{+}.$

\begin{lemma}
Let $\mathcal{F}$ be a filter, $\mathcal{C}\subseteq\mathcal{F}$ a compact set
and $X\in\left(  \mathcal{F}^{<\omega}\right)  ^{+}.$ Let $A$ such that
$\bigstar\left(  A,\mathcal{F},X\right)  $ holds and let $l\in\omega.$ There
is $n\in\omega$ with the property that for all $\left\langle B_{1}%
,...,B_{l}\right\rangle \in Part_{l}\left(  A\right)  $ there is $i\leq l$
such that if $F\in\mathcal{C}$ then $X\cap\wp\left(  B_{i}\cap n\right)  $
contains a subset of $F.$ \label{Partirestrellita}
\end{lemma}

\begin{proof}
Let $U_{n}$ be the set of all $\left\langle B_{1},...,B_{l}\right\rangle \in
Part_{l}\left(  A\right)  $ such that there is $i\leq l$ with the property
that if $F\in\mathcal{C}$ then $X\cap\wp\left(  B_{i}\cap n\right)  $ contains
a subset of $F.$ Note that $\left\{  U_{n}\mid n\in\omega\right\}  $ is an
open cover (recall that $\bigstar\left(  A,\mathcal{F},X\right)  $ holds and
if we split $A$ into finitely many pieces, then one of the pieces must be in
$\mathcal{F}^{+}$) and the result follows since $Part_{l}\left(  A\right)  $
is compact.
\end{proof}

\qquad\ \ \ \ \ \qquad\ \ \ 

We will now prove the following:

\begin{proposition}
Let $\mathcal{F}$ be a filter, $\mathcal{C}\subseteq\mathcal{F}$ a compact
set, $X\in\left(  \mathcal{F}^{<\omega}\right)  ^{+},$ $A\in\left[
\omega\right]  ^{\omega}$ such that $\bigstar\left(  A,\mathcal{F},X\right)  $
holds and $l\in\omega.$ There is $Y\in\left[  X\right]  ^{<\omega}$ such that
if $C_{1},...,C_{l}\in\mathcal{C}\left(  Y\right)  $ and $F\in\mathcal{C}$
then there is $s\in Y\cap\left[  A\right]  ^{<\omega}$ such that $s\subseteq
C_{1}\cap...\cap C_{l}\cap F.$ \label{PartircompactosMAD}
\end{proposition}

\begin{proof}
Let $n$ such that for every $\left\langle B_{1},...,B_{2^{l}}\right\rangle \in
Part_{2^{l}}\left(  A\right)  $ and for every $F\in\mathcal{C}$ there is
$j\leq2^{l}$ for which $X\cap\wp\left(  B_{j}\cap n\right)  $ contains a
subset of $F.$ Let $Y=X\cap\wp\left(  l\right)  ,$ we will see that $Y$ has
the desired properties. Let $C_{1},...,C_{l}\in\mathcal{C}\left(  Y\right)  $
and $F\in\mathcal{C}.$ For every $s:l\longrightarrow2$ define $B_{s}$ as the
set of all $a\in A$ such that $a\in C_{i}$ if and only if $s\left(  i\right)
=1.$ Clearly $\left\langle B_{s}\right\rangle _{s\in2^{l}}\in Part_{2^{l}%
}\left(  A\right)  $ and we may conclude that there is $s$ such that
$Y\cap\wp\left(  B_{s}\cap n\right)  $ contains an element of $F.$ Since
$C_{1},...,C_{l}\in\mathcal{C}\left(  Y\right)  $ we conclude that $s$ must be
the constant function $1$, and this entails the desired conclusion.
\end{proof}

\qquad\ \ \ 

We can then finally conclude the following:

\begin{proposition}
If $\mathcal{A}$ is a Laflamme \textsf{MAD} family, then $\mathbb{F}_{\sigma
}\left(  \mathcal{A}\right)  $ forces that $\mathcal{\dot{U}}_{gen}\left(
\mathcal{A}\right)  $ is $\mathcal{B}$-Canjar for every $\mathfrak{b}$-family
$\mathcal{B}$ in the ground model. \label{MSDgenesCanjar}
\end{proposition}

\begin{proof}
It is enough to show that if $\mathcal{F}\Vdash``\overline{X}=\left\langle
X_{n}\right\rangle _{n\in\omega}\subseteq(\mathcal{\dot{U}}_{gen}\left(
\mathcal{A}\right)  ^{<\omega})^{+}\textquotedblright$ then there is
$\mathcal{G\leq F}$ and $f\in\mathcal{B}$ such that $\mathcal{C}\left(
\overline{X}_{f}\right)  \subseteq\mathcal{G}.$ Let $\mathcal{F=}%
\bigcup\mathcal{C}_{n}$ where each $\mathcal{C}_{n}$ is compact and they form
an increasing chain. By the previous results, we may find $\left\{  A_{n}\mid
n\in\omega\right\}  \subseteq\mathcal{A}$ such that $\bigstar\left(
A_{n},\mathcal{F},X_{m}\right)  $ holds for every $n,m\in\omega.$ We can then
find an increasing function $g:\omega\longrightarrow\omega$ such that the
following holds:

\begin{description}
\item[*)] For every $n\in\omega$ and for every $i\leq n,$ if $Y=X\cap
\wp\left(  g\left(  n\right)  \right)  $ then for every $C_{0},...,C_{n}%
\in\mathcal{C}\left(  Y\right)  $ and $F\in\mathcal{C}_{n}$ there is $s\in
Y\cap\left[  A_{i}\right]  ^{<\omega}$ such that $s\subseteq C_{0}\cap...\cap
C_{l}\cap F.$
\end{description}

\qquad\ \ 

Since $\mathcal{B}$ is unbounded, we can find $f\in\mathcal{B}$ that is not
dominated by $g.$ It is easy to see that $\mathcal{G}=\left\langle
\mathcal{F}\cup\mathcal{C}\left(  \overline{X}_{f}\right)  \right\rangle $ is
a condition in $\mathbb{F}_{\sigma}\left(  \mathcal{A}\right)  $ and has the
desired properties.
\end{proof}

\qquad\ \ 

We can conclude:

\begin{corollary}
Every \textsf{MAD} family can be destroyed with a forcing that is proper, adds
an unsplit real, preserves all $\mathfrak{b}$-families from the ground model
and does not add Cohen reals.
\end{corollary}

\begin{proof}
If $\mathcal{A}$ is not Laflamme, then it can be destroyed by $\mathbb{F}%
_{\sigma}\ast\mathbb{PT}(\mathcal{\dot{U}}_{gen})$ (below some condition). If
$\mathcal{A}$ is Laflamme, then we can destroy it with $\mathbb{F}_{\sigma
}\left(  \mathcal{A}\right)  \ast\mathbb{PT(}\mathcal{\dot{U}}_{gen}\left(
\mathcal{A}\right)  ).$
\end{proof}

\qquad\ \ \ \qquad\qquad

By iterating the forcings in the previous corollary, we get the following:

\begin{proposition}
[Shelah]There is a model of $\omega_{1}=\mathfrak{b<a}=\mathfrak{s=c}%
=\omega_{2}.$
\end{proposition}

\qquad\ \ \ \ 

Note that if we iterate with countable support forcings of the type
$\mathbb{F}_{\sigma}\left(  \mathcal{A}\right)  \ast\mathbb{M(}\mathcal{\dot
{U}}_{gen}\left(  \mathcal{A}\right)  )$ and $\mathbb{F}_{\sigma}%
\ast\mathbb{M}(\mathcal{\dot{U}}_{gen}),$ we will get a model of $\omega
_{1}=\mathfrak{b<a}=\mathfrak{s=}$ \textsf{cov}$\left(  \mathcal{M}\right)
=\mathfrak{c}=\omega_{2}.$ In order to preserve $P$-points, we must use the
Miller forcing instead of the Mathias forcing, as we are going to show now. We
will prove that for every \textsf{MAD} family $\mathcal{A},$ the forcing
$\mathbb{F}_{\sigma}\left(  \mathcal{A}\right)  \ast\mathbb{PT(}%
\mathcal{\dot{U}}_{gen}\left(  \mathcal{A}\right)  \mathcal{)}$ preserves all
ground model $P$-points. We will need the following generalization of lemma
\ref{Partirestrellita}:

\begin{lemma}
Let $l\in\omega,\ \mathcal{F}$ a filter, $\mathcal{D}\subseteq\mathcal{F}$ a
compact set, $X_{1},...,X_{n}\subseteq\left[  \omega\right]  ^{<\omega
}\setminus\left\{  \emptyset\right\}  $ such that $\mathcal{C}\left(
X_{1}\right)  ,...,\mathcal{C}\left(  X_{n}\right)  \subseteq\mathcal{F}$ and
$A\in\left[  \omega\right]  ^{\omega}$ such that $\bigstar\left(
A,\mathcal{F},X_{i}\right)  $ holds for every $i\leq n.$ There is $m\in\omega$
such that for every $\left\langle B_{1},...,B_{l}\right\rangle \in
Part_{l}\left(  A\right)  ,$ there is $i\leq l$ such that for every
$F\in\mathcal{D}$ and for every $k\leq n,$ the set $\left(  B_{i}\cap
F\right)  \cap m$ contains an element of $X_{k}.$ \label{Partirestrellitagen}
\end{lemma}

\begin{proof}
Let $U_{m}$ be the set of all $\left\langle B_{1},...,B_{l}\right\rangle \in
Part_{l}\left(  A\right)  $ such that there is $i\leq l$ with the property
that if $F\in\mathcal{C}$ and $k\leq n$ then, $X_{k}\cap\wp\left(  B_{i}\cap
m\right)  $ contains a subset of $F.$ Note that $\left\{  U_{n}\mid n\in
\omega\right\}  $ is an open cover and the result follows since $Part_{l}%
\left(  A\right)  $ is compact.
\end{proof}

\qquad\ \ \ 

With this result, we can prove the following generalization of lemma
\ref{PartircompactosMAD}:

\begin{lemma}
Let $\mathcal{F}$ be a filter, $\mathcal{D\subseteq F}$ a compact set,
$X_{1},...,X_{n}\subseteq\left[  \omega\right]  ^{<\omega}$ such that
$\mathcal{C}\left(  X_{1}\right)  ,...,\mathcal{C}\left(  X_{n}\right)
\subseteq\mathcal{F}$ and $A\in\left[  \omega\right]  ^{\omega}$ such that
$\bigstar\left(  A,\mathcal{F},X_{i}\right)  $ holds for every $i\leq n.$
There are $Y_{1}\in\left[  X_{1}\right]  ^{<\omega},...,Y_{n}\in\left[
X_{n}\right]  ^{<\omega}$ such that for every $F\in\mathcal{D}$ and for every
$C_{i}^{1},...,C_{i}^{n}\in\mathcal{C}\left(  Y_{i}\right)  $ (with $i\leq
n$), for every $k\leq n$ there is $s\in Y_{k}\cap\left[  A\right]  ^{<\omega}$
such that $s\subseteq F\cap%
{\textstyle\bigcap\limits_{i,j\leq n}}
C_{i}^{j}.$
\end{lemma}

\begin{proof}
Let $l=n^{2}$ and by lemma \ref{Partirestrellitagen}, we know there is
$m\in\omega$ such that for every $\left\langle B_{1},...,B_{2^{l}%
}\right\rangle \in Part_{2^{l}}\left(  A\right)  $ there is $i\leq2^{l}$ such
that for every $F\in\mathcal{D}$ and for every $k\leq n,$ the set $\left(
B_{i}\cap F\right)  \cap m$ contains an element of $X_{k}.$ Let $Y_{i}%
=X_{i}\cap\wp\left(  m\right)  $ for every $i\leq n.$ We will see that the
sets $Y_{1},...,Y_{n}$ have the desired properties. Let $C_{i}^{1}%
,...,C_{i}^{n}\in\mathcal{C}\left(  Y_{i}\right)  $ (with $i\leq n$) and
$F\in\mathcal{C}.$ For every $s:n\times n\longrightarrow2$ define $B_{s}$ as
the set of all $a\in A$ such that $a\in C_{i}^{j}$ if and only if $s\left(
i,j\right)  =1.$ Clearly $\left\langle B_{s}\right\rangle _{s\in2^{l}}\in
Part_{2^{l}}\left(  A_{l}\right)  $ and we may conclude that there is
$s:n\times n\longrightarrow2$ such that $Y_{i}\cap\wp\left(  B_{s}\cap
m\right)  $ contains an element of $F$ for every $i\leq n.$ Since $C_{i}%
^{1},...,C_{i}^{n}\in\mathcal{C}\left(  Y_{i}\right)  $ we conclude that $s$
must be the constant function $1$, and this entails the desired conclusion.
\end{proof}

\qquad\ \ \ \ \ \ \qquad\ \qquad\ 

We can now prove the following result, which is a combination of proposition
\ref{UgenesWCanjar} and proposition \ref{MSDgenesCanjar}:

\begin{proposition}
If $\mathcal{W}$ is a $P$-point and $\mathcal{A}$ is a Laflamme \textsf{MAD}
family, then $\mathbb{F}_{\sigma}\left(  \mathcal{A}\right)  $ forces that
$\mathcal{\dot{U}}_{gen}\left(  \mathcal{A}\right)  $ is $\mathcal{W}$-Canjar.
\end{proposition}

\begin{proof}
We will prove the proposition by contradiction. Assume there is $\mathcal{F}%
\in\mathbb{F}_{\sigma}\left(  \mathcal{A}\right)  $ and $\sigma$ such that
$\mathcal{F}$ forces that $\sigma$ is a winning strategy for player
$\mathsf{I}$ in $\mathcal{H\mathbb{(}W},\mathcal{\dot{U}}_{gen}\left(
\mathcal{A}\right)  )$. Since strategies for player $\mathsf{I}$ are countable
objects and since $\mathbb{F}_{\sigma}\left(  \mathcal{A}\right)  $ is
$\sigma$-closed, it is enough to consider ground model strategies. Let
$\mathcal{F}=%
{\textstyle\bigcup\limits_{n\in\omega}}
\mathcal{C}_{n}$ where $\left\langle \mathcal{C}_{n}\right\rangle _{n\in
\omega}$ is an increasing sequence of compact sets. We will use $\sigma$ to
construct a winning strategy for player $\mathsf{I}$ in the game
$\mathcal{G}_{P\text{-}point}\left(  \mathcal{W}\right)  $, which will be a contradiction.


Let $L$ be the collection of all $p_{s}$ such that $p$ is a possible response
of player $\mathsf{I}$ according to $\sigma$ and $s\in split\left(  p\right)
.$ In the same way as in the proof of proposition \ref{UgenesWCanjar}, if
$p_{s}\in L,$ then $\mathcal{F}\Vdash``spsuc_{p}\left(  s\right)
\in(\mathcal{\dot{U}}_{gen}^{<\omega}\mathcal{)}^{+}\textquotedblright,$ hence
in particular $\mathcal{C}\left(  spsuc_{p}\left(  s\right)  \right)
\subseteq\left\langle \mathcal{F\cup I}\left(  \mathcal{A}\right)  ^{\ast
}\right\rangle .$ Since $L$ is countable we may assume (by extending
$\mathcal{F}$ if necessary) that $\mathcal{C}\left(  spsuc_{p}\left(
s\right)  \right)  \subseteq\mathcal{F}$ for every $p_{s}\in L.$ By
\ref{ObtenerEstrellita}, we can find a family $\left\{  A_{n}\mid n\in
\omega\right\}  \subseteq\mathcal{A}$ such that $\bigstar\left(
A_{n},\mathcal{F},\mathcal{C}\left(  spsuc_{p}\left(  s\right)  \right)
\right)  $ holds for every $n\in\omega$ and $p_{s}\in L.$


For every $\mathcal{X}=\left\{  X_{1},...,X_{n}\right\}  \in\left[  \left\{
spsuc_{p}\left(  s\right)  \mid p\in L\wedge s\in split\left(  p\right)
\right\}  \right]  ^{<\omega}$ and for every $k\in\omega,$ fix a function
$F_{\left(  \mathcal{X},k\right)  }:\mathcal{X\longrightarrow}$ $\left[
\left[  \omega^{<\omega}\right]  \right]  ^{<\omega}$ with the following properties:

\begin{enumerate}
\item $Y_{i}=F_{\left(  \mathcal{X},k\right)  }\left(  X_{i}\right)
\in\left[  X_{i}\right]  ^{<\omega}$ for every $i\leq n.$

\item For every $B\in\mathcal{C}_{k}$, for every $C_{i}^{1},...,C_{i}^{n}%
\in\mathcal{C}\left(  Y_{i}\right)  $ (with $i\leq n$) and for every
$k_{1},k_{2}\leq n$, we have that $B\cap%
{\textstyle\bigcap\limits_{i,j\leq n}}
C_{i}^{j}$ contains an element of $Y_{k_{1}}\cap\left[  A_{k_{2}}\right]
^{<\omega}.$
\end{enumerate}

\qquad\ \ 

We know such $F_{\left(  \mathcal{X},k\right)  }$ exists by the previous
lemma. The proof now proceeds in a very similar way as the proof of
\ref{UgenesWCanjar}. We define $\pi$ a strategy for player $\mathsf{I}$ in
$\mathcal{G}_{P\text{-}point}\left(  \mathcal{W}\right)  $ as follows:

\begin{enumerate}
\item Player $\mathsf{I}$ starts by playing $W_{0}=\sigma\left(
\emptyset\right)  $.

\item Assume player $\mathsf{II}$ plays $z_{0}\in\left[  W_{0}\right]
^{<\omega}$. Let $p_{0}=\sigma\left(  W_{0},z_{0}\right)  $, $s_{0}$ be the
stem of $p_{0}$ and $\mathcal{X}_{0}=\left\{  spsuc_{p_{0}}\left(
s_{0}\right)  \right\}  .$ Define $n_{0}>d^{-1}\left(  s_{0}\right)  $ to be
the least integer such that$\ d^{-1}(s_{0}{}^{\frown}t)<n_{0}$ for all $t\in
F_{\left(  \mathcal{X}_{0},0\right)  }\left(  spsuc_{p_{0}}\left(
s_{0}\right)  \right)  $. Player $\mathsf{I}$ will play (in $\mathcal{G}%
_{P\text{-}point}\left(  \mathcal{W}\right)  $) $W_{1}=\sigma\left(
W_{0},z_{0},p_{0},n_{0}\right)  $.

\item In general, assume that it has been played the sequence $\left\langle
W_{0},z_{0},...,W_{m}\right\rangle .$ At the same time, the player
$\mathsf{I}$ has secretly been constructing a sequence $\left\langle
W_{0},z_{0},p_{0},n_{0},W_{1},z_{1},p_{1},n_{1}...,W_{m}\right\rangle $ that
is being forced to be a partial play of the game $\mathcal{H(W},\mathcal{\dot
{U}}_{gen}\left(  \mathcal{A}\right)  \mathcal{)}$ following $\sigma$ such
that for every $i<m,$ the integer $n_{i}$ has the following important
property: letting $\mathcal{X}_{i}=\left\{  spsuc_{p_{i}}\left(  u\right)
\mid u\in T\left(  p_{i},n_{i-1}\right)  \right\}  $ (we define $n_{-1}%
=d^{-1}\left(  s_{0}\right)  $), for every $t\in F_{\left(  \mathcal{X}%
_{i},i\right)  }\left(  spsuc_{p_{i}}\left(  u\right)  \right)  ,$ we have
that $d^{-1}(u^{\frown}t)<n_{i}.$ Assume that player $\mathsf{II}$ plays
$z_{m}$ as her next response in $\mathcal{H(W},\mathcal{\dot{U}}_{gen}\left(
\mathcal{A}\right)  ).$ Let $p_{m}$ be the tree $\sigma\left(  W_{0}%
,z_{0},n_{0},W_{1},...,W_{m},z_{m}\right)  $ and let $n_{m}>n_{m-1}$ the least
integer with the following property: letting $\mathcal{X}_{m}=\left\{
spsuc_{p_{m}}\left(  u\right)  \mid u\in T\left(  p_{m},n_{m-1}\right)
\right\}  $, for every $t\in F_{\left(  \mathcal{X}_{m},m\right)  }\left(
spsuc_{p_{m}}\left(  u\right)  \right)  ,$ we have that $d^{-1}(u^{\frown
}t)<n_{m}.$ Player $\mathsf{I}$ will play $W_{m+1}=\sigma\left(  W_{0}%
,z_{0},n_{0},W_{1},...,W_{m},z_{m},p_{m},n_{m}\right)  .$
\end{enumerate}

\qquad\ \ \ \ \ \qquad\ \ \ \ 

The game $\mathcal{G}_{P\text{-}point}\left(  \mathcal{W}\right)  :$

\qquad\ \ \qquad\ \ \ %

\begin{tabular}
[c]{|l|l|l|l|l|l|}\hline
$\mathsf{I}$ & $W_{0}$ &  & $W_{1}$ &  & $...$\\\hline
$\mathsf{II}$ &  & $z_{0}$ &  & $z_{1}$ & \\\hline
\end{tabular}

\qquad\ \ \ \qquad\ \ \qquad\ \ \qquad\ \ \ \ \ \ \newline

The game $\mathcal{H(W},\mathcal{\dot{U}}_{gen}\left(  \mathcal{A}\right)
\mathcal{)}:$

\qquad\ \ \ \qquad\ \ \ %

\begin{tabular}
[c]{|l|l|l|l|l|l|l|l|l|l|}\hline
$\mathsf{I}$ & $W_{0}$ &  & $p_{0}$ &  & $W_{1}$ &  & $p_{1}$ &  &
$...$\\\hline
$\mathsf{II}$ &  & $z_{0}$ &  & $n_{0}$ &  & $z_{1}$ &  & $n_{1}$ & \\\hline
\end{tabular}

\qquad\ \ \ \ \ \qquad\ \ \ \ \bigskip

We claim that $\pi$ is a winning strategy for player $\mathsf{I}$ in
$\mathcal{G}_{P\text{-}point}\left(  \mathcal{W}\right)  .$ Consider a run of
the game in which player $\mathsf{I}$ played according to $\pi.$ Let $Z=%
{\textstyle\bigcup\limits_{n\in\omega}}
z_{n},$ we will prove that $Z\notin\mathcal{U}.$ Let $q=%
{\textstyle\bigcup\limits_{i\in\omega}}
T\left(  p_{i},n_{i}\right)  $ be the tree that was constructed by player
$\mathsf{I}$ during the play. It is easy to see that $\mathcal{F\cup}\left\{
\mathcal{C}\left(  spsuc_{q}\left(  s\right)  \right)  \mid s\in split\left(
q\right)  \right\}  $ generates a condition of $\mathbb{F}_{\sigma}\left(
\mathcal{A}\right)  $, call if $\mathcal{K}.$ Note that $\mathcal{K\leq F}$
hence $\mathcal{K}$ forces that $\sigma$ is a winning strategy for player
$\mathsf{I}$ in $\mathcal{H\mathbb{(}W},\mathcal{\dot{U}}_{gen})$. Moreover,
$\mathcal{K}$ forces that $q\in\mathbb{PT(}\mathcal{\dot{U}}_{gen}\left(
\mathcal{A}\right)  ).$ Since player $\mathsf{I}$ is forced to win in
$\mathcal{H\mathbb{(}W},\mathcal{\dot{U}}_{gen}\left(  \mathcal{A}\right)  )$,
it must be the case that $Z\notin\mathcal{W}.$ This shows that $\pi$ is a
winning strategy for player $\mathsf{I}$ in $\mathcal{G}_{P\text{-}%
point}\left(  \mathcal{W}\right)  .$ Since player $\mathsf{I}$ can not have a
winning strategy in the $P$-point game, we get a contradiction.
\end{proof}

\qquad\ \ \ \qquad\ \ 

We can now answer the questions of Brendle and Shelah:

\begin{theorem}
Every \textsf{MAD} family can be destroyed with a proper forcing that
preserves all $P$-points from the ground model. In particular, it is
consistent that $\omega_{1}=\mathfrak{u<a}=\omega_{2}.$
\end{theorem}

\section{\textsf{MAD} families build up from closed sets}

\qquad\ \ \qquad\ \ \ \qquad\ \ 

The \emph{closed almost disjointness number} was introduced by Brendle and
Khomskii in \cite{BrendleKhomskii}. The invariant $\mathfrak{a}_{closed}$ is
defined as the smallest number of closed sets of $\left[  \omega\right]
^{\omega}$ such that its union is a \textsf{MAD} family. Since singletons are
closed, it follows that $\mathfrak{a}_{closed}\leq\mathfrak{a}$ and it is
uncountable by a result of Mathias (see \cite{HappyFamilies} or
\cite{MADTornquist}). The following are some known results regarding this
cardinal invariant:

\begin{proposition}
\qquad\ \qquad\ \qquad\ 

\begin{enumerate}
\item (Raghavan, T\"{o}rnquist independently) $\mathfrak{p}\leq\mathfrak{a}%
_{closed}$ (see \cite{MADTornquist}).

\item (Brendle and Khomskii) It is consistent that $\mathfrak{a}%
_{closed}<\mathfrak{b}$ (see \cite{BrendleKhomskii} and \cite{BrendleRaghavan}).

\item (Brendle and Raghavan) It is consistent that $\mathfrak{b<a}_{closed}$
(see \cite{BrendleRaghavan}).

\item (Raghavan and Shelah) $\mathfrak{d}=\omega_{1}$ implies $\mathfrak{a}%
_{closed}=\omega_{1}$ (see
\cite{ComparingtheClosedAlmostDisjointnessandDominatingNumbers} and
\cite{BrendleRaghavan}).
\end{enumerate}
\end{proposition}

\qquad\ \ \ 

There are still many interesting open questions regarding $\mathfrak{a}%
_{closed}.$ The following problems are still open:

\begin{problem}
[Raghavan]Does $\mathfrak{h}\leq\mathfrak{a}_{closed}?$
\end{problem}

\begin{problem}
[Brendle, Khomskii]Does $\mathfrak{s}=\omega_{1}$ imply $\mathfrak{a}%
_{closed}=\omega_{1}?$
\end{problem}

\qquad\ \ 

The reader may consult \cite{BrendleKhomskii} or \cite{BrendleRaghavan} for
more information and open problems regarding $\mathfrak{a}_{closed}.$ If
$\mathcal{D}\subseteq\left[  \omega\right]  ^{\omega}$ is an \textsf{AD}
family, we denote its \emph{ortogonal }$\mathcal{D}^{\perp}$ as the set of all
$B\subseteq\omega$ that are almost disjoint with every element of
$\mathcal{D}.$

\qquad\ \qquad\ \qquad\ \qquad\ \qquad\ \qquad\ 

As mentioned before, Brendle and Raghavan proved that it is consistent that
$\mathfrak{b<a}_{closed}.$ In fact, they build two models in which this
inequality holds. One is using the creature forcing of Shelah and another was
using a \textsf{c.c.c. }forcing, similar to the model constructed in
\cite{MobandMAD}. In both cases, their forcings add Cohen reals. We will
combine their results with ours to build a model of $\mathfrak{b<a}_{closed}$
without adding Cohen reals, moreover; this inequality holds in the model
constructed in the previous section. The key result is the following
proposition, which was implicitly proved in the Lemma 7 of
\cite{BrendleRaghavan}:

\begin{proposition}
[Brendle, Raghavan]Let $\mathcal{D}\subseteq\left[  \omega\right]  ^{\omega}$
be a closed \textsf{AD} family. If $\mathcal{U}$ is a $P$-point such that
$\mathcal{D\cap U}=\emptyset,$ then there is $U\in\mathcal{U}\cap
\mathcal{D}^{\perp}.$
\end{proposition}

\begin{proof}
Note that if $\mathcal{D\cap U}=\emptyset$ then $\mathcal{I}\left(
\mathcal{D}\right)  \cap\mathcal{U}=\emptyset$ since $\mathcal{U}$ is an
ultrafilter. We argue by contradiction, assume that $\mathcal{U}%
\cap\mathcal{D}^{\perp}=\emptyset.$ Note that this implies that for every
$B\in\mathcal{U}$ there are $\left\{  A_{n}\mid n\in\omega\right\}
\subseteq\mathcal{D}$ such that $\left\vert B\cap A_{n}\right\vert =\omega$
for every $n\in\omega.$ Consider the Laver forcing $\mathbb{L}\left(
\mathcal{U}\right)  .$ Since $\mathcal{I}\left(  \mathcal{D}\right)  $ is an
analytic set, by a result of Blass (see
\cite{SelectiveUltrafiltersandHomogeneity}), there is $p\in\mathbb{L}\left(
\mathcal{U}\right)  $ such that either $\left[  p\right]  \subseteq
\mathcal{I}\left(  \mathcal{D}\right)  $ or $\left[  p\right]  \cap
\mathcal{I}\left(  \mathcal{D}\right)  =\emptyset.$ Since $\mathcal{U}$ is a
$P$-point, we can find $q\leq p$ and $U\in\mathcal{U}$ such that
$suc_{q}\left(  s\right)  =^{\ast}U$ for every $s\in q$ such that $s$ extends
the stem of $q.$


Let $\left\{  A_{n}\mid n\in\omega\right\}  \subseteq\mathcal{D}$ such that
$A_{n}\cap U$ is infinite for every $n\in\omega.$ On one hand, we can find a
branch $f_{0}\in\left[  q\right]  $ such that $im\left(  f_{0}\right)
\subseteq A_{0},$ so $im\left(  f_{0}\right)  \in\mathcal{I}\left(
\mathcal{D}\right)  $, but on the other hand, we can find $f_{1}\in\left[
q\right]  $ such that $im\left(  f_{1}\right)  \cap A_{n}$ is infinite for
every $n\in\omega,$ so $im\left(  f_{1}\right)  \notin\mathcal{I}\left(
\mathcal{D}\right)  .$ These two statement are in contradiction since we know
that either $\left[  q\right]  \subseteq\mathcal{I}\left(  \mathcal{D}\right)
$ or $\left[  q\right]  \cap\mathcal{I}\left(  \mathcal{D}\right)
=\emptyset.$
\end{proof}

\qquad\ \ \ \ 

We can now conclude the following:

\begin{corollary}
Let $\mathcal{A}=%
{\textstyle\bigcup\limits_{\alpha\in\kappa}}
\mathcal{C}_{\alpha}$ be a \textsf{MAD} family such that each $\mathcal{C}%
_{\alpha}$ is a closed subset of $\left[  \omega\right]  ^{\omega}.$ Let
$\mathcal{U}$ be a $P$-point such that $\mathcal{A}\cap\mathcal{U}=\emptyset.$
If $\mathbb{P}$ is a forcing that diagonalizes $\mathcal{U},$ and
$G\subseteq\mathbb{P}$ is a generic filter, then $V\left[  G\right]
\models``\mathcal{A}^{V\left[  G\right]  }=%
{\textstyle\bigcup\limits_{\alpha\in\kappa}}
\mathcal{C}_{\alpha}^{V\left[  G\right]  }$ is not a \textsf{MAD}
family$\textquotedblright$ (where $\mathcal{C}_{\alpha}^{V\left[  G\right]  }$
is the reinterpretation of $\mathcal{C}_{\alpha}$ in the model $V\left[
G\right]  $).
\end{corollary}

\begin{proof}
Let $G\subseteq\mathbb{P}$ be a generic filter. We argue in $V\left[
G\right]  .$ Let $B\in\left[  \omega\right]  ^{\omega}$ be a
pseudointersection of $\mathcal{U}.$ We claim that $B$ is almost disjoint with
$\mathcal{A}^{V\left[  G\right]  }.$ By the previous result, for every
$\alpha<\kappa$ there is $U_{\alpha}\in\mathcal{U}$ such that the following
statement holds in $V:$ \textquotedblleft$U_{\alpha}$ is almost disjoint with
every element of $\mathcal{C}_{\alpha}$\textquotedblright$.$ Since
$\mathcal{C}_{\alpha}$ is a closed set, this is an absolute statement, so
$U_{\alpha}$ is almost disjoint with every element of $\mathcal{C}_{\alpha
}^{V\left[  G\right]  }.$ Furthermore, since $B\subseteq^{\ast}U_{\alpha},$
then $B$ is also almost disjoint with every element of $\mathcal{C}_{\alpha
}^{V\left[  G\right]  }$ for every $\alpha<\kappa,$ hence $B$ is almost
disjoint with $\mathcal{A}^{V\left[  G\right]  }.$
\end{proof}

\qquad\ \ \ \qquad\ \ 

In \cite{Canjar} it was proved that Canjar ultrafilters are $P$-points
(moreover, Canjar ultrafilters are precisely the \textquotedblleft strong
$P$-points\textquotedblright, see \cite{HrusakVerner} for the definition of
strong $P$-point). In this way, we can conclude the following:

\begin{corollary}
The following statements are consistent with the axioms of \textsf{ZFC:}

\begin{enumerate}
\item (Brendle, Raghavan) $\omega_{1}=\mathfrak{b<a}_{closed}$ $\mathfrak{=}$
\textsf{cov}$\left(  \mathcal{M}\right)  =\mathfrak{c}=\omega_{2}.$

\item $\omega_{1}=\mathfrak{u}$ $\mathfrak{<a}_{closed}$ $\mathfrak{=}$
$\omega_{2}.$
\end{enumerate}
\end{corollary}

\section{Open Problems}

We will list some open problems the authors do not know how to answer.
Regarding the forcings $\mathbb{Q}\left(  \mathcal{F}\right)  $ of Sabok and
Zapletal, we know that they might or might not diagonalize $\mathcal{F}.$ It
would be interesting to know the answer of the following:

\begin{problem}
Is there a nice combinatorial characterization for the filters $\mathcal{F}$
for which $\mathbb{Q}\left(  \mathcal{F}\right)  $ diagonalizes $\mathcal{F}$?
\end{problem}

\qquad\ \ \ \ 

The forcings $\mathbb{M}\left(  \mathcal{F}\right)  $ and $\mathbb{L}\left(
\mathcal{F}\right)  $ have been proven to be very useful (see for example
\cite{Malhykyn}, \cite{AModelwithnoStronglySeparableAlmostDisjointFamilies},
\cite{SemiSelectiveCoideals}, \cite{MathiasPrikryandLaverPrikryTypeForcing},
\cite{FourandMore} or \cite{CountableFrechetGroups} for some applications of
this forcings). We expect the forcings $\mathbb{PT}\left(  \mathcal{F}\right)
$ to have interesting applications as well. It would then be useful to have a
deeper understanding of such forcings. For example we have the following:

\begin{problem}
Let $\mathcal{F}$ be a filter.

\begin{enumerate}
\item Is there a combinatorial characterization of when $\mathbb{PT}\left(
\mathcal{F}\right)  $ does not add Cohen reals?

\item Is there a combinatorial characterization of when $\mathbb{PT}\left(
\mathcal{F}\right)  $ does not add dominating reals?

\item Is there a filter $\mathcal{F}$ such that $\mathbb{PT}\left(
\mathcal{F}\right)  $ does not add dominating reals but $\mathbb{M}\left(
\mathcal{F}\right)  $ adds dominating reals?
\end{enumerate}
\end{problem}

\qquad\ \ 

It would be interesting if the previous properties have characterizations in
terms of the Kat\v{e}tov order, similar to the results for $\mathbb{Q}\left(
\mathcal{F}\right)  $ obtained in \cite{ZapletalSabok}. Regarding the
preservation of $P$-points, we have the following:

\begin{problem}
\qquad\ \ \ 

\begin{enumerate}
\item Let $\mathcal{U}$ be a $P$-point and $\mathcal{F}$ a Canjar filter such
that $\mathbb{PT}\left(  \mathcal{F}\right)  $ forces that $\mathcal{U}$
generates a non-meager filter. Does $\mathbb{PT}\left(  \mathcal{F}\right)  $
preserve $\mathcal{U}$?\footnote{This question was recently answered
positively by Chodounsk\'{y}, Verner and the first author. This result will be
published in another article.}

\item Assuming \textsf{CH}, is there a Canjar filter $\mathcal{F}$ such that
$\mathbb{PT}\left(  \mathcal{F}\right)  $ destroys all $P$-points?
\end{enumerate}
\end{problem}

\qquad\ \qquad\ \qquad\ \qquad\ \ \ \ \ 

Regarding half-Cohen reals, we do not know the answer of the following questions:

\begin{problem}
If $\mathbb{P}$ does not add half-Cohen reals and $\mathbb{P\Vdash
}``\mathbb{\dot{Q}}$ does not add half-Cohen reals$\textquotedblright,$ is it
true that $\mathbb{P\ast\dot{Q}}$ does not add half-Cohen reals?
\end{problem}

\begin{problem}
If $\delta$ is a limit ordinal, $\langle\mathbb{P}_{\alpha},\mathbb{\dot{R}%
}_{\alpha}\mid\alpha<\delta\rangle$ is a countable support iteration of proper
forcings such that each $\mathbb{P}_{\alpha}$ does not add half-Cohen reals,
is it true that $\mathbb{P}_{\delta}$ does not add half-Cohen reals? What if
each $\mathbb{P}_{\alpha}$ does not add dominating reals?
\end{problem}

\qquad\ \ \ 

Recall that not adding Cohen reals is not preserved under two step iteration
by the result of Zapletal in \cite{DimensionalForcing}.

\begin{acknowledgement}
The authors would like to thank Juris Stepr\={a}ns for valuable comments and
hours of stimulating conversations. We would also like to thank the referee for her/his generous comments that
improved the paper. The research for this paper was started
when the first author was a postdoc at York University, concluded after he
became a postdoc in the University of Toronto, and the final version was
finished while working for the Universidad Aut\'{o}noma de M\'{e}xico (UNAM). He would like to express his
gratitude to the three institutions. 
\end{acknowledgement}

\ \ \ 

\bibliographystyle{plain}
\bibliography{Bibliografia}

\qquad\ \ 

Osvaldo Guzm\'{a}n

Centro de Ciencias Matematicas, UNAM.

oguzman@matmor.unam.mx\qquad\ \ 

{}\qquad\ \ \ \qquad\ \ \qquad\qquad\qquad\ \ \ \ \ \ \ \qquad\ \ \ \ \ \ 

Damjan Kalajdzievski

York University

damjank7354@gmail.com
\end{document}